\documentclass[preprint]{imsart}
\setattribute{journal}{name}{}
\usepackage[english]{babel}

\usepackage{a4wide}
\usepackage{hyperref}
\usepackage[numbers]{natbib}

\usepackage{xcolor}

\usepackage{amsmath,amssymb,mathrsfs,bbm,cases}
\usepackage{mathtools}
\usepackage{amsthm}
\usepackage[normalem]{ulem}

\usepackage{graphicx}
\usepackage{subfigure}

\usepackage{enumerate}

\theoremstyle{plain}
\newtheorem{thm}{Theorem}[section]
\newtheorem{prop}[thm]{Proposition}
\newtheorem{lem}[thm]{Lemma}
\newtheorem{cor}[thm]{Corollary}

\theoremstyle{definition}

\theoremstyle{remark}
\newtheorem{rem}[thm]{Remark}

\newcommand{\C}{\mathbbm{C}}
\newcommand{\R}{\mathbbm{R}}
\newcommand{\Q}{\mathbbm{Q}}
\newcommand{\Z}{\mathbbm{Z}}
\newcommand{\N}{\mathbbm{N}}

\renewcommand{\P}{\mathbf{P}}

\newcommand{\E}{\mathbf{E}}

\newcommand{\V}{\mathrm{Var}}

\newcommand{\Cov}{\mathrm{Cov}}

\newcommand{\1}{\mathbf{1}}

\renewcommand{\O}{\mathcal O}

\newcommand{\e}{\mathrm{e}}
\renewcommand{\i}{\mathrm{i}}

\newcommand{\eps}{\varepsilon}

\newcommand{\vt}{\vartheta}

\DeclareMathOperator*{\argmin}{arg\,min} 

\begin{document}

\newcommand{\si}[1]{\Delta^{\mathrm{sp}}_{#1}}
\newcommand{\ti}[1]{\Delta^{\mathrm{t}}_{#1}}
\newcommand{\sti}[1]{D_{#1}}

\newcommand{\rsi}[1]{\tilde S_{#1}}
\newcommand{\rti}[1]{\tilde D_{#1}}
\newcommand{\rsti}[1]{\tilde D_{#1}}

\begin{frontmatter}
\title{Parameter estimation for SPDEs based on discrete observations in time and space}
\runtitle{Parameter estimation for SPDEs based on discrete observations}

\begin{aug}
\author{\fnms{Florian} \snm{Hildebrandt}\ead[label=e1]{florian.hildebrandt@uni-hamburg.de}} \and
\author{\fnms{Mathias} \snm{Trabs}\ead[label=e2]{mathias.trabs@uni-hamburg.de}}%\thanksref{t1}

%\thankstext{t1}{Corresponding author.}

\address{Florian Hildebrandt,\\ Fachbereich Mathematik,\\ Universit\"at Hamburg,\\ Bundesstraße 55, 20146 Hamburg, Germany,\\
\printead{e1}}
\address{Mathias Trabs,\\ Fachbereich Mathematik,\\ Universit\"at Hamburg,\\ Bundesstraße 55, 20146 Hamburg, Germany,\\
\printead{e2}}

\runauthor{Hildebrandt and Trabs}

\affiliation{Universit\"at Hamburg}

\end{aug}

\begin{abstract}
\,Parameter estimation for a parabolic linear stochastic partial differential equation in one space dimension is studied observing the solution field on a  discrete grid in a fixed bounded domain. Considering an infill asymptotic regime in both coordinates, we prove central limit theorems for realized quadratic variations based on temporal and spatial increments as well as on double increments in time and space. Resulting method of moments estimators for the diffusivity and the volatility parameter inherit the asymptotic normality and can be constructed robustly with respect to the sampling frequencies in time and space. Upper and lower bounds reveal that in general the optimal convergence rate for joint estimation of the parameters is slower than the usual parametric rate. The theoretical results are illustrated in a numerical example.
\end{abstract}

\begin{keyword}[class=AMS]
\kwd[Primary ]{60H15, 60F05}
\kwd[; secondary ]{62F12, 62M09}
\end{keyword}
\begin{keyword}
\kwd{Central limit theorems}
\kwd{infill asymptotics}
\kwd{optimal rate of convergence}%oder eher infinite factor model?
\kwd{realized quadratic variation}
\kwd{stochastic partial differential equations}
\end{keyword}

\end{frontmatter}

\section{Introduction}

%\subsection{Overview and contributions}
Stochastic partial differential equations (SPDEs) combine the ability of deterministic PDE models to describe complex mechanisms with the key feature of  diffusion models, namely a stochastic signal which evolves within the system. While SPDEs have been intensively studied in stochastic analysis, their statistical theory is only at its beginnings. Since we first need to have a thorough statistical understanding for basic SPDEs before more complex models can be studied, let us consider the prototype for the large class of parabolic SPDEs given by the stochastic heat equation on $[0,1]$:
\begin{equation}\label{eq:stochHeat}
\begin{cases}
dX_t(x) = \vt_2\frac{\partial^2}{\partial x^2} X_t(x) \,dt+\sigma\,dW_t(x),\\
X_t(0)=X_t(1)=0, \\
X_0 = \xi,
\end{cases}
\end{equation}
where $dW$ denotes white noise in space and time, $\xi$ is some independent initial condition and we impose Dirichlet boundary conditions. More general, we will later incorporate also a first and zero order term in the differential operator. The statistical aim is to infer on the diffusivity parameter $\vt_2>0$ and the diffusion or volatility parameter $\sigma^2>0$. 

 In the seminal works by \citet{huebnerEtAl1993} as well as \citet{HuebnerRozovskii1995} a spectral approach has been considered where the processes $t\mapsto u_\ell(t):=\langle X_{t},e_{\ell}\rangle_{L^{2}}$ are observable for the eigenfunctions $e_{\ell}$ of the underlying differential operator. These so called Fourier modes $u_\ell$ are independent and satisfy Ornstein-Uhlenbeck dynamics. Consequently, classical results from statistics for stochastic processes can be applied directly. While the spectral approach is studied in numerous papers, see \citet{lototsky2009} or \citet{Cialenco18} for a review, this specific observation scheme is limiting and too restrictive in potential applications. Especially, for more general equations the eigenfunctions will depend on unknown parameters, which is already the case if we add a first order term $\vt_1\frac{\partial}{\partial x} X_t(x)dt$ with unknown $\vt_1\in\R$ in \eqref{eq:stochHeat}.

Complementary to this spectral approach, the canonical problem of parameter estimation based on discrete observations of the solution field of the SPDE recently attracted an increased research activity. Assuming $X$ is observed on a discrete grid $(t_{i},y_{k})_{i=0,\dots,N,k=0,\dots,M}\subset[0,T]\times[0,1]$, approximate maximum likelihood estimators have been first investigated by \citet{Markussen03} for $T\to\infty$. For various linear SPDEs central limit theorems for method of moment type estimators based on realized quadratic variations have been studied by  \citet{torresEtAl2014}, \citet{Cialenco17}, \citet{Bibinger18,bibingerTrabs2019}, \citet{Chong18,chong2019}, \citet{shevchenkoEtAl2019}, as well as \citet{Uchida19}. However, all these works only give partial answers to the estimation problem. Even for the stochastic heat equation there neither is a  sharp analysis for joint estimation of $\vt_2$ and $\sigma^2$ nor the case where the number of spatial observations $M$ dominates the number of temporal observations $N$ has been explored in general. 

Therefore, in this relatively young research field basic and elementary questions even for simple (linear, parabolic) SPDEs still need to be answered. This becomes most important with regard to an increasing number of SPDE models in applications, e.g., in neurobiology \cite{tuckwell2013}, for the description of oceans \cite{piterbargOstrovskii1997, dostal2019}, climate modelling \cite{hottovyStechmann2015} or the description of interest rates \cite{Cont2005, santaclaraSornette2000}.

\smallskip

In order to provide a complete statistical analysis of parametric estimation for linear parabolic SPDEs in dimension one based on discrete observations on a finite time horizon $T>0$, our main contributions reveal that:
\begin{enumerate}
 \item[(i)] $\vt_2$ and $\sigma^2$ cannot be jointly estimated if $N$ or $M$ is fixed.
 \item[(ii)] The optimal convergence rate for estimating $(\vt_2,\sigma^2)$ is $1/\sqrt{M^3\wedge N^{3/2}}$ which generally is  slower than the parametric rate $1/\sqrt{MN}$.
 \item[(iii)] Realized space-time quadratic variations can be used to construct estimators which are robust with respect to the sampling frequencies $N$ and $M$ in time and space, respectively.
\end{enumerate}
In view of (i), we will consider the double asymptotic regime $M,N \to \infty$ in our analysis which results in infill asymptotics in time and space. Since the vector of observations $(X_{t_i}(y_k))_{i=0,\dots,N,k=0,\dots,M}$ is normally distributed with only two unknown parameters in equation~\eqref{eq:stochHeat}, it might surprise that there is no estimator with parametric rate for $(\vt_2,\sigma^2)$. Indeed, our lower bound verifies that the parametric rate can only be achieved if $N$ and $M^2$ are of the same order of magnitude. In view of the scaling invariance of the stochastic heat equation, this particular asymptotic regime $N\eqsim M^2$ implies that we add the same amount of information in time and space as $N$ and $M$ increase. In this sense we have a balanced design. An unbalanced regime $N=o( M^2)$ or $M=o(\sqrt{N})$ causes a deterioration of the convergence rate. 

Our statistical analysis also gives insights into the relation between the spectral and the discrete observation scheme. While both are heuristically comparable in view of the discrete Fourier transform, it turns out that there are important differences. In particular, the fully discrete observation scheme is not statistically equivalent (in the sense of Le Cam) to time discrete observations of the first $M$ Fourier modes in general.

%\subsection{Methodology}
\smallskip

Our estimators rely on realized quadratic variations, taking into account time and space increments
\begin{equation} \label{eq:increments_def}
\begin{aligned}
(\Delta_i^NX)(y_k) &:= X_{t_{i+1}}(y_k) -X_{t_{i}}(y_k),\qquad
(\delta_k^M X)(t_i)&:= X_{t_i}(y_{k+1})-X_{t_i}(y_k),
\end{aligned}
\end{equation} 
respectively, as well as space-time increments or double increments
\begin{equation} \label{eq:double_def}
D_{ik}:=(\delta_k^M \circ\Delta_i^N) X =(\Delta_i^N\circ \delta_k^M ) X =X_{t_{i+1}}(y_{k+1})-X_{t_{i+1}}(y_{k})-X_{t_i}(y_{k+1})+X_{t_i}(y_{k}).
\end{equation}
In contrast to the maximum likelihood approach which requires inversion of the large $MN\times MN$ covariance matrix,  method of moments type estimators based on \eqref{eq:increments_def} and \eqref{eq:double_def} are easy to implement.
As observed in \cite{Bibinger18}, a central limit theorem for realized temporal quadratic variations requires that the observation frequency in time dominates the observation frequency in space, more precisely, $M=o(\sqrt N)$ is necessary. Complementarily, we show that the realized spatial quadratic variation satisfies a central limit theorem if $N=o(M)$. The remaining gap can be filled by double increments and the corresponding realized space-time quadratic variation turns out to be robust with respect to the sampling frequencies $M$ and $N$. Based on these statistics, we construct method of moments estimators for $\vt_2$ and $\sigma^2$ (as well as $\vt_1$ from a first order term). Hereby, the rate optimal method for joint estimation of all identifiable parameters is an M-estimator relying on double increments. Our proofs employ directly the Gaussian distribution of $X$ which allows for an explicit covariance condition for asymptotic normality of quadratic forms of Gaussian triangular schemes. Let us remark that our estimators could be directly generalized to a nonparametric model with time dependent coefficients, as indicated in \cite{Bibinger18,Chong18}.

Note that the solution process $X$ to the SPDE~\eqref{eq:stochHeat} admits continuous trajectories only in one spatial dimension. In the multi-dimensional case one could consider noise processes which are more regular in space as studied by \citet{Chong18}. Alternatively, \citet{krizMaslowski2019} as well as \citet{altmeyerReiss2019} generalize the spectral approach to the observation of functionals $\langle X_{t},K\rangle$ for some (localizing) kernel $K$.

%Let $\Delta >0$ be a time difference between consecutive observations and consider spatial locations $0<y_1\cdots < y_M<1$. For $0\leq i< N$ and $0\leq k < M $ we introduce three types of increments,  
%\begin{equation} \label{eq:increments_def}
%\begin{aligned}
%\si{ik} &= X_{i\Delta}(y_{k+1}) - X_{i\Delta}(y_{k})\\
%\ti{ik} &= X_{(i+1)\Delta}(y_{k}) - X_{i\Delta}(y_{k})\\
%\sti{ik} &= X_{(i+1)\Delta}(y_{k+1}) - X_{(i+1)\Delta}(y_{k}) - X_{i\Delta}(y_{k+1})+ X_{i\Delta}(y_{k}),
%\end{aligned}
%\end{equation} 
%which we will refer to as space, time and double (or space-time) increments, respectively.\\

\smallskip
%\subsection{Outline}
  
This work is organized as follows: In Section~\ref{sec:prop} we give a precise definition of the model and study probabilistic properties of the solution field. In Section~\ref{sec:CLTs} we present the central limit theorems for realized quadratic variations based on space and double increments. The resulting method of moments estimators are constructed in Section~\ref{sec:ParameterEstimation}. Lower bounds are derived in Section~\ref{sec:lowerBound}. In Section~\ref{sec:sim} we illustrate our results with a numerical example. The proofs of the main results are collected in Section~\ref{sec:mainProofs} while auxiliary results are postponed to the appendix.

\section{Properties of the solution process}\label{sec:prop}
For parameters $\sigma^2>0$ and $\vt=(\vt_2,\vt_1,\vt_0)\in \R_+\times\R^2$ we consider the linear parabolic SPDE
\begin{equation} \label{eq:SPDE}
\begin{cases}
dX_t(x) = \left(\vt_2\frac{\partial^2}{\partial x ^2} X_t(x) + \vt_1\frac{\partial}{\partial x } X_t(x) +\vt_0 X_t(x)\right)\,dt+\sigma\,dW_t(x),\quad x\in[0,1],t\ge0,\\
X_t(0)=X_t(1)=0, \\
X_0 = \xi
\end{cases}
\end{equation}
driven by a cylindrical Brownian motion $W$ and where  $\xi \in L^2([0,1])$ is some {independent} initial condition.
More  precisely, we study the weak solution $X = (X_t(x), \,t\geq 0 ,\, x \in [0,1])$ to $dX_t=A_\vt X_t\,dt+\sigma dW_t$
associated with the differential operator $A_\vt=\vt_2\frac{\partial^2}{\partial x ^2}+ \vt_1\frac{\partial}{\partial x }  +\vt_0 $. As usual, the Dirichlet boundary condition in \eqref{eq:SPDE} is implemented in the domain 
$\mathcal D (A_\vt) = H^2((0,1))\cap H_0^1((0,1))$ of $A_\vt$ where $H^k((0,1))$  denotes the $L^2$-Sobolev spaces of order $k\in\N$ and with $H_0^1((0,1))$ being the closure of $C_c^\infty((0,1))$ in $H^1((0,1))$. The cylindrical Brownian motion $W$ is defined as a linear mapping $L^2((0,1)) \ni u \mapsto W_\cdot(u)$ such that 
 $t\mapsto  W_t(u)$ is a one-dimensional standard Brownian motion for all normalized $u\in L^2([0,1])$
and such that the covariance structure is 
$\Cov\left( W_t(u) , W_s(v) \right)=(s\wedge t)\, \langle u,v \rangle,$ for $u,v\in L^2([0,1]),\,s,t\geq 0$. $W$ can thus be understood as the anti-derivative in time of space-time white noise.

The differential operator $A_\vt$ has a complete orthonormal system of eigenvectors. Indeed, the eigenpairs $(-\lambda_\ell,e_\ell)_{\ell \geq 1}$ associated with $A_\vt$ are given by 
\begin{gather*} e_\ell (y) = \sqrt{2} \sin(\pi \ell y )\e^{-\kappa y/2},\quad
\lambda_\ell = \vartheta_2 (\pi^2 \ell^2+\Gamma), \qquad  y\in [0,1],\,\ell\in \N,
\end{gather*}
denoting
\begin{gather*}
\kappa := \frac{\vartheta_1}{\vartheta_2}\qquad\text{and}\qquad\Gamma := \frac{\vartheta_1^2}{4\vartheta_2 ^2}-\frac{\vartheta_0}{\vartheta_2}.
\end{gather*} 
The functions $(e_\ell)_{\ell \geq 1}$ are orthonormal with respect to the weighted $L^2$-inner product 
$$\langle u, v \rangle:=\langle u, v \rangle_\vt:= \int_0^1 u(x)v(x)\e^{\kappa x}\,dx,\qquad u,v \in L^2([0,1]).$$
Note that in absence of the first derivative in $A_\vt$, i.e~$\vt_1=0$, the system $(e_\ell)_{\ell\ge 1}$ reduces to the usual sine-base and $\langle\cdot,\cdot\rangle$ to the standard inner product on $L^2([0,1])$. In general, both the eigenpairs and the inner product depend on {the model} parameters. Hence, they are not accessible from a statistical point of view. \\

Throughout, we restrict the parameter space to 
$$\Theta = \left\{(\sigma^2,\vt_2,\vt_1,\vt_0)\in \R^4:\,\sigma^2, \vt_2,\frac{\vartheta_1^2}{4\vartheta_2 ^2}-\frac{\vartheta_0}{\vartheta_2}+\pi^2>0\right\}$$ such that all the eigenvalues are negative and $A_\vt$ is a negative self-adjoint operator. Consequently, the weak solution to the SPDE~\eqref{eq:SPDE} exists and is given by the variation of constants formula 
$
X_t  =\e^{tA_\vartheta}\xi + \sigma\int_0^t \e^{(t-s)A_\vartheta}\,dW_s,\, t\geq 0,
$ 
where $(\e^{tA_\vt})_{t\geq 0}$ denotes the strongly continuous semigroup generated by $A_\vt$, see \cite[Theorem 5.4]{DaPrato14}.

Since $(e_\ell)_{\ell \geq 1}$ is a complete orthonormal system, the cylindrical Brownian motion $W$ can be realized via $W_t=\sum_{\ell\geq 1} \beta_\ell(t)e_\ell$ in the sense of 
$W_t(\cdot)= \sum_{\ell\geq 1} \beta_\ell(t)\langle \cdot,e_k\rangle$ for a sequence of independent standard Brownian motions $(\beta_\ell)_{\ell \geq 1}$. In terms of the projections or Fourier modes $u_\ell(t):=\langle X_t,e_\ell\rangle,\,t\ge0,\ell \in \N$, we obtain the representation
\begin{equation} \label{eq:SPDE_solution}
{X_t (x)}  {=\sum _{\ell\geq 1} u_\ell(t)e_\ell(x)}, \quad t \geq 0,\,x\in [0,1], 
\end{equation}
where $(u_\ell)_{\ell \ge1}$ are one dimensional independent processes satisfying the Ornstein-Uhlenbeck dynamics 
$
du_\ell(t)= -\lambda_\ell u_\ell(t)\,dt +\sigma\, d\beta_\ell(t)
$
or equivalently
$$u_\ell(t)= u_\ell(0)\e^{-\lambda_\ell t}+\sigma\int_0^t \e^{-\lambda_\ell (t-s)}\,d\beta_\ell(s),\qquad u_\ell(0)=\langle\xi,e_\ell \rangle$$
in the sense of the usual finite dimensional stochastic integral.
For simplicity, we will assume throughout that $\{\beta_\ell,u_\ell(0),\,\ell\in \mathbb N\}$ is an independent family and $u_\ell(0)\sim \mathcal N(0,{\sigma^2}/({2 \lambda_\ell}))$ such that each coefficient process $u_\ell$ is stationary with covariance  
$
\Cov(u_\ell(s),u_\ell(t))
%=\e^{-\lambda_\ell (s+t)}\left(\V(u_\ell(0))+\sigma^2 \int_0^{s\wedge t}\e^{2\lambda_\ell r}\,dr\right)
=\frac{\sigma^2}{2\lambda_\ell}\e^{-\lambda_\ell|t-s|}$, $s,t\geq 0$.
The reduction of more general conditions on $X_0$ to the stationary case is discussed in \cite{Bibinger18}.

From representation~\eqref{eq:SPDE_solution} it is evident that $X$ is a two parameter centered Gaussian field. Therefore, the model is completely specified by its covariance structure 
\begin{equation} \label{eq:cov}
\Cov\left(X_s(x),X_t(y)\right)= \sigma^2 \sum_{\ell\geq 1}\frac{\e^{-\lambda_\ell|t-s|}}{2\lambda_\ell}e_\ell(x)e_\ell(y), \quad s,t\geq 0,\,x,y\in [0,1]. 
\end{equation}
While $\sigma^2$ is only a multiplicative factor, the covariance structure depends on $\vt$ through $\lambda_\ell$ and $e_\ell$. 
By Kolmogorov's criterion there is a continuous version of the process $( X_t(x),\,t\geq 0, x\in [0,1])$, cf.~\cite[Chapter 5.5]{DaPrato14}. In particular, point evaluations $X_t(x)$ for fixed values of $t$ and $x$ are well defined. \\

For a fixed spatial location $x$ the sample paths of the process $ X_\cdot(x)$ are no semi-martingales. In fact, $t\mapsto X_t(x)$ is only H\"older continuous of order almost 1/4  \cite[Theorem 5.22]{DaPrato14} and thus has infinite quadratic variation over any time interval. On the other hand, regarding $X$ as a function of space at a fixed point in time substantially simplifies the probabilistic structure of the process:

\begin{prop} \label{prop:Ito_representation}
Fix $t\geq 0$ and define $\Gamma_0 = \sqrt{\left|\Gamma\right|}$.
\begin{enumerate}[(i)]
\item For $x\leq y$,
$$
\Cov\left(X_t(x),X_t(y) \right)=\frac{\sigma^2}{2\vartheta_2}\e^{-\frac{\kappa}{2}(x+y)}\cdot 
\begin{cases} 
\frac{\sin(\Gamma_0 (1-y))\sin(\Gamma_0 x)}{\Gamma_0\sin(\Gamma_0)}, & \Gamma <0,\\
x(1-y), & \Gamma=0,\\
\frac{\sinh(\Gamma_0(1-y))\sinh(\Gamma_0 x)}{\Gamma_0 \sinh(\Gamma_0)}, & \Gamma >0.
\end{cases}
$$
\item The process $[0,1]\ni x\mapsto Z(x):=X_t(x)$ is an It\^o diffusion. In particular,
$$
dZ(x)=\sqrt{\frac{\sigma^2}{2\vartheta_2}}\e^{-\frac{\kappa}{2}x}\,dB(x) - 
\begin{cases}
\left(\frac{\Gamma_0 \cos(\Gamma_0 (1-x))}{\sin(\Gamma_0 (1-x))}+\frac{\kappa}{2} \right)Z(x)\,dx, & \Gamma<0, \\
\left(\frac{1}{1-x}+\frac{\kappa}{2} \right)Z(x)\,dx, & \Gamma=0, \\
\left(\frac{\Gamma_0 \cosh(\Gamma_0(1-x))}{\sinh(\Gamma_0 (1-x))}+\frac{\kappa}{2} \right)Z(x)\,dx, & \Gamma>0,
\end{cases}
$$
where $B(\cdot) = B_t(\cdot)$ is a standard Brownian motion.
\end{enumerate}
\end{prop}
Note the similarity of the covariance structures of $X_t(\cdot)$ and of the Brownian  bridge, especially  in the case $\Gamma =0$. This resemblance is in line with the Dirichlet boundary conditions $X_t(0) = X_t(1)=0$ in our model.

\begin{rem}
For $N\geq 2$ and fixed $0\leq t_1 <t_2<\ldots<t_N$ the multi-dimensional process $x\mapsto (X_{t_1}(x),\ldots,X_{t_N}(x))$ is not an It\^o diffusion. Indeed, it is not even a Markov process: Take $N=2$ and let $s<t$. It is a well known fact that for Markov processes past and future are independent, given the present state. For $x<y<z$ on the other hand, using the Gaussianity of $X$, the (Gaussian) conditional distribution of  $(X_s(x), X_t(z))$ given $(X_s(y),X_t(y))$ can be computed explicitly. From here, independence is easily disproved by checking the non-diagonal entries of the conditional covariance matrix.
\end{rem}

We conclude this section by studying absolute continuity properties for different parameter values $(\sigma^2, \vt)$ which in particular has implications for  their identifyability. To that aim we introduce the notations 
\begin{align*}
(X_t(\cdot),t \in [0,T]) &\sim P_{(\sigma^2, \vt)} \text{ on } C([0,T],L^2[0,1]),\\
(X_{t_0}(x),\,x\in[0,1]) &\sim P^{(t_0,\cdot)}_{(\sigma^2, \vt)}\text{ on } L^2[0,1],\\
(X_t(x_0),\,t \in [0,T]) &\sim P^{(\cdot,x_0)}_{(\sigma^2, \vt)}\text{ on } L^2[0,T]
\end{align*}
for fixed values $t_0\geq 0$, $x_0\in (0,1)$ and a finite time horizon $T>0$. Further, for probability measures $Q$ and $P$ we write $Q \sim P$ if they are equivalent. 

\begin{prop} \label{prop:abs_cont}
Let $t_0\geq 0$, $x_0\in (0,1)$ be fixed and consider a finite time horizon $T>0$. For any two sets of parameters  $(\sigma^2, \vt),$ $( \tilde\sigma^2, \tilde\vt) \in \Theta$ we have 
\begin{enumerate}[(i)]
\item $P_{(\sigma^2, \vt)} \sim P_{(\tilde \sigma^2, \tilde \vt)}$ if and only if ${\displaystyle(\sigma^2, \vt_2,\vt_1)=(\tilde \sigma^2, \tilde \vt_2, \tilde \vt_1)}$,
\item $P^{(t_0,\cdot)}_{(\sigma^2, \vt)} \sim P^{(t_0,\cdot)}_{(\tilde \sigma^2, \tilde \vt)}$ if and only if ${\displaystyle\left(\frac{\sigma^2}{{\vt_2}},\kappa \right)= \left(\frac{\tilde \sigma^2}{{ \tilde \vt_2}},\tilde \kappa \right)}$,
\item $ P^{(\cdot,x_0)}_{(\sigma^2, \vt)} \sim P^{(\cdot,x_0)}_{(\tilde \sigma^2, \tilde \vt)}$ if and only if ${\displaystyle\frac{\sigma^2}{\sqrt{\vt_2}}\e^{-\kappa x_0}= \frac{\tilde \sigma^2}{\sqrt{\tilde\vt_2}}\e^{-\tilde\kappa x_0}}$,
\end{enumerate}
where $\kappa = \vt_1/\vt_2,\,\tilde \kappa = \tilde\vt_1/\tilde\vt_2$.
\end{prop}
Firstly, $(i)$ shows that it is impossible to estimate $\vt_0$ consistently on a finite time horizon. Secondly, $(ii)$ and $(iii)$ reveal that an estimator that only exploits the temporal or spatial covariance structure cannot consistently estimate any other parameters than $\left({\sigma^2}/{\sqrt{\vt_2}},\kappa \right)$ or $\left({\sigma^2}/{{\vt_2}},\kappa \right)$, respectively. On the other hand, such estimators can be constructed by using squared time increments at least at two different spatial positions (cf.~\cite[Theorem 4.2]{Bibinger18}) or squared space increments (cf.~Section \ref{sec:ParameterEstimation}),  respectively.  

%%%%%%%%%%%%%%%%%%%%%%%%%%%%%%%%%%%%%%%%%%%%%%%%%%%%%%%%%%%%%%%%%%
\section{Central limit theorems for realized quadratic variations}\label{sec:CLTs}

We will now study central limit theorems for realized quadratic variations based on the  space and double increments from \eqref{eq:increments_def} and \eqref{eq:double_def}, respectively. To fix assumptions and notation, let $X$ be given by \eqref{eq:SPDE_solution} and  suppose we have $(M+1)(N+1)$ time and space discrete observations 
$$X_{t_i} (y_k ),\qquad i=0,\ldots, N,\,k=0,\ldots, M, $$
at a regular grid $(t_i,y_k)\subset [0,T] \times[0,1]$  with a fixed time horizon $T>0$ and $M,N\in \N_0$.
More precisely,  assume that 
\begin{align*}
y_k=b+k\delta\quad\text{and}\quad t_i=i\Delta \qquad\text{where}\qquad\delta =\frac{1-2b}{M},\quad \Delta=\frac TN
\end{align*}
for some fixed $b\in[0,1/2)$. The spatial locations $y_k$ are thus equidistant inside a (possibly proper) sub-interval $[b,1-b]\subset[0,1]$. {Note that whenever $M\to\infty$ or/and $N\to\infty$, we obtain infill asymptotics in space $\delta\to0$ or/and time $\Delta\to0$, respectively. }

Throughout, $M,N \to \infty$ should be understood in the sense of $\min(M,N)\to \infty$. For two sequences $(a_n),(b_n)$, we write $a_n \lesssim b_n$ to indicate that there exist some $c>0$ such that $|a_n|\leq c\cdot |b_n|$ for all $n\in \N$ and we write {$a_n \eqsim b_n$ if $a_n\lesssim b_n\lesssim a_n$}. If  $a_n=a$ for some $a\in \R$ and  all $n\in\N$, we write $(a_n)\equiv a$. Moreover, $\Vert\cdot\Vert_2$ denotes the spectral norm  and $\Vert\cdot\Vert_F$ denotes the Frobenius norm for matrices.\\

%The following central limit theorem for the sum of squares of dependent normal random variables  is our key tool for proving asymptotic normality.
The  realized quadratic variations can be regarded as sums of squares of certain Gaussian random vectors. Hence, our central limit theorems embed into the literature on quadratic forms in random variables and their asymptotic properties, see e.g.~\cite{Mathai92}. Our key tool for proving asymptotic normality is the following proposition which is tailor made for the situation present in this work and which gives an explicit covariance condition that ensures convergence to the normal distribution.

\begin{prop} \label{prop:CLT_fund}
Let $(Z_{i,n},\,1\leq i\leq d_{n},\,n\in\N)$ be a triangular array which satisfies
$(Z_{1,n}\ldots,Z_{d_{n},n})\sim\mathcal{N}\left(0,\Sigma_{n}\right)$
for a covariance matrix $\Sigma_{n}\in \R^{d_n\times d_n}$, $n\in \N$, and let $(\alpha_{i,n},\,1\leq i\leq d_{n},\,n\in\N)$ be a deterministic triangular array with values in $\{-1,1\}$. Define $
S_{n}:=\sum_{i=1}^{d_{n}}\alpha_{i,n}Z_{i,n}^{2}
$
for $n\ge1$. 
If $\Vert\Sigma_{n}\Vert_{2}^2/\V(S_n)\to 0$ as $n\to \infty$, then
we have
\[
\frac{S_{n}-\E(S_{n})}{\sqrt{\V \,S_{n}}}\overset{\mathcal D}{\longrightarrow}\mathcal{N}(0,1)\quad\text{for}\quad n\to \infty.
\]
\end{prop} 
The proof relies on the fact that $S_n$ can be represented as a linear combination of independent $\chi^2(1)$-distributed random variables. $\Vert\Sigma_{n}\Vert_{2}^2/\V(S_n)\to 0$ then implies that the corresponding Lyapunov condition is fulfilled. In this section we only require $\alpha_{i,n}= 1$ for all $i$ and $n$, i.e.~$S_n=\Vert Z_{\bullet,n} \Vert^2_2$. The general case will be necessary to verify asymptotic normality of the M-estimator in Section~\ref{sec:ParameterEstimation}. It is worth noting that Proposition~\ref{prop:CLT_fund} reveals a quite elementary proof strategy to verify several central limit theorems in \cite{Bibinger18,Cialenco17,shevchenkoEtAl2019,torresEtAl2014} instead of advanced techniques from Malliavin calculus or mixing theory.
\begin{rem}\label{rem:clt}$\,$
\begin{enumerate}
\item  If $\alpha_{i,n}=1$ for all $i,n$, it follows from Isserlis' theorem \cite{Isserlis18} that $\V(S_n)=2 \Vert \Sigma_n \Vert_F^2$ and thus, the condition for asymptotic normality may be written as $\Vert\Sigma_{n}\Vert_{2}/\Vert\Sigma_{n}\Vert_{F}\to 0$. This condition is essentially optimal: In case of independent observations it is in fact equivalent to asymptotic negligibility of the individual normalized and centered summands and hence equivalent to Lindeberg's condition. 
\item The spectral norm is bounded by the maximum absolute row sum. Writing $\Sigma_{n}=\big(\sigma_{ij}^{(n)}\big)_{i,j}$, asymptotic normality thus holds under the sufficient condition
\begin{equation} \label{eq:lyapunov}
{\displaystyle \frac{\left(\max_{i\leq d_{n}}\sum_{j=1}^{d_{n}}\Big|\sigma_{ij}^{(n)}\Big|\right)^{2}}{\V\, S_{n}}}\longrightarrow 0,\quad n\to \infty.
\end{equation}
\end{enumerate}
\end{rem}
%\textcolor{blue}
%{
%$$\rsi{ik}= \exp\left(\frac{\kappa }{2}y_k \right)\si{ik}, \quad
%\rti{ik}= \exp\left(\frac{\kappa }{2}y_k \right)\ti{ik}, \quad
%\rsti{ik}= \exp\left(\frac{\kappa }{2}y_k \right)\sti{ik}.$$
%}
%%%%%%%%%%%%%%%%%%%%%%%%%%%%%%%%%%%%%%%%%%%%%%%%%%%%%%%%%%%%%%%%%%%%%%%%%%%%%%%%

So far, the double asymptotic regime $M,N\to \infty$ has only been studied for time increments 
$(\Delta_i^NX)(y_k) = X_{t_{i+1}}(y_k) -X_{t_{i}}(y_k)$: If $b>0$ and if there exists $\rho\in(0,1/2)$ such that $M =\O(N^\rho)$, then the \emph{rescaled realized temporal quadratic variation}
\begin{equation}
  V_\mathrm{t}:=\frac{1}{MN\sqrt{\Delta}}\sum_{i=0}^{N-1}\sum_{k=0}^{M-1}\e^{\kappa y_k}(\Delta_i^N X)^2(y_k)\label{eq:Vt} 
\end{equation}
satisfies
\begin{equation} \label{eq:CLT_time}
\sqrt{MN}\left(V_{\mathrm{t}}-\frac{\sigma^2}{\sqrt{\pi \vt_2}}\right)\overset{\mathcal D}{\longrightarrow}\mathcal N \left(0,  \frac{B\sigma^4}{\pi \vt_2} \right),\quad N,M \to \infty,
\end{equation}
where
\begin{equation} \label{eq:CLT_time_B}
B= 2+\sum_{J=1}^\infty \left(2\sqrt{J}  -\sqrt{J+1}-\sqrt{J-1}\right)^2, 
\end{equation}
cf. \cite[Thm. 3.4]{Bibinger18}.
%Bibinger and Trabs used their result to set up method of moments estimators with parametric $\sqrt{MN}$-rate of convergence for some of the parameters, for details see also Section \ref{section:ParameterEstimation} of the current text.
Note that this result is only valid under the condition $M=o(\sqrt{N})$, i.e., the observation frequency in time is much higher than in space. This constraint is due to a {non-}negligible correlation of realized temporal quadratic variations at two neighboring points in space if the distance $\delta$ of these points is small compared to $\Delta$ or, equivalently, if $M$ is large compared to $N$.  
\smallskip

%%%%%%%%%%%%%%%%%%%%%%%%%%%%%%%%%%%%%%%%%%%%%%%%%%%%%%%%%%%%%%%%%%%%%%%%%%%%%%%%%%%

In the situation where the number of spatial observations dominates the number of temporal observations the above result is not applicable. In this case, spatial increments $(\delta_k^M X)(t_i)= X_{t_i}(y_{k+1})-X_{t_i}(y_k)$ and the corresponding rescaled realized spatial quadratic variations
$$V_{\mathrm{sp}}(t_i):=\frac{1}{M\delta}\sum _{k=0}^{M-1} \e^{\kappa y_k}(\delta_k^MX)^2(t_i)$$
at time $t_i$ turn out to be useful. In contrast to squared time increments, which have to be renormalized by $\sqrt \Delta$ due to the roughness of $t\mapsto X_t(y)$, squared space increments have to be renormalized by $\delta$ due to the semi-martingale nature of  $y\mapsto X_t(y)$.

In the extreme case where observations are only available at one point $t$ in time (and assuming $\vt_1=\vt_0=0$ as well as $X_0=0$) \citet{Cialenco17} showed that $V_{\mathrm{sp}}(t)$ is asymptotically normal with $1/\sqrt M$-rate of convergence. An analogous result has been proved by \citet{shevchenkoEtAl2019} for the wave equation. 
Proposition \ref{prop:Ito_representation} reveals that $V_{\mathrm{sp}}(t)$ is in fact a rescaled realized quadratic variation of the It\^o diffusion $y\mapsto X_t(y)$. Hence,
$$\sqrt M \left (V_{\mathrm{sp}}(t)-\frac{\sigma^2}{2\vt_2}\right)\overset{\mathcal D}\longrightarrow \mathcal{N}\left(0,\frac{\sigma^4}{2\vt_2^2}\right),\qquad M\to \infty,$$  
follows from standard theory on quadratic variation for semi-martingales.
In order to generalize this central limit theorem to the double asymptotic regime $M,N\to \infty$, we define the time average of the \emph{rescaled realized spatial quadratic variations}:
\begin{equation}
  V_{\mathrm{sp}}:=\frac{1}{N}\sum_{i=0}^{N-1}V_{\mathrm{sp}}(t_i)=\frac{1}{NM\delta}\sum_{i=0}^{N-1}\sum _{k=0}^{M-1} \e^{\kappa y_k}(\delta_k^MX)^2(t_i).\label{eq:Vsp} 
\end{equation}

\begin{thm} \label{thm:CLT_space}
Let $b\in[0,1/2)$. If
$
N/M\to 0 
$
then
$$\sqrt{MN}\left(V_\mathrm{sp}-\frac{\sigma^2}{2\vt_2}\right)\overset{\mathcal D}{\longrightarrow} \mathcal{N}\left(0,\frac{\sigma^4}{2\vt_2^2}\right),\quad M,N \to \infty.$$
\end{thm}

\begin{rem}
The condition $N/M\to 0$ is necessary in order to to neglect the bias: The proof of the theorem reveals that $\delta^{-1}\E\left( \e^{-\kappa y_k}(\delta_k^MX)^2(t_i)\right)-\frac{\sigma^2}{2\vt_2}\eqsim \delta$ and consequently, the overall bias is of the order
$$\E\left(\sqrt{MN}\left(V_\mathrm{sp}-\frac{\sigma^2}{2\vt_2}\right) \right)\eqsim {\sqrt{MN}}\cdot{\delta}\eqsim \sqrt{\frac{N}{M}}.$$
\end{rem}
%As for time increments, based on Theorem  \ref{thm:CLT_space} it is possible to define method of moments estimators for some of the parameters as long as $N/M\to 0$. We defer details to Section \ref{section:ParameterEstimation}.\\

%%%%%%%%%%%%%%%%%%%%%%%%%%%%%%%%%%%%%%%%%%%%%%%%%%%%%%%%%%%%%%%%%%%%%%%%%%%%%%%%%%%

We conclude that the central limit theorem for realized temporal quadratic variations $V_{\mathrm t}$ holds when (roughly) $M=o(\sqrt N)$, whereas the central limit theorem for realized spatial quadratic variations $V_\mathrm{sp}$ is fulfilled if $N=o(M)$. To close the remaining gap, we finally study the space-time  increments $D_{ik}$ from \eqref{eq:double_def}. The corresponding rescaled realized quadratic variations are robust with respect to the sampling regime, as indicated by the representation
$$D_{ik}=\sum_{\ell \geq 1}\big(u_\ell(t_{i+1})-u_\ell(t_{i})\big)\big(e_\ell(y_{k+1})-e_\ell(y_{k})\big)$$
in terms of the series expansion \eqref{eq:SPDE_solution}.

In contrast to the case of space increments (and in line with the result for time increments), we impose $b>0$ for the remainder of this section. Inspection of the proofs suggests that this condition may be relaxed to $b\to0$ as long as the decay is sufficiently slow. As a first step, we calculate the expectation of the double increments

\begin{prop} \label{prop:mean_double}
Let $b\in(0,1/2)$. Then:
\begin{enumerate}[(i)]
\item It holds uniformly in $0\leq k\leq M-1$ and $1\leq i \leq N-1$ that
\[
\E\left(D_{ik}^2\right)=\sigma^{2}\e^{-\kappa y_k}\,\Phi_{\vt}(\delta,\Delta)+\O\left(\delta\sqrt{\Delta}\left(\delta\wedge\sqrt{\Delta}\right)\right), \quad \max(\delta, \Delta)\to 0, 
\]
where
\begin{align*}
 \Phi_{\vt}(\delta,\Delta)&:=F_{\vt_2}(0,\Delta)\left(1+\e^{-\kappa \delta}\right)-2F_{\vt_2}(\delta,\Delta)\e^{-\kappa\delta/2}\qquad\text{and}\\
 F_{\vt_2}(\delta,\Delta)&:=\sum_{\ell\geq1}\frac{1-\e^{-\pi^{2}\vt_{2}\ell^{2}\Delta}}{\pi^{2}\vt_{2}\ell^{2}}\cos(\pi \ell \delta).
\end{align*}
\item Assuming that $r=\lim \delta /\sqrt \Delta\in [0,\infty]$ exists, $\Phi_{\vt}$ admits three different asymptotic regimes:
\begin{align}
\Phi_{\vt}(\delta,\Delta)&=\begin{cases}
\frac{1}{\vt_{2}}\cdot\delta+o\left(\delta\right),&r= 0,\\
{\psi_{\vt_2}(r)\cdot \sqrt \Delta +o(\sqrt \Delta),}&r \in (0,\infty),\\
\frac{2}{\sqrt{\vt_{2}\pi}}\cdot\sqrt{\Delta}+o (\sqrt{\Delta}),& r=\infty,\\
\end{cases}\qquad\text{where}\notag\\
\label{eq:psi_r_def}
\psi_{\vt_2}(r) 
%&= 2\int_{0}^{\infty}\frac{1-\e^{-\pi^{2}\vt_{2}z^{2}}}{\pi^{2}\vt_{2}z^{2}}(1-\cos(\pi r z))\,dz\\
&:=\frac{2}{\sqrt{\pi\vt_2}}\left(1-\e^{-\frac{r^2}{4\vt_2}}+\frac{r}{\sqrt{\vt_2}}\int_{\frac{r}{2\sqrt{\vt_2}}}^\infty\e^{-z^2}\,dz \right).
%&=\frac{2}{\sqrt{\vt_2}} \Psi \left(\frac{r}{\sqrt{\vt_2}}\right)
\end{align}
%where 
%$$\Psi(x)=\frac{1}{\sqrt \pi}\left(1-\e^{-x^2/4}+\frac{x}{2}\int_x^\infty \e^{-z^2/4}\,dz\right).$$
If moreover ${\delta}/{\sqrt{\Delta}} \equiv r \in (0,\infty)$, we have
\begin{equation} \label{eq:double_mean_r}
\Phi_{\vt}(\delta,\Delta)=\e^{-\kappa \delta/2}\psi_{\vt_2}(r)\cdot \sqrt \Delta +\mathcal O(\Delta^{3/2}).
\end{equation}

\end{enumerate}
\end{prop} 

\begin{rem}
The  first order constants appearing in the asymptotic expressions in $(ii)$ stem from a first derivative of $F_{\vt_2}(\cdot,\Delta)$ in 0 in case $r=0$ and a Riemann sum approximation of $F_{\vt_2}(\delta,\Delta)$ in case $r\neq 0$, respectively. 
Assuming for simplicity that $\kappa =0$, the proof of Proposition~\ref{prop:mean_double} shows a more precise expression for the remainder terms in case $r\in \{0,\infty\}$:
$$
\E\left(D_{ik}^2\right)=
\begin{cases}
\frac{1}{\vt_2}\cdot\delta +\O(\delta^2 /\sqrt{\Delta}), &r=0,\\
\frac{2}{\sqrt{\pi\vt_2}}\cdot \sqrt{\Delta} +\O(\Delta^{3/2}/\delta^2), &r=\infty.
\end{cases}
$$
 Thus, if our analysis of the remainder terms is sharp (which we believe is the case),  the first order approximations have a poor quality if $\delta/\sqrt \Delta$ converges slowly.
\end{rem}

Proposition \ref{prop:mean_double} suggests to renormalize double increments with $\delta$ if $\delta /\sqrt{\Delta}\to 0$ and with $\sqrt{\Delta}$ {otherwise}, which is in line with the renormalization of $V_\mathrm{sp}$ and $V_\mathrm{sp}$, respectively. However, this approach might not be feasible:
Firstly, it requires the knowledge which asymptotic regime is present, i.e., whether or not $\delta/\sqrt{\Delta} \to 0$. Especially for one given set of observations this information may be inaccessible. In this case renormalizing with $\Phi_{\vt}(\delta,\Delta)$ automatically captures the correct asymptotic regime. 
Secondly, if $r\in \{0,\infty\}$, the previous remark shows that the asymptotic expressions for $\Phi_{\vt}(\delta,\Delta)$ may lead to an undesirably large bias. In fact, in order to obtain a central limit theorem with $1/\sqrt{MN}$-rate of convergence, we would have to impose the assumptions $N^2/M \to 0$ and $M^5/N \to 0$,  respectively. These constraints are even more restrictive than the ones required for time or space increments.
\smallskip

Therefore, we renormalize with $\Phi_{\vt}(\delta,\Delta)$ and introduce the \emph{rescaled realized quadratic space-time variation}
$$\mathbb V := \frac{1}{MN\Phi_{\vt}(\delta,\Delta)}\sum_{k=0}^{M-1} \sum_{i=0}^{N-1}\e^{\kappa y_k}D_{ik}^2.$$ 

\begin{thm} \label{thm:CLT_double}
Let $b>0$.
If either ${\delta}/{\sqrt \Delta} \to r \in \{0, \infty\}$ or ${\delta}/{\sqrt \Delta} \equiv r \in(0, \infty)$, then 
$$\sqrt{MN}(\mathbb V-\sigma^2)\overset{\mathcal D}{\longrightarrow}\mathcal N \big(0, C\big({r/\sqrt{\vt_2}}\big)\sigma^4 \big),\quad N,M \to \infty,$$
where $ C(\cdot)$ is a bounded continuous function on $[0,\infty]$, given by \eqref{eq:C.r}, satisfying 
\[
C(0)=3 \qquad\text{and}\qquad
C(\infty)=3+\frac{3}{2}\sum_{J=1}^{\infty}\left(\sqrt{J-1}-\sqrt{J+1}-2\sqrt{J}\right)^{2}.
\]
\end{thm}
The condition $\delta/\sqrt \Delta\equiv r \in (0,\infty)$ can be relaxed to $\delta/\sqrt \Delta \to r\in (0,\infty)$ as long as the convergence is fast enough which we omit for the sake of simplicity. If $\delta/\sqrt \Delta\equiv r \in (0,\infty)$ holds,  \eqref{eq:double_mean_r} shows that the renormalization $\Phi_{\vt}(\delta,\Delta)$  and its first order approximation are close enough to be exchanged in the previous theorem. In this case we obtain a central limit theorem with a simpler renormalization which particularly does not depend on the model parameters:

\begin{cor} \label{cor:CLTr}
If $b>0$ and ${\delta}/{\sqrt{\Delta}}\equiv r \in (0,\infty)$,
then
\begin{equation}
  \mathbb V_r :=\frac{1}{MN\sqrt \Delta }\sum_{k=0}^{M-1} \sum_{i=0}^{N-1}\exp\left(\frac{\kappa}{2}(y_k+y_{k+1}) \right)D_{ik}^2\label{eq:Vr} 
\end{equation}
satisfies with $\psi_{\vt_2}(r)$ from \eqref{eq:psi_r_def} and $C(\cdot)$  from \eqref{eq:C.r}:
$$\sqrt{MN}\Big(\mathbb V_r-\psi_{\vt_2}(r)\sigma^2\Big)\overset{\mathcal D}{\longrightarrow}\mathcal N \Big(0, C(r/\sqrt{\vt_2})\psi_{\vt_2}^2(r)\sigma^4 \Big),\quad N,M \to \infty.$$
\end{cor}

\begin{rem}
  The previous central limit results are satisfied for a possibly growing time horizon $T_{N,\Delta}:= N\Delta$, too.  Theorem~\ref{thm:CLT_space} only requires that $T_{N,\Delta}>\eps$ for some $\eps>0$. Theorem~\ref{thm:CLT_double} holds if $T_{N,\Delta}=o(M)$ and, in particular, Corollary~\ref{cor:CLTr} is applicable if $N\Delta^{3/2} \to 0$.
\end{rem}

To end this section, we compare the realized quadratic variations $V_{\mathrm t}, V_\mathrm{sp}$ and $\mathbb V$ and their asymptotic variances. For this purpose, we scale the statistics in such a way that they are asymptotically centered around the same mean, say $\sigma^2$:
\begin{equation} \label{eq:Vprime.def}
V_\mathrm{t}^\prime=\sqrt{\pi \vt_2}V_\mathrm{t},\quad V_\mathrm{sp}^\prime=2\vt_2V_\mathrm{sp},\quad V^\prime = \mathbb V.
\end{equation}
For simplicity, let $\kappa =0$. Plugging in the asymptotic expressions for $\Phi_{\vt}(\delta,\Delta)$ from Proposition \ref{prop:mean_double} shows that 
$$V' \approx \frac{1}{2}\sum_{k=0}^{M-1} \sum_{i=0}^{N-1}D_{ik}^2\cdot
\begin{dcases}
\frac{2\vt_2}{NM\delta},& \delta/\sqrt{\Delta}\to 0,\\
\frac{\sqrt{\vt_2 \pi}}{NM\sqrt{\Delta}},& \delta/\sqrt{\Delta}\to \infty.
\end{dcases}
$$
Therefore, $V'$ approximately coincides with $V_\mathrm{sp}^\prime$ and $V_\mathrm{t}^\prime$ for $r\in \{0,\infty\}$, respectively, except for the factor $1/2$ and using double increments instead of time or space increments, respectively.

Further, denoting the asymptotic variances of $V'_\mathrm t,V'_\mathrm{sp}$ and $V'$ by $\mathfrak S_\mathrm{t}$, $\mathfrak S_\mathrm{sp}$ and $\mathfrak S(r)$, respectively, we  observe the relations
$\mathfrak S(\infty) = \frac{3}{2}\mathfrak S_\mathrm{t}$
and 
$\mathfrak S(0) = \frac{3}{2}\mathfrak S_\mathrm{sp}$, {where the factor $3/2$ occurs since each double increment consists of two space or time increments, respectively.}

%The presence of the factor $3/2$ may be explained as follows, e.g.~for space increments:  Since each double increment consists of two space increments and neighboring (in time) double increments have one space increment in common, the covariances that contribute to the asymptotic variance are given by
%\begin{align*}
%\V(D_{ik}^2)&=2\V(D_{ik})^2\approx 2\big(2\V((\delta_k^MX)(t_i))\big)^2 %=4 \V((\delta_k^MX)(t_i)^2) ,\\
%\Cov(D_{ik}^2,D_{i(k+1)}^2)&=2\Cov(D_{ik},D_{i(k+1)})^2\approx 2 (\V((\delta_{k}^MX)(t_i)))^2 =  \V((\delta_{k}^MX)(t_i)^2),\\
%\Cov(D_{ik}^2,D_{i(k-1)}^2)&\approx \V((\delta_{k}^MX)(t_i)^2),
%\end{align*}
%where we have applied Isserlis' theorem. Hence, we get a factor of $6/4=3/2$ in the asymptotic variance of $V'$.

\section{Parameter estimation} \label{sec:ParameterEstimation}
In view of the covariance structure of the observation vector and the fact that the value of $\vt_0$ is irrelevant from a statistical point of view (cf. Proposition \ref{prop:abs_cont}), we consider the parameter {vector}
$$\eta = (\sigma^2, \vt_2,\kappa).$$ 
It is straightforward to use the results from the previous section to construct method of moments estimators for the volatility parameter $\sigma^2$ or the diffusivity parameter $\vt_2$, provided that the other two parameters in $(\sigma^2,\vt_2,\kappa)$ are known, respectively. Doing so, we generalize the spatial increments based estimator from \cite{Cialenco17} to the double asymptotic regime and we complement the time increments based methods in \cite{Bibinger18,Chong18}. Our estimators do not hinge on $\vt_0$ (or $\Gamma$) such that the knowledge of its true value is not required.
\smallskip

Assuming firstly that $\vt_2$ and $\kappa$ are known, we obtain the following volatility estimators:
$$\hat \sigma^2_{\mathrm{sp}} := V_{\mathrm{sp}}',\qquad \hat \sigma^2_\mathrm{t}:=V_{\mathrm{t}}'\qquad\text{and} \quad  \hat \sigma^2 :=  \mathbb V$$
where $V_{\mathrm{sp}}'$ and $V_{\mathrm{t}}'$ have been introduced in \eqref{eq:Vprime.def}. 
\begin{prop}\label{prop:sigma}$\,$
\begin{enumerate}[(i)]
\item If $N = o(M)$, then we have 
$$\sqrt{MN}\left(\hat \sigma^2_\mathrm{sp}-\sigma^2\right)\overset{\mathcal D}\longrightarrow \mathcal{N}(0,2\sigma^4),\quad N,M \to \infty.$$
\item If $M = o(N^\rho)$ for some $\rho \in (0,1/2)$, then we have with $B$ defined in \eqref{eq:CLT_time_B}:
$$\sqrt{MN}\left(\hat \sigma^2_\mathrm{t}-\sigma^2\right)\overset{\mathcal D}\longrightarrow \mathcal{N}(0,
B\sigma^4),\quad N,M \to \infty.$$
\item If $\sqrt N=o(M)$, $M=o(\sqrt N)$ or $\sqrt{N}/M \equiv r_0>0$, then  we have with $r=r_0 \frac{1-2b}{\sqrt T}$ and $C(\cdot)$ from \eqref{eq:C.r}:
$$\sqrt{MN}(\hat \sigma^2-\sigma^2)\overset{\mathcal D}{\longrightarrow}\mathcal N (0, C(r/\sqrt{\vt_2})\sigma^4 ),\quad N,M \to \infty,$$
\end{enumerate}
\end{prop}

As discussed above, the double increments estimator has a larger variance than the single increments estimators. Hence, if one of the regimes $N = o(M)$ or $M= o(\sqrt{N})$ certainly applies, the single increments estimators are preferable. If none of the regimes is present or the situation is unclear, one can profit from the robustness of the double increments estimator with respect to the sampling regime.

If $N=o(M)$, the situation is close to that of $N$ independent semi-martingales (cf.~Proposition \ref{prop:Ito_representation}) and the asymptotic variance $2\sigma^4$ of the spatial increments estimator equals the Cram\'er-Rao lower bound for estimating $\sigma^2$, as can be seen by a simple calculation. Consequently, $\hat \sigma^2_\mathrm{sp}$ is an asymptotically efficient estimator.  The efficiency loss of the other estimators is due to the fact that for increasingly more temporal observations the infinite dimensional nature of the process  $X$ becomes apparent, leading to non-negligible covariances between increments.
\smallskip

If $\sigma^2$ and $\kappa$ are known, the diffusivity $\vt_2$ can be estimated by
$$\hat\vt_{2,\mathrm{sp}}:=\frac{\sigma^2}{2V_{\mathrm{sp}}}\quad\text{and}\quad\hat\vt_{2,\mathrm{t}}:=\frac{\sigma^4}{\pi V_{\mathrm{t}}^2}$$
using $V_{\mathrm{sp}}$ and $V_{\mathrm{t}}$ from \eqref{eq:Vsp} and \eqref{eq:Vt}, respectively.
Due to the non-trivial dependence of the renormalization $\Phi_{\vt}(\delta,\Delta)$ on $\vt$, it is not apparent how to construct a method of moments estimator for $\vt_2$ based on Theorem~\ref{thm:CLT_double} in general. However, if $\sqrt N/M\equiv r_0>0$, the renormalization can  be decoupled from the unknown parameter as exploited in Corollary~\ref{cor:CLTr}. Since the function  $\vt_2\mapsto \psi_{\vt_2}(r)$ has range $(0,\infty)$ and is monotonic, there is an inverse $H_r(\cdot)$ and we can define the method of moments estimator
\begin{equation*}
\hat \vt_{2,r} = H_r(\mathbb V_r / \sigma^2)
\end{equation*}
with $\mathbb V_r$ from \eqref{eq:Vr} and $r=\frac{\delta}{\sqrt\Delta}=r_0 \frac{1-2b}{\sqrt T}$.
As a direct consequence of the delta method, 
\begin{align*}
H_r'(\psi_{\vt_2}(r)) = \Big(\frac{\partial}{\partial \vt_2}\psi_{\vt_2}(r)\Big)^{-1} =-\vt_2^{3/2}\sqrt{\pi}\Big(1-\e^{-\frac{r^2}{4\vt_2}}+\frac{2r}{\sqrt{\vt_2}}\int_{\frac{r}{2\sqrt{\vt_2}}}\e^{-z^2}\,dz\Big)^{-1}
\end{align*}
and the above central limit theorems, we obtain:
\begin{prop}\label{prop:theta}$\,$
\begin{enumerate}[(i)]
\item If $N = o(M)$, then we have 
$$\sqrt{MN}\left(\hat\vt_{2,\mathrm{sp}}-\vt_2\right)\overset{\mathcal D}\longrightarrow \mathcal{N}\left(0,2\vt_2^2\right),\quad N,M \to \infty.$$
\item If $M = o(N^\rho)$ for some $\rho \in (0,1/2)$, then we have with $B$ from \eqref{eq:CLT_time_B}:
$$\sqrt{MN}\left(\hat\vt_{2,\mathrm{t}}-\vt_2\right)\overset{\mathcal D}\longrightarrow \mathcal{N}(0,{4 \vt_2^2
B}),\quad N,M \to \infty.$$
\item If $\sqrt{N}/M \equiv r_0>0$, then we have with $r=r_0 \frac{1-2b}{\sqrt T}$ and $C(\cdot)$ from \eqref{eq:C.r}:
$$\sqrt{MN}(\hat \vt_{2,r} -\vt_2)\overset{\mathcal D}{\longrightarrow}\mathcal N \bigg(0, C(r/\sqrt{\vt_2})\Big(\psi_{\vt_2}(r)\Big/\frac{\partial}{\partial \vt_2}\psi_{\vt_2}(r)\Big)^2 \bigg),\quad N,M \to \infty.$$
\end{enumerate}
\end{prop} 

We now consider parameter estimation when $(\sigma^2, \vt)$ is unknown. Recall from Proposition \ref{prop:abs_cont} and its subsequent discussion that $\vt_0$ cannot be estimated consistently on a finite time horizon. Moreover, it is not possible to estimate other parameters than $(\sigma^2/\sqrt{\vt_2},\kappa)$ or $(\sigma^2/{\vt_2},\kappa)$ only based on the temporal or the spatial covariance structure, respectively.
Estimation of $(\sigma^2/\sqrt{\vt_2},\kappa)$ via a least squares procedure based on temporal increments is disussed in \cite{Bibinger18} in the $M=o(\sqrt N)$ regime. Analogously, it is possible to estimate $(\rho^2,\kappa)$, where $\rho^2 = \sigma^2/\vt_2$,  using spatial increments and Theorem~\ref{thm:CLT_space}: Provided that $N=o(M)$, classical M-estimation theory reveals that
\begin{equation*} 
(\hat \rho^2, \hat \kappa) := \argmin_{(\tilde \rho^2,\tilde \kappa)} \sum_{k=0}^{M-1} \left(\frac{2}{N\delta}\sum_{i=0}^{N-1}(\si{ik})^2-\tilde \rho^2\e^{-\tilde \kappa y_k}\right)^2 
\end{equation*}
satisfies a central limit theorem with rate $1/\sqrt{MN}$. {We omit a detailed analysis of this estimator.} 
\smallskip

To estimate all three identifiable parameters $\eta=(\sigma^2,\vt_2,\kappa)$, we employ a least squares approach based on double increments. Due to the highly nontrivial dependence of the normalization $\Phi_\vt(\delta,\Delta)$ on $\vt$, a direct application of Theorem \ref{thm:CLT_double} is impossible. Assuming, however, a balanced design in the sense of $\delta/\sqrt \Delta\equiv r\in(0,\infty)$, we can use Corollary~\ref{cor:CLTr} where the normalization is decoupled from the unknown parameter $\vt$. 

Let $\delta/\sqrt \Delta\equiv r\in(0,\infty)$ and define $\bar{D}_{ik}:=D_{ik}+D_{(i+1)k}$ as well as $z_k = (y_k+y_{k+1})/2$. Corollary~\ref{cor:CLTr} suggests that 
$$\frac{1}{N\sqrt \Delta}\sum_{i=0}^{N-1}D_{ik}^2\approx \e^{-\kappa z_k}\sigma^2 \psi_{\vt_2}(r)\quad\text{and}\quad \frac{1}{N\sqrt {2\Delta}}\sum_{i=0}^{N-2}\bar D_{ik}^2\approx \e^{-\kappa z_k}\sigma^2 \psi_{\vt_2}(r/\sqrt 2).$$
By considering the two different sampling frequency ratios $r$ and $r/\sqrt 2$, we can distinguish $\sigma^2$ and $\vt_2$ instead of recovering only the product $\sigma^2\psi_{\vt_2}(r)$. 
To estimate $\eta=(\sigma^2,\vt_2,\kappa)$, we thus introduce the contrast process  
\begin{align*}
K_{M,N}(\tilde \eta)&:=K_{M,N}^1(\tilde \eta)+K^2_{M,N}(\tilde \eta)\quad\text{where}\\
K_{M,N}^1(\tilde \eta)&:=\frac{1}{M}\sum_{k=0}^{M-1}\Big(\frac{1}{N\sqrt \Delta}\sum_{i=0}^{N-1}D_{ik}^2 - f_{\tilde \eta}^1\left(z_k\right)\Big)^2,\\
K_{M,N}^2(\tilde \eta)&:=\frac{1}{M}\sum_{k=0}^{M-1}\Big(\frac{1}{N\sqrt{ 2\Delta}}\sum_{i=0}^{N-2}\bar D_{ik}^2 - f_{\tilde \eta}^2\left(z_k\right)\Big)^2,
\end{align*}
and $f_\eta^\nu(z) := \sigma^2 \e^{-\kappa z}\psi_{\vt_2}(r/\sqrt \nu),\,\nu=1,2$. The corresponding M-estimator is given by
\begin{equation} \label{eq:def_LSestimator}
\hat \eta = \argmin_{\tilde \eta\in H}K_{M,N}(\tilde \eta),
\end{equation}
where $H$ is some subset of $(0,\infty)^2\times \R$ containing the true parameter $\eta$.

\begin{thm} \label{thm:CLT_LS}
Assume $b>0$ and $\delta/\sqrt \Delta \equiv r>0$. If $\eta=(\sigma^2,\vt_2,\kappa)$ lies in the interior of $H$ for some compact set $H \subset (0,\infty)^2\times \R$, then
the least squares estimator $\hat \eta$ from \eqref{eq:def_LSestimator} satisfies
$$\sqrt{MN}(\hat{\eta}-\eta) \overset{\mathcal D}{\longrightarrow} \mathcal N(0,\Omega_{\eta}^r),\qquad M,N \to \infty,$$
where $\Omega_{\eta}^r\in \R^{3\times 3}$ is a strictly positive definite covariance matrix, explicitly given by \eqref{eq:LS_Cov}.
\end{thm}
\begin{rem}
{Based on $\hat\eta$, we can define $\hat \vt_1 :=\hat \eta_2 \hat \eta_3=\hat\vt_2\hat\kappa$ to estimate $\vt_1$. The delta method then yields a central limit theorem for $(\hat \sigma^2,\hat \vt_2,\hat\vt_1)$.}
\end{rem}

Even when $\delta/\sqrt\Delta \equiv r>0$ does not hold, there are always subsets of the data having the balanced sampling design. Hence, the estimation procedure treated in Theorem \ref{thm:CLT_LS} can be generalized to an arbitrary set $\{X_{t_i}(y_k),i\leq  N,\,k\leq M\}$  of discrete observations by considering an averaged version of the above contrast process. To that aim, choose $v,w \in \N$ such that $v\eqsim \max(1,N/M^2)$ and $w\eqsim \max(1,M/\sqrt N)$. Then, $\tilde \Delta:=v\Delta$ and $\tilde \delta := w \delta$  satisfy
$$ r:= {\tilde \delta}/{\sqrt{\tilde \Delta}} \eqsim 1.$$
Using double increments on the coarser grid 
$$D_{v,w}(i,k)=X_{t_{i+v}}(y_{k+w})-X_{t_{i}}(y_{k+w})-X_{t_{i+v}}(y_{k})+X_{t_{i}}(y_{k}),$$ 
we set
$$\mathcal K^\nu_{N,M}(\tilde \eta) =\frac{1}{M-w+1}\sum_{k=0}^{M-w}\left(\frac{1}{(N-\nu v+1)\sqrt{\nu v \Delta}}\sum_{i=0}^{N-\nu v}D^2_{\nu v,w}(i,k)-f_{\tilde \eta}^\nu\Big(\frac{y_k+y_{k+w}}{2}\Big) \right)^2,$$ where 
$
f^\nu_\eta(z)=2\sigma^2\psi_{\vt_2}(r/\sqrt \nu)\e^{-{\kappa}z}
$
and $\nu=1,2$. The final estimator for $\eta$ is then defined as 
\begin{align} \label{eq:def_LSestimator.1}
\hat \eta_{v,w}= \argmin_{\tilde \eta \in H}\big(\mathcal K^1_{N,M}(\tilde \eta) +\mathcal K^2_{N,M}(\tilde \eta) \big).
\end{align}

The rate of convergence of this estimation procedure is inherited from the observations on the coarser grids $\{(t_{i+jv},y_{k+lw}):0\le j\le N/v-1,0\le l\le M/w-1\}$, $i=0,\dots,v-1,k=0,\dots,w-1,$ on which we calculate the double increments. Each such subset consists of 
$$\frac{M}{w}\cdot \frac{N}{v}\eqsim (M\wedge N^{1/2})(N\wedge M^2)=M^3\wedge N^{3/2}$$ 
observations and has  a balanced design by construction. Therefore, Theorem~\ref{thm:CLT_LS} implies the convergence rate $1/\sqrt{M^3\wedge N^{3/2}}$. 

\begin{prop} \label{prop:LSestimator.1_rate}
Assume $b>0$ and let $\eta=(\sigma^2,\vt_2,\kappa)$ lie in the interior of $H$ for some compact set $H \subset (0,\infty)^2\times \R$. If there exist values $v\eqsim \max(1,N/M^2)\in\N$ and $w\eqsim \max(1,M/\sqrt N)\in\N$ such that $w\delta/\sqrt{v\Delta}$ is constant, then the estimator given by \eqref{eq:def_LSestimator.1} satisfies 
$$ \Vert \hat \eta_{v,w}-\eta \Vert =\O_{P}\Big( \frac{1}{\sqrt{M^3\wedge N^{3/2}}}\Big),\quad M,N\to \infty.$$
%$$\E\left( \Vert \hat \eta - \eta\Vert^2 \right)=\O\left(\frac{1}{N^{3/2}\wedge M^3}\right)$$
\end{prop}

\begin{rem}
Integer values $v $ and $w $ such that $w\delta/\sqrt{v\Delta} $ is constant exist, for instance, if the observations are recorded at a diadic grid, i.e.~$M=2^m$ and $N=4^{n}$ where $m,n\to \infty$. 
\end{rem}
Compared to the thinning method of \cite{Uchida19}, this rate is a considerable improvement.
Indeed,  it is (almost) optimal in the minimax sense, as shown in Section~\ref{sec:lowerBound}.

%\noindent
%\textbf{Algorithm for estimation of $(\sigma^2,\vt_2,\vt_1)$}
%\begin{itemize}
%\item Choose $\mathfrak I, \mathfrak K \in \N$ such that $\tilde \Delta=\mathfrak I\Delta$ and $\tilde \delta = \mathfrak K \delta$ satisfy
%$$0\ll r:= \frac{\tilde \delta}{\sqrt{\tilde \Delta}}\ll \infty.$$
%\item Let $D_{\mathfrak I,\mathfrak K}(i,k)=X_{t_{i+\mathfrak I}}(y_{k+\mathfrak K})-X_{t_{i}}(y_{k+\mathfrak K})-X_{t_{i+\mathfrak I}}(y_{k})+X_{t_{i}}(y_{k})$ and for $\nu=1,2$
%$$K_\nu(\eta) =\frac{1}{M-\mathfrak K+1}\sum_{k=0}^{M-\mathfrak K}\left(\frac{1}{(N-\nu \mathfrak I+1)\sqrt{\mathfrak I \Delta}}\sum_{i=0}^{N-\nu \mathfrak I}D^2_{\nu \mathfrak I,\mathfrak K}(i,k)-f_\eta^\nu\left(\frac{y_k+y_{k+\mathfrak K}}{2}\right) \right)^2,$$ where 
%\begin{equation*}
%f^\nu_\eta(z)=2\sigma^2\psi_{r/\sqrt \nu}(\vartheta_2)\e^{-{\kappa}z}.
%\end{equation*}
%\item Define $\hat \eta= \argmin_{\tilde \eta \in H}(K_1(\eta)+K_2(\eta))$.
%\end{itemize}

\section{Lower bounds}\label{sec:lowerBound}
Our next theorem proves that the estimator $\hat\eta$ from \eqref{eq:def_LSestimator.1} for $\eta=(\sigma^2,\vt_2,\kappa)$ is optimal in the minimax sense, up to a logarithmic factor. To obtain a lower bound, it suffices to consider the sub-problem where $\vt_1=\vt_0=0$ and only $(\sigma^2,\vt_2)$ has to be estimated.
\begin{thm} \label{thm:lower_bound_main}
Let  $\vt_1=\vt_0=0$, $(\sigma^2,\vt_2)\in H$ for some open set $H\subset (0,\infty)^2$ and consider observations at $t_i=i/N,\,i\leq N,$ and  $y_k = b+k\delta,\,k\leq M,$ for some $b\in [0,1/2)\cap \Q$.  Then: 
\begin{enumerate}[(i)]
\item 
If $\min(M,N)$ remains finite, there is no consistent estimator of $(\sigma^2,\vt_2)$.
\item
There is a constant $c>0$ such that 
\begin{gather*}
\liminf_{M,N\to \infty}\, \inf_{T}\sup_{(\sigma^2,\vt_2)\in H}\P_{(\sigma^2,\vt_2)}\Big(\Big\Vert T -\binom{\sigma^2}{\vt_2}\Big\Vert>c\cdot r_{M,N}\Big)>0,\\
\text{ where }r_{M,N}:=
\begin{dcases}
N^{-3/4},&\frac{M}{\sqrt N} \gtrsim 1,\\
\Big(M^3 \log \frac{N}{M^2} \Big)^{-1/2} ,& \frac{M}{\sqrt N} \to 0.
\end{dcases}
\end{gather*}
and $\inf_{T}$ is taken over all estimators $T$ of $(\sigma^2,\vt_2)$ based on observations $\{X_{t_{i+1}}(y_k)-X_{t_{i}}(y_k),\,i< N,\,k\leq M\}$.
  \end{enumerate}
\end{thm}
\begin{rem}
The lower bound for the case $M/\sqrt N \gtrsim 1$ is also valid for estimators based on $\{X_{t_i}(y_j),\,i\leq N,k\leq M\}$ instead of the increments. We conjecture that this is also true for the case $M/\sqrt N \to 0$.
\end{rem}

This lower bound shows that, in general, $(\sigma^2,\vt_2)$ cannot be estimated with the parametric rate $1/\sqrt{MN}$, in contrast to a conjecture in \cite{Cialenco17}. Instead, we observe a phase transition in the rate depending on the sampling frequency. The parametric rate can only be attained for a balanced design $N\eqsim M^2$.

The proof of Theorem~\ref{thm:lower_bound_main} relies on the standard lower bound technique, cf. \citet{Tsybakov10}. Using an inequality by \citet{IbragimovHasminskii1981}, we will bound the Hellinger distance of the laws of the observations in terms of the corresponding Fisher information for suitably chosen reparametrizations of $(\sigma^2,\vt_2)$.  For each sampling regime we choose a reparametrization $(\gamma_1,\gamma_2)$ of $(\sigma^2,\vt_2)$ in such a way that $\gamma_1$ can be estimated with parametric rate, even without knowledge of $\gamma_2$. {Bounding the Fisher information for $\gamma_2$, we then obtain} a lower bound for the simpler problem of estimating the one dimensional parameter $\gamma_2$, assuming that $\gamma_1$ is known. Clearly, the resulting lower bound for $\gamma_2$ carries over to $(\gamma_1,\gamma_2)$ and consequently to $(\sigma^2,\vt_2)$. The main effort, noting that the observations are significantly correlated, is to derive sharp upper bounds for the Fisher information in the different sampling regimes. 
\smallskip

In the case  $M/\sqrt{N}\gtrsim 1$ we apply the following bound on the Fisher information for discrete observations of the first $M$ coefficient processes. Thanks to the Markov property, the probability density function for discrete observations of an Ornstein-Uhlenbeck process is provided by the transition density and allows for explicit computations.
\begin{prop} \label{prop:FisherCoeff}
Let $\vt_1=\vt_0=0$ and consider a sample $(u_{\ell}(i\Delta),\,\ell\leq M,\,i\leq N)$
where $(u_{\ell},\,\ell\in\N)$ are independent Ornstein-Uhlenbeck processes given by
\[
du_{\ell}(t)=-\lambda_{\ell}u_{\ell}(t)\,dt+\sigma\,d\beta_{\ell}(t),\quad u_{\ell}(0)\sim\mathcal{N}\left(0,\,\frac{\sigma^{2}}{2\lambda_{\ell}}\right).
\]
Consider the reparametrization $(\sigma^2,\rho^2)$ where $\rho^2={\sigma^2}/{\vt_2}$ and the corresponding Fisher information $ J_{N,M}$. For $\max(M,N)\to \infty$, the diagonal entries of $J_{N,M}$ satisfy 
\begin{equation}
  J_{N,M}(\sigma^2)= \mathcal{O} (N^{3/2}\wedge(MN))\qquad\text{and}\qquad 
  J_{N,M}(\rho^2)= \mathcal{O} (M^3\wedge(MN)).
\end{equation}
In particular, 
$\min \big(J_{N,M}(\sigma^2),J_{N,M}(\rho^2)\big)\lesssim N^{3/2}\wedge M^3$ for $\max(N,M)\to \infty.$
\end{prop} 
\begin{rem}$\,$
\begin{enumerate}
\item If $M\lesssim \sqrt N$ and $\sigma^2$ is known, Proposition \ref{prop:FisherCoeff} suggest a lower bound of $M^{-3/2}$ for estimation of $\vt_2$ in the spectral approach. Indeed, this rate is achieved by the maximum likelihood estimator for time continuous observations of the coefficient processes, cf. \cite{lototsky2009}. 
\item The reparametrization was chosen since $\sigma^2$ can be computed from the quadratic variation of any coefficient process $u_\ell$ when $N\to \infty$, while $\rho^2$ can be computed from the empirical variance of $\ell u_\ell(t_i),\,\ell \leq M,$ for a fixed $t_i$ as $M\to \infty$, even without knowledge of the other parameter, respectively.
 \end{enumerate}

\end{rem}
Letting $M\to \infty$, Proposition~\ref{prop:FisherCoeff} suggests that based on observations of the coefficient processes it is not possible to estimate $\sigma^2$ (and in particular $(\sigma^2,\vt_2)$) at a rate faster than $N^{-3/4}$. Further, assuming $\vt_1=0$, the eigenfunctions $e_\ell(\cdot)$ do not depend on unknown parameters and hence, the space-time discrete observations of the SPDE may be reconstructed from  $\{u_\ell(t_i),\,i \leq N, \ell \in \N\}$. Consequently,  the lower bound $N^{-3/4}$ carries over to discrete observations of the SPDE. 

Although the  lower bounds resulting from Proposition~\ref{prop:FisherCoeff} and Theorem~\ref{thm:lower_bound_main} are almost the same, their proofs require a very different reasoning if $M/\sqrt N \to 0$:
In this case, if $\sigma^2$ is known, Proposition~\ref{prop:theta} shows that it is possible to estimate $\vt_2$ with parametric rate of convergence based on discrete observations of the SPDE whereas Proposition~\ref{prop:FisherCoeff} suggests that $\vt_2= \sigma^2/\rho^2$  cannot be estimated at a faster rate than $M^{-3/2}$ based on  the coefficient processes. In particular, both observation schemes are not asymptotically equivalent in the sense of Le Cam.
\smallskip

To derive the lower bound in the case $M/\sqrt N \to 0$, we consider  the situation where observations are recorded at rational positions $y_k=\frac{k}{M},\,k = 1,\ldots,M-1$, where we work with $M-1$ instead of $M$ spatial observations for ease of notation. Thus, we potentially add spatial observations on the margin $[0,b)\cup(1-b,1]$ which can only increase the amount of information contained in the data. Since $e_\ell(\cdot)=\sqrt 2 \sin(\pi\ell \,\cdot)$ is the sine basis, trigonometric identities imply that the vectors 
$$\bar e_k := (e_k(y_1),\ldots,e_k(y_{M-1}))\in \R^{M-1},\qquad k\in\N,$$ 
satisfy $\bar{e}_{k+2M} =\bar{e}_{k}$ for all $k\in\N$ and $\langle\bar{e}_{k},\bar{e}_{l}\rangle = M\1_{\{k=l\neq M\}}-M\1_{\{k+l=2 M\}}$ for $k,l\leq 2M$. Equivalently, $(e_k)_{k=1,\dots,M-1}$ form an orthonormal basis with respect to the empirical scalar product and the relations for $(\bar e_k)_{k\geq 1}$ follow from the symmetry of the sine. Therefore, observing $\{X_{t_i}(y_k),\,i\leq N, k\leq M-1\}$ is equivalent to observing
\begin{equation} \label{eq:U_k_def}
\{U_k(t_i),\,k\leq M-1, i\leq N\},\qquad U_k(t):=\frac{1}{M}\left\langle X_{t}(y_\cdot),\bar e_k \right\rangle=\sum_{\ell \in \mathcal{I}_k^+}u_\ell (t)-\sum_{\ell \in \mathcal{I}_k^-}u_\ell (t),
\end{equation}
where $\mathcal I_k^+ := \{k+2M\ell,\,\ell \geq 0\},\,\mathcal I_k^- := \{2M-k+2M\ell,\,\ell \geq 0\}$. Since the sets $\mathcal I_k = \mathcal I_k^+\cup \mathcal I_k^-$ are disjoint for different values of $k$, the processes $\{U_1,\ldots,U_{M-1}\}$ are independent which simplifies the calculation of the Fisher information considerably.  Based on their spectral densities and Whittle's formula \eqref{eq:Whittle} for the asymptotic Fisher information of a stationary Gaussian time series, we obtain the following result for the increment processes $\bar{U}_k,\,k\leq M-1,$  defined by 
\begin{equation}
\label{eq:barU_k_def}
\bar U_k(j) := U_k(t_{j+1})-U_k(t_j),\qquad j=0,\ldots,N-1.
\end{equation}
\begin{prop} \label{prop:FisherCoeff_2}
Consider the parametrization $(\sigma_0^2,\vartheta_2)$ where $\sigma_0^2:=\sigma^2/\sqrt{\vt_2}$. If $M/\sqrt N \to 0$, the Fisher information $J_{M,N}$ with respect to $\vt_2$ of a sample $\{\bar U _k(j),\,j\leq N-1,k\leq M-1\}$ satisfies 
$$J_{M,N}(\vt_2)=\mathcal O\Big( M^{3} \log\frac{N}{M^2}\Big).$$
\end{prop}
Hereby, the reparametrization allows for estimation of $\sigma_0^2=\sigma^2/\sqrt{\vt_2}$ with parameteric rate based on time increments in the regime $M/\sqrt N \to 0$, even when $\vt_2$ is unknown. We have considered $\bar U_k$ instead of $U_k$ due to the technical reason that the $N$-th order Fourier approximation of the spectral density of the increment process is positive and hence, a spectral density as well. We conjecture that the same bound holds for the Fisher information of $U_k$. 

\section{Simulations}\label{sec:sim}
The following numerical example illustrates the asymptotic results for the estimators derived in Section~\ref{sec:ParameterEstimation}. In order to simulate $X$ on a grid in time and space, we have considered the approximation $X^K_{t_i}(y_k)= \sum _{\ell=1}^K u_\ell(t_i)e_{\ell}(y_k)$ where $K$ is a large number. Moreover, the Ornstein-Uhlenbeck processes $u_\ell$ are simulated exploiting their AR(1)-structure, namely $$u_\ell(0)=\frac{\sigma}{\sqrt{2\lambda_\ell}}N_0^\ell,\qquad u_\ell(t_{i+1})= \e^{-\lambda_\ell \Delta}u_\ell(t_{i})+\sigma \sqrt{\frac{1-\e^{-2\lambda_\ell \Delta}}{2\lambda_\ell}}N_{i}^\ell,\quad i\in \N,$$
where $(N_i^\ell )$ are independent standard normal random variables.

Hereby, we have considered a fixed number $N=625=25^2$ of temporal observations and $M\in\{10,  15 , 25  ,40  ,70, 110, 180, 300\}$. The margin was set to $b=0.1$. In general, an appropriate choice for the cut of frequency $K$ highly depends on these values. For our setting $K=70,000$ produced accurate results.  The parameters are chosen as $\sigma^2=0.1,\,\vt_2=0.5,\,\vt_1=-0.4$ and $\vt_0=0.3$.

First, we consider the estimators for the volatility $\sigma^2$ and the diffusivity $\vt_2$ which have been analyzed in Propositions~\ref{prop:sigma} and \ref{prop:theta}, respectively. Figure \ref{fig:sigmatheta} shows the normalized mean squared error based on 500 Monte Carlo iterations plotted against the logarithm of the sampling ratio $\sqrt N /M$. The simplified double increments estimator $\hat \vt_{2,r}$ is computed with $r=(1-2b)\frac{\sqrt N}{M}$. Using the same value for $r$, the simplified double increments estimator for $\sigma^2$ is computed by replacing the normalization $\Phi_\vt(\delta,\Delta)$ by $\e^{-\kappa \delta/2}\psi_{\vt_2}(r)\sqrt \Delta$. 

As expected, the estimators based on temporal increments only achieve the parametric rate of convergence as long as $M$ is not too large, whereas estimators based on space increments only work well when $M$ is not too small. The estimators based on double increments perform very well throughout any regime depicted in the plot. Even the simplified versions work surprisingly well, although their applicability is only supported by our theory as long as $M\eqsim \sqrt N$. In particular, the double increments estimator for $\sigma^2$ can barely be distinguished from the simplified one. The theory suggests that the estimators based on space increments or time increments should have a smaller mean squared error than the double increments estimators in the regimes $\sqrt N/M\to0$ or $\sqrt N/M\to\infty$, respectively. The simulation confirms this effect for time increments, while we would require larger values of $M$ to see the asymptotic behavior for space increments. {However, to simulate the spatial increments estimator for large $M$, a considerably larger value of $K$ turns out to be crucial since otherwise the statistical bias of the estimator is amplified by a numerical bias.}

\begin{figure}[t]
\centering
\includegraphics[scale=.35]{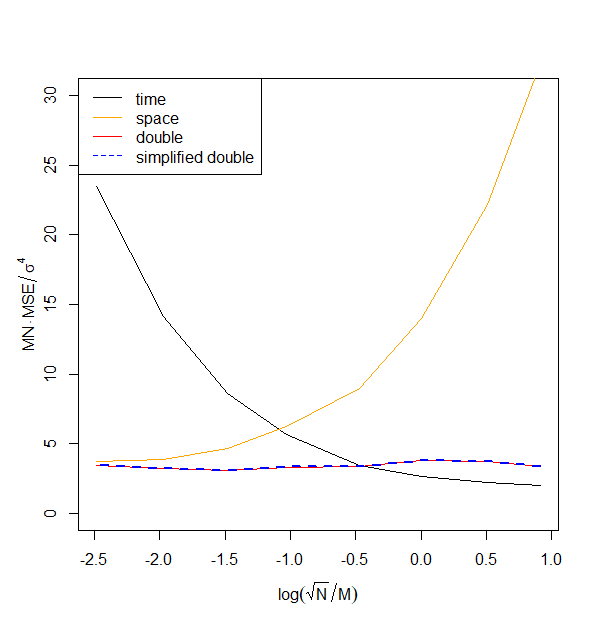}
\includegraphics[scale=.35]{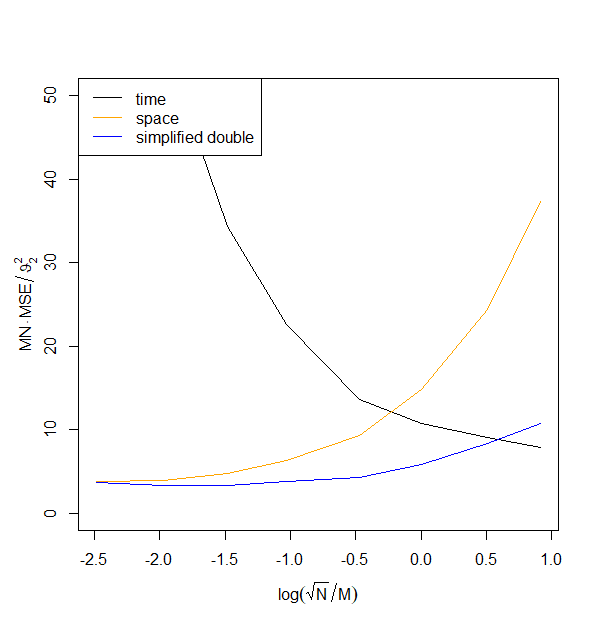}
\caption{Normalized mean squared errors of estimators for $\sigma^2$ (left) and $\vt_2$ (right) based on 500 Monte Carlo simulations.}
 \label{fig:sigmatheta}
\end{figure}

The above estimators require that all but one of the parameters $(\sigma^2,\vt_2,\kappa)$ are known. In the more difficult statistical problem where all parameters are unknown, $\eta= (\sigma^2,\vt_2,\kappa)$ can be estimated by $\hat \eta$ from \eqref{eq:def_LSestimator} and by $\hat \eta_{v,w}$ from \eqref{eq:def_LSestimator.1}. Figure~\ref{fig:LS} shows their mean squared error, again based on 500 Monte Carlo iterations. For the averaged estimator $\hat \eta_{v,w}$, we set $v=[ \max (1,\frac{N}{4M^2})]$ and $w=[\max(1,M/\sqrt N)]$ where $[\cdot]$ indicates rounding to the next integer. Since minimizing a functional of the type $\Vert F(\tilde \eta)\Vert^2$ for some function $F$ on a compact set is a hard numerical task we have considered the corresponding ridge regression problem, that is we minimize $\Vert F(\tilde \eta)\Vert^2+\lambda \Vert \tilde\eta\Vert^2$ instead. Regularizing with the squared inverse of the expected rate of convergence, i.e.~$\lambda = 1/({N^{3/2}\wedge M^3})$ for $\hat \eta_{v,w}$ and $\lambda=1/({N M})$ for $\hat \eta$ produced reasonable results, respectively. 

\begin{figure}[t] 
\centering
\includegraphics[scale=.35]{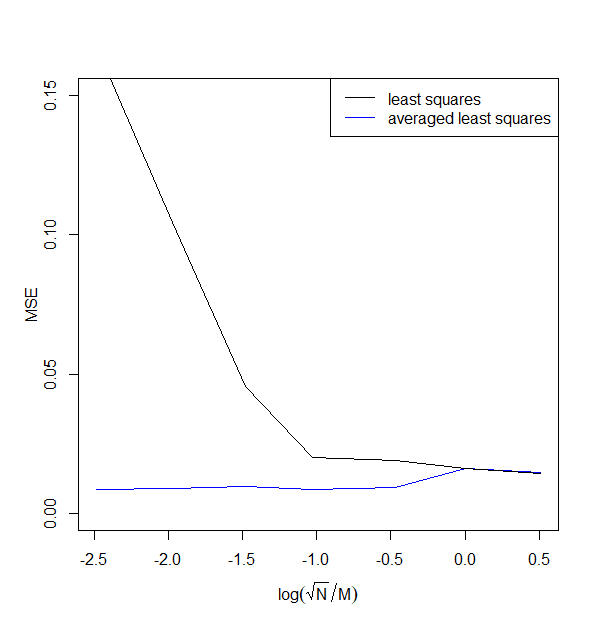}
\caption{Mean squared errors for the least squares estimator $\hat \eta$ and its averaged version $\hat \eta_{v,w}$.}
\label{fig:LS}
\end{figure}

In contrast to the double increments estimators for single parameters, $\hat \eta$ only produces good results as long as $M\eqsim \sqrt N$, which is covered by the theoretical foundation. The averaged version $\hat \eta_{v,w}$ works well throughout.  Furthermore, we see that it is only possible to profit from an increasing number of spatial observations up to a certain degree. Indeed, for $M \geq \sqrt N$ the optimal rate is $N^{-3/2}$ and the Monte Carlo mean squared error does not improve further. To cover also the regime $\sqrt N/M\to\infty$ for sufficiently large values of $M,N$, corresponding simulations are costly and not part of this simulation study.
The Monte Carlo mean squared error of $\hat \eta_{v,w}$ is not everywhere monotonic in $M$ since the effective sampling frequency ratio $r$ on the coarser grid where the double increments are computed is only approximately constant throughout the plot. Finally, we remark that our choice of $v$ and $w$ results in $v=w=1$ for the two smallest values of $M$ and hence, the two estimators are the same.

\section{Proofs of the main results}\label{sec:mainProofs}

\subsection{Proofs for the central limit theorems for realized quadratic variations}
First, we prove the generic central limit result in Proposition~\ref{prop:CLT_fund}. Afterwards, we can verify the central limit theorems for realized quadratic variations based on spatial increments (Theorem~\ref{thm:CLT_space}) and double increments (Theorem~\ref{thm:CLT_double}).
\begin{proof}[Proof of Proposition \ref{prop:CLT_fund}]
Since $\Sigma_{n}=Q_{n}^{\top}\Lambda_{n}Q_{n}$ for an orthogonal matrix
$Q_{n}\in\R^{d_{n}\times d_{n}}$ and a diagonal matrix $\Lambda_{n}$, the vector $Z_{\bullet,n}$ has the same distribution as $B_n X^n$ for $B_n:=Q_n^T\Lambda^{1/2}$ and $X^n:=(X_1,\ldots,X_{d_n})$ with independent standard normal random variables $(X_k)_{k\in \N}$. Denoting $A_n = \mathrm{diag}(\alpha_{1,n},\ldots,\alpha_{d_n,n})$, we obtain
$S_n = Z_{\bullet,n}^\top A_n Z_{\bullet,n} \overset{\mathcal D}{=}  {X^n}^\top B_n^\top  A_n B_n X^n$. Furthermore, $ B_n^\top  A_n B_n $ is symmetric such that $ B_n^\top  A_n B_n =P_n^\top \Gamma_n P_n$ where $P_n$ is an orthogonal matrix and $\Gamma_n$  is a diagonal matrix. Since $P_n X^n\sim\mathcal N(0,E_{d_n})$, we conclude as in \cite[p.~36]{Mathai92}
$$
  S_n\overset{\mathcal D}={X^n}^\top B_n^\top  A_n B_n X^n=(P_nX^n)^T \Gamma (P_nX^n)\overset{\mathcal D}{=}{X^n}^\top\Gamma_n X^n=\sum_{i=1}^{d_n}\gamma_{i,n}X_i^2,
$$ 
where $\gamma_{i,n},\,i\leq d_n$ are the eigenvalues of  $B_n^\top  A_n B_n$. 
The statement now follows by Lyapunov's condition and $\Vert B_n \Vert^2_2=\Vert \Sigma_n\Vert_2$:
\begin{align*}
\frac{\sum_{i=1}^{d_{n}}\gamma_{i,n}^{4}\E\left(\left(X_{k}^2-\E X_{k}^2\right)^{4}\right)}{\left(\V S_{n}\right)^{2}}
&\eqsim\frac{\sum_{i=1}^{d_{n}}\gamma_{i,n}^{4}}{\left(\sum_{i=1}^{d_{n}}\gamma_{i,n}^{2}\right)^{2}}
\lesssim\frac{\max_{i\leq d_{n}}\gamma_{i,n}^{2}}{\sum_{i=1}^{d_{n}}\gamma_{i,n}^{2}}
=\frac{\Vert B_n^T A_n B_n\Vert_{2}^{2}}{\V S_n}\\
&\leq \frac{(\Vert B_n \Vert_2^2\Vert A_n \Vert_2)^2}{\V S_n}=\frac{\Vert \Sigma \Vert_2^2}{\V S_n}. \qedhere
\end{align*}
\end{proof}
Throughout, for a function $f\colon\R \to \R$ we use  the notation
\begin{align*}
D_\delta f(x) &:= f(x+\delta)-f(x)\quad\text{and}\quad
D_\delta^2 f(x) := f(x+2\delta)-2f(x+\delta)+f(x).
\end{align*}
%%%%%%%%%%%%%%%%%%%%%%%%%%%%%%%%%%%%%%%%%%%%%%%%%%%%%%%%%%%%%%%%%%%%
\begin{proof}[Proof of Theorem \ref{thm:CLT_space}] 
We abbreviate the (rescaled) space increments by 
$$S_{ik}:=(\delta_k^M X)(t_i)\quad\text{and}\quad \tilde S_{ik}:=\e^{\kappa y_k/2}(\delta_k^M X)(t_i).$$
\textit{Step 1.} We calculate the asymptotic mean of $V_{\mathrm{sp}}$. Application of the  trigonometric identity
$
\sin(\alpha)\sin(\beta) =\frac{1}{2}\left(\cos(\alpha-\beta)-\cos(\alpha+\beta)\right) 
$
yields
\begin{equation} \label{eq:eig_mult}
\begin{aligned}
&\e^{\kappa x/2}(e_\ell(x+\delta)-e_\ell(x))\e^{\kappa y/2}(e_\ell(y+\delta)-e_\ell(y)) \\ 
%&=2\left(\e^{-\frac{\kappa}{2}\delta}\sin(\pi \ell(x+\delta))-\sin(\pi \ell x) \right)\left(\e^{-\frac{\kappa}{2}\delta}\sin(\pi \ell (y+\delta))-\sin(\pi \ell y) \right) \\
%&=\e^{-\kappa\delta}(\cos(\pi \ell (y-x))-\cos(\pi \ell (x+y+2\delta))) \\
%& \quad-\e^{-\frac{\kappa}{2}\delta}(\cos(\pi \ell (y-x-\delta))-\cos(\pi \ell (y+x+\delta))) \\
%& \quad-\e^{-\frac{\kappa}{2}\delta}(\cos(\pi \ell (y-x+\delta))-\cos(\pi \ell (y+x+\delta)))\\
%& \quad+\cos(\pi \ell (y-x))-\cos(\pi \ell (y+x)) \\
&=g(\delta)\left(2\cos(\pi \ell (y-x)) -\cos(\pi \ell (y-x-\delta))-\cos(\pi \ell(y-x+\delta))\right) \\
&  \quad+(g(2\delta)+g(0)-2g(\delta))(\cos(\pi \ell (y-x))) \\
& \quad+2g(\delta)\cos(\pi \ell (y+x+\delta))-g(0)\cos(\pi \ell (y+x))-g(2\delta)\cos(\pi \ell (x+y+2\delta)),
\end{aligned}
\end{equation}
where $g(x)=\exp(-\kappa x /2)$.
Plugging in $x=y$ gives
\begin{equation}
\begin{aligned} \label{eq:eig_single}
&\e^{\kappa y}(e_\ell(y+\delta)-e_\ell(y))^2\\
% &=2\left(\e^{-\frac{\kappa}{2}\delta}\sin(\pi k (x+\delta))-\sin(\pi k x) \right)^2\\
%&=\e^{-\kappa\delta}(1-\cos(2\pi \ell (y+\delta)))+(1-\cos(2\pi \ell y)) \\
%& \qquad  -2\e^{-\frac{\kappa}{2}\delta}\left( \cos(\pi \ell \delta)-\cos(\pi \ell(2y+\delta ))\right)  \\
%&=2g(\delta) (1-\cos(\pi \ell \delta)) +( g(2\delta)+g(0)-2g(\delta))\\
%&\qquad +2g(\delta)\cos(\pi\ell(2y+\delta))-g(0)\cos(2\pi \ell y)-g(2\delta)\cos(2\pi \ell(y+\delta))\\
&=2(1-\cos(\pi \ell \delta ))+2(1-g(\delta))(\cos(\pi \ell \delta) - 1)+(g(2\delta)+g(0)-2g(\delta))\\
&\qquad  +2g(\delta) \cos(\pi \ell(2y+\delta )) -g(2\delta)\cos(2\pi \ell (y+\delta))-g(0)\cos(2\pi \ell y).
\end{aligned}
\end{equation}
Thus, in terms of
$$f(y):= \sum_{\ell\geq 1} \frac{1}{2\lambda_\ell}\cos(\pi \ell y),\qquad y\in[0,1],$$
we have
\begin{align*}
&\E \left(\e^{\kappa y}\left(X_t(y+\delta)-X_t(y) \right)^2 \right)
 =\sigma^2 \sum_{\ell\geq 0}  \frac{1}{2\lambda_\ell}\e^{\kappa y} \left(e_\ell(y+\delta)-e_\ell(y) \right)^2\\
% &=\sigma^2 \left(-(F(\delta)-F(0))-(g(\delta)-g(0))(F(\delta)-F(0)) +\frac{1}{2}(g(2\delta)-2g(\delta)+g(0))\right.\\
% &\quad \left. -\frac{1}{2}(g(2\delta)f(2x+2\delta))-2g(\delta)f(2x+\delta)+g(0)f(2x) \right)\\
 &\qquad\qquad=\sigma^2 \left(-2D_\delta f(0)-2D_\delta g(0)D_\delta f (0)+f(0)D_\delta^2 g(0)-D^2_\delta (g(\cdot) f (2y+\cdot))(0) \right).
\end{align*}
Owing to its closed form expression in \eqref{eq:Fourier_expl} below, we see that $f\in C_b^\infty([0,2])$ and $f'(0)=-\frac{1}{4\vt_2}$. 
Hence,  
\begin{align*}
\E\left(\e^{\kappa y}(X_t(y+\delta)-X_t(y))^2 \right) &=-2\sigma^2 f'(0)\cdot \delta +\O (\delta ^2)
= \frac{\sigma^2 }{2\vt_2}\cdot \delta +\O (\delta ^2).
\end{align*}
For $y = y_{k}$ we obtain the asymptotic mean
$\E (V_\mathrm{sp})=\frac{\sigma^2}{2\vt_2} + \O(\delta)$ and in particular, under the condition $N/M\to 0$,
\begin{equation*} 
\sqrt{MN}\left(V_\mathrm{sp}-\frac{\sigma^2}{2\vt_2}\right) = \sqrt{MN}(V_\mathrm{sp}-\E(V_\mathrm{sp})) + o(1).
\end{equation*}
\textit{Step 2.} We calculate the asymptotic variance. By Isserlis' Theorem \cite{Isserlis18} we have 
$$
\Cov((\tilde S_{ik})^2,(\tilde S_{jl})^2)=2 \,\Cov(\tilde S_{ik},\tilde S_{jl})^2.
$$
Together with the symmetry $\Cov(\tilde S_{ik},\tilde S_{jl})=\Cov(\tilde S_{jk},\tilde S_{il})$ this implies
$$
\V(V_\mathrm{sp})=\frac{2}{N^2M^2\delta^2}(v_1+v_2+v_3+v_4)
$$
where 
\begin{gather*}
v_1 := \sum_{i=0}^{N-1} \sum_{k=0}^{M-1} \V (\tilde S_{ik})^2,\qquad
v_2 := 2\sum_{i=0}^{N-2}\sum_{j=i+1}^{N-1} \sum_{k=0}^{M-1} \Cov(\tilde S_{ik},\tilde S_{jk})^2 \\
v_3 := 2\sum_{i=0}^{N-1}\sum_{k=0}^{M-2} \sum_{l=k+1}^{M-1} \Cov(\tilde S_{ik},\tilde S_{il})^2,\qquad 
v_4 := 4\sum_{i=0}^{N-2}\sum_{j=i+1}^{N-1}\sum_{k=0}^{M-2} \sum_{l=k+1}^{M-1} \Cov(\tilde S_{ik},\tilde S_{jl})^2.
\end{gather*}
We have already shown that $\V (\tilde S_{ik})= \E((\tilde S_{ik})^2) =\frac{\sigma^2}{2\vt_2}\cdot \delta + \O(\delta^2)$. Therefore,
$$
v_1 = NM\delta^2\cdot\frac{\sigma^4}{4\vt_2^2}+\O\left(\frac{N}{M^2}\right)=NM\delta^2\cdot\frac{\sigma^4}{4\vt_2^2}+o\left(\frac{N}{M}\right).
$$
In the sequel, we show that the remaining covariances do not contribute to the asymptotic variance.

For $v_2$ we define $\omega := \vt_2 (\pi^2\wedge (\pi^2+\Gamma))>0$ such that $\lambda_\ell\geq \omega \ell^2$ for all $\ell\in \N$.
Since 
$\left(e_\ell(y_{k+1})-e_\ell(y_{k}) \right)^2\lesssim  \ell^2 \delta^2,$
we get for $J=|i-j|\geq 1$
\begin{align*}
\Cov(\tilde S_{ik},\tilde S_{jk})&=\sigma^2\sum_{\ell \geq 1} \frac{\e^{-\lambda_\ell J \Delta}}{2\lambda_\ell}\e^{\kappa y_{k}}\left(e_\ell(y_{k+1})-e_\ell(y_{k}) \right)^2
%\lesssim \sum_{\ell \geq 1} \frac{\e^{-\lambda_\ell J\Delta}}{2\lambda_\ell}\ell^2 \delta^2\\
%&\lesssim \delta^2\sum_{\ell\geq 1} \e^{-\lambda_\ell J\Delta}
\lesssim \delta^2\sum_{\ell \geq 1} \e^{-\omega \ell^2 J\Delta}\lesssim \frac{\delta^2}{\sqrt{J\Delta}}
\end{align*}
where the last step follows by Riemann summation with mesh size $\sqrt{J\Delta}$. Since $\frac{\log N}{M^2 \Delta}\leq\frac{N}{M^2 \Delta}=\frac{N^2}{M^2}\frac{1}{T}\to 0 $, 
$$
v_2 \lesssim \frac{M\delta^4}{ \Delta}\sum_{i=0}^{N-1}\sum_{j=i+1}^N  \frac{1}{(j-i)}\leq \frac{NM\delta^4}{ \Delta}\sum_{i=1}^{N}  \frac{1}{i}=\O\left(\frac{N \log N}{M^3 \Delta} \right)=o\left( \frac{N}{M} \right).
$$

To bound $v_3$ we follow the same strategy as for the mean: Since \eqref{eq:eig_mult} consists exclusively of second order differences we have $\Cov(\rsi{ik},\rsi{il})=\O(\delta^2)$ for $k\neq l$. Therefore,
$
v_3 = \O(NM^2\delta^4)=o(N/M).
$ 

To estimate $v_4$, we deduce from \eqref{eq:eig_mult} for $k<l$ and $J=|i-j|\geq 1$ that
\begin{align*}
\Cov(\tilde S_{ik},\tilde S_{jl})=&-g(\delta)D^2_{\delta} f_{J\Delta}\left(y_l-y_{k+1}\right)\\
&+f_{J\Delta}\left(y_l-y_{k}\right)D_{\delta}^2g(0)-D^2_{\delta} \left(g(\cdot)f_{J\Delta}\left(y_l+y_{k}+\cdot\right)\right)(0),\quad\text{where}\\
f_t(y):=& \sigma^2 \sum_{\ell\geq 1} \frac{\e^{-\lambda_\ell t}}{2\lambda_\ell}\cos(\pi  \ell y ).
\end{align*}
By Riemann summation we have
$f_t''(y)\lesssim \sum_{\ell\geq 1} \e^{-\lambda_\ell t}\lesssim \frac{1}{\sqrt t}.$ On the other hand, by Lemma \ref{lem:trig_series_fundamental}, 
$$f_t''(y)\lesssim \frac{1}{y\wedge (2-y)}\sup_k \left\vert\frac{k^2}{\lambda_k}\e^{-\lambda_k t} \right\vert\lesssim \frac{1}{y\wedge (2-y)}.$$
Therefore,
$$f_t''(y)\lesssim B(t,y):=\frac{1}{y\wedge (2-y)}\wedge\frac{1}{\sqrt t}.$$
Similarly, $f_t(y),\,f_t'(y)\lesssim B(t,y)$ can be shown. We conclude 
\begin{align*}
v_4 &\lesssim NM\sum_{i=0}^{N-1} \sum_{k=0}^{2M-2} \delta^4 B\left(i\Delta,\frac{k}{M}\right)^2
\lesssim \frac{N}{M^3}\sum_{i=1}^{N} \sum_{k=0}^M \frac{M^2}{k^2}\wedge \frac{1}{i\Delta} \\
&= \frac{N}{M^3}\sum_{i=1}^N \left( \sum_{k<M\sqrt{i\Delta}}  \frac{1}{i\Delta}  + \sum_{ M \geq k\geq M\sqrt{i\Delta}} \frac{M^2}{k^2}\right)
\lesssim \frac{N}{M^3}\sum_{i=1}^N \frac{M}{\sqrt{i\Delta}}
\lesssim \frac{N^{3/2}}{M^2 \sqrt{\Delta}}
=o\left( \frac{N}{M} \right)
\end{align*}
where the last step follows from $\frac{\sqrt{N}}{M\sqrt{\Delta}}=\frac{N}{M}\frac{1}{\sqrt{T}}\to 0$. Summing up, we have proved that 
\begin{equation*} 
\V(V_\mathrm{sp})=\frac{\sigma^4}{2\vt_2^2} \cdot \frac{1}{MN}+ o\left( \frac{1}{NM} \right).
\end{equation*}
\textit{Step 3.} 
To prove asymptotic normality, we interpret the number of temporal and spatial observations as sequences $M=M_n$, $N=N_n$ indexed by $n\in \N$ and consider the triangular array $(Z_{ik,n},n\in \N, k< M_n,i< N_n)$, where $Z_{ik,n}=\tilde S_{ik}/\sqrt{NM\delta}$. Since $\V(\sum_{i,k} Z_{ik}^2)\eqsim (MN)^{-1}$, Proposition~\ref{prop:CLT_fund} applies if:
$$\frac{1}{MN\delta^2}\Big(\sum_{i,k}|\Cov(\tilde S_{ik},\tilde S_{jl})|\Big)^2\to 0 $$
uniformly in $j<N,l<M$ in view of criterion \eqref{eq:lyapunov}.
The covariance bounds in Step 2 yield uniformly in $j<N,k<M$:
\begin{gather*}
  \sum_{k<M}|\Cov(\tilde S_{jk},\tilde S_{jl})|=\mathcal O(\delta),\qquad
  \sum_{i<N}|\Cov(\tilde S_{il},\tilde S_{jl})|=\mathcal O(\delta^2\sqrt{N}/\sqrt \Delta),\\
  \Big(\sum_{\substack{i\neq j,k\neq l}}|\Cov(\tilde S_{ik},\tilde S_{jl})|\Big)^2\lesssim MN \sum_{\substack{i\neq j,k\neq l}}|\Cov(\tilde S_{ik},\tilde S_{jl})|^2=o(N/M),
\end{gather*}
where we have used the Cauchy-Schwarz inequality to obtain the last bound. It remains to note $N/M\to 0$ and $N\Delta \gtrsim 1$.
\end{proof}

%%%%%%%%%%%%%%%%%%%%%%%%%%%%%%%%%%%%%%%%%%%%%%%%%%%%%%%%%%%%%%%%%%%%

The proof of Theorem~\ref{thm:CLT_double} is similar to the previous one but the more complex covariance structure of the double increments has to be taken into account carefully, see Section~\ref{sec:AuxDoubleIncr}. The (asymptotic) mean of the realized quadratic space-time variation is provided by Proposition~\ref{prop:mean_double}, which we prove first. In the following, we write
\begin{equation}
  \rsti{ik} := \e^{\kappa y_k /2}D_{ik}.\label{eq:DTilde}
\end{equation}

\begin{proof}[Proof of Proposition \ref{prop:mean_double}] 
\textit{Step 1. }We show asymptotic independence of $\Gamma$, i.e.,
\[
\E\left((\sti{ik})^{2}\right)=\sigma^{2}\sum_{\ell\geq1}\frac{1-\e^{-\pi^{2}\vt_{2}\ell^{2}\Delta}}{\pi^{2}\vt_{2}\ell^{2}}(e_{\ell}(y_{k+1})-e_{\ell}(y_k))^{2}+\O\left(\delta\sqrt{\Delta}\left(\delta\wedge\sqrt{\Delta}\right)\right).
\]
Define $f(x):=\frac{1-\e^{-x}}{x}$. A first order Taylor approximation
of $f$ yields 
\[
\E\left((\sti{ik})^{2}\right)=\sigma^{2}\Delta\sum_{\ell\geq1}f\left(\pi^{2}\vt_{2}\ell^{2}\Delta\right)(e_{\ell}(y_{k+1})-e_{\ell}(y_k))^{2}+R
\]
where 
$
R\lesssim\Delta^{2}\sum_{\ell\geq1}f'(\vt_{2}(\pi^{2}\ell^{2}+\xi_{\ell})\Delta)(e_{\ell}(x+\delta)-e_{\ell}(x))^{2}
$
for some $|\xi_k| \leq |\Gamma|$. Since 
\begin{align*}
(e_{\ell}(y+\delta)-e_{\ell}(y))^{2} %& \lesssim\left(\e^{-\kappa \delta/2}\sin(\pi \ell(y+\delta))-\sin(\pi \ell y)\right)^{2}\\
 & \lesssim\left(\e^{-\kappa \delta/2}(\sin(\pi \ell(y+\delta))-\sin(\pi \ell y))+\sin(\pi \ell y)(\e^{-\kappa\delta/2}-1)\right)^{2}
  \lesssim1\wedge\left(\ell\delta\right)^{2}
\end{align*}
and noting that $f'(x^2)$ and $x^2f'(x^2)$ are integrable,
we deduce 
\[
R\lesssim\Delta^{2}\sum_{\ell\geq1}(1\wedge(\ell\delta)^2)f'(\vt_2(\pi^{2}\ell^{2}+\xi_{\ell})\Delta)=\O\big(\Delta^{3/2}\wedge( \delta^{2}\sqrt{\Delta})\big)=\O\big((\delta\Delta)\wedge (\delta^{2}\sqrt{\Delta})\big).
\]

 \textit{Step 2. } We verify $(i)$. Thanks to Step 1 we may assume $\lambda_\ell = \pi^2\vt_2 \ell^2 $. It follows from \eqref{eq:eig_single} that
\begin{align*}
\E(\rsti{ik}^2)
=&\sigma^{2}\e^{-\kappa y}\left(F_{\vt_2}(0,\Delta)\left(1+\e^{-\kappa \delta}\right)-2F_{\vt_2}(\delta,\Delta)\e^{-\kappa\delta/2}\right)\\
&-\sigma^{2}\e^{-\kappa y}D^2_\delta\Big(g(\cdot)F_{\vt_2}(2y_k+\cdot\,,\Delta)\Big)(0).
\end{align*}
%This follows from applying the identities $\sin^{2}(\alpha)=\frac{1}{2}\left(1-\cos(2\alpha)\right)$
%and $\sin(\alpha)\sin(\beta)=\frac{1}{2}\left(\cos(\alpha-\beta)-\cos(\alpha+\beta)\right),\:\alpha,\beta\in\R$
%to
%\begin{align*}
%(e_{k}(x+\delta)-e_{k}(x))^{2} & =2\e^{-2Rx}\left(\e^{-R\delta}\sin(\pi k(x+\delta))-\sin(\pi kx)\right)^{2}\\
% & =\e^{-2Rx}\left(\e^{-2R\delta}\left(1-\cos(\pi k(2x+2\delta))\right)+\left(1-\cos(\pi k2x)\right)\right.\\
% & \qquad\qquad\left.-2\e^{-R\delta}\left(\cos(\pi k\delta)-\cos(\pi k(2x+\delta))\right)\right)\\
% & =\e^{-2Rx}\left(\e^{-2R\delta}+1-2\cos(\pi k\delta)\e^{-R\delta}\right.\\
% & \qquad\qquad\left.+2\e^{-R\delta}\cos(\pi k(2x+\delta))-\cos(\pi k2x)-\e^{-2R\delta}\cos(\pi k(2x+2\delta))\right).
%\end{align*}
Consequently, it remains to show  
\[
D^2_\delta\Big(g(\cdot)F_{\vt_2}(2y+\cdot\,,\Delta)\Big)(0)=\O\left(\delta\sqrt{\Delta}\left(\delta\wedge\sqrt{\Delta}\right)\right)
\]
uniformly in $y\in[b,1-b]$. As before, this is done by showing
\[
F_{\vt_2}(x,\Delta)\lesssim\Delta,\quad\frac{\partial F_{\vt_2}(x,\Delta)}{\partial x}\lesssim\Delta\quad\text{and}\quad \frac{\partial^2 F_{\vt_2}(x,\Delta)}{\partial x^2}\lesssim\sqrt{\Delta}
\]
uniformly in $x\in[2b,2(1-b)]$. By Lemma
\ref{lem:trig_integral} we have 
$
F_{\vt_2}(x,\Delta)=\Delta\sum_{\ell\geq1}f(\lambda_\ell\Delta)\cos(\pi\ell x)=\O(\Delta).
$
In order to access the first two derivatives of $F_{\vt_2}(\cdot,\Delta)$, we split
it into two summands, 
\begin{align*}
F_{\vt_2}(x,\Delta) & = \underbrace{\Delta\sum_{\ell\geq1}\frac{1}{1+\lambda_\ell\Delta}\cos(\pi \ell x)}_{=:H_{\Delta}(x)}+\underbrace{\Delta\sum_{\ell\geq1}\left(\frac{1-\e^{-\lambda_\ell\Delta}}{\lambda_\ell\Delta}-\frac{1}{1+\lambda_\ell\Delta}\right)\cos(\pi \ell x)}_{=:G_{\Delta}(x)}.
\end{align*}
Using the cosine series formula \eqref{eq:Fourier_expl}, we can compute
\begin{align*}
H_{\Delta}(x) & =\frac{1}{\vt_{2}\pi^{2}}\sum_{\ell\geq1}\frac{1}{\ell^{2}+\frac{1}{\pi^{2}\vt_{2}\Delta}}\cos(\pi \ell x)
  =\frac{\sqrt{\Delta}}{2\sqrt{\vt_{2}}}\frac{\cosh\left(\frac{1}{\sqrt{\vt_{2}\Delta}}(x-1)\right)}{\sinh\left(\frac{1}{\sqrt{\vt_{2}\Delta}}\right)}-\frac{\Delta}{2},
\end{align*}
from which it easily follows that $H'_\Delta(x)\lesssim \Delta$ and $H''_\Delta(x)\lesssim \sqrt \Delta$.
%\begin{align*}
%H_{\Delta}'(x) & \eqsim\frac{\sinh\left(\frac{\pi}{\sqrt{\vt_{2}\Delta}}(x-1)\right)}{\sinh\left(\frac{\pi}{\sqrt{\vt_{2}\Delta}}\right)}\lesssim\frac{\exp\left(\frac{\pi}{\sqrt{\vt_{2}\Delta}}|x-1|\right)}{\exp\left(\frac{\pi}{\sqrt{\vt_{2}\Delta}}\right)}\\
%&\lesssim\exp\left(-\frac{\pi}{\sqrt{\vt_{2}\Delta}}(x\wedge(2-x))\right)\lesssim\Delta,\\
%H_{\Delta}''(x) & \eqsim\frac{\cosh\left(\frac{\pi}{\sqrt{\vt_{2}\Delta}}(x-1)\right)}{\sqrt{\Delta}\sinh\left(\frac{\pi}{\sqrt{\vt_{2}\Delta}}\right)}\lesssim\frac{\exp\left(-\frac{\pi}{\sqrt{\vt_{2}\Delta}}(x\wedge(2-x))\right)}{\sqrt{\Delta}}\lesssim\sqrt{\Delta}.
%\end{align*}
The derivatives of
\begin{align*}
G_{\Delta}(x) & =\Delta\sum_{\ell\geq 1}h(\ell\sqrt \Delta)\cos(\pi \ell x),\quad\text{where}\quad h(z):=\frac{1-\e^{-z}(1+z)}{z(1+z)},
\end{align*}
can be bounded summand-wisely,
\begin{align*}
G_{\Delta}'(x) & \eqsim \sqrt{\Delta}\sum_{\ell \geq1}(\ell\sqrt{\Delta}) h(\ell\sqrt{\Delta})\sin(\pi \ell x) \lesssim\Delta,\quad
G_{\Delta}''(x)  \eqsim \sum_{\ell\geq1}(\ell^2 \Delta) h(\ell\sqrt{\Delta})\cos(\pi \ell x)\lesssim\sqrt{\Delta},
\end{align*}
where the bounds follow from the Riemann sum approximations in Lemma~\ref{lem:trig_integral}, owing to $xh(x)\vert_{x=0}=x^2 h(x)\vert_{x=0}=0$.

\textit{Step 3.} We show the asymptotic expressions in $(ii)$. 
Due to a Riemann sum argument, we have $\left\Vert F_{\vt_2}(\cdot,\Delta)\right\Vert _{\infty}\lesssim\sqrt{\Delta}$
and consequently,
\begin{align*}
\Phi_{\vt}(\delta,\Delta) & =2\left(F_{\vt_2}(0,\Delta)-F_{\vt_2}(\delta,\Delta)\right)+F_{\vt_2}(0,\Delta)\left[1+\e^{-\kappa \delta}-2\e^{-\kappa\delta/2}\right]\\
&\qquad-2\left(F_{\vt_2}(\delta,\Delta)-F_{\vt_2}(0,\Delta)\right)\left(\e^{-\kappa\delta/2}-1\right)\\
 & =2\left(F_{\vt_2}(0,\Delta)-F_{\vt_2}(\delta,\Delta)\right)+\O(\delta\sqrt{\Delta}).
\end{align*}
In the case $\delta/\sqrt{\Delta}\to 0$ Taylor's formula yields 
\[
F_{\vt_2}(0,\Delta)-F_{\vt_2}(\delta,\Delta)=-\delta \frac{\partial F_{\vt_2}(0,\Delta)}{\partial x}-\frac{\delta^{2}}{2}\frac{\partial^2 F_{\vt_2}(\eta,\Delta)}{\partial x^2}
\]
 for some $\eta\in[0,\delta]$. We employ the representation $F_{\vt_2}(\cdot,\Delta)=H_{\Delta}+G_{\Delta}$
from Step~2: Since $\sin(0)=0$ we have $\frac{\partial F_{\vt_2}(0,\Delta)}{\partial x}=H_{\Delta}'(0)=-\frac{1}{2\vt_{2}}$.
Further, $H_\Delta''(\eta)={1}/{\sqrt \Delta}$ and the Riemann sum argument yields
$
G''_\Delta(\eta)\lesssim
 \sum_{\ell\geq1}(\ell^2\Delta) h(\ell\sqrt \Delta)
  \lesssim {1}/{\sqrt{\Delta}}.
$
Therefore,
$
F_{\vt_2}(0,\Delta)-F_{\vt_2}(\delta,\Delta)=\frac{1}{2\vt_{2}}\cdot\delta+\O\left(\frac{\delta^{2}}{\sqrt{\Delta}}\right).
$

If $\delta/\sqrt{\Delta}\to \infty$, Lemma~\ref{lem:trig_integral} implies
$
F_{\vt_2}(\delta,\Delta)=-\frac{\Delta}{2}+\O(\frac{\Delta^{3/2}}{\delta^{2}})
$
and Lemma \ref{lem:Riemann} yields 
\begin{equation}
F_{\vt_2}(0,\Delta)=\sqrt{\Delta}\int_{0}^{\infty}\frac{1-\e^{-\pi^{2}\vt_{2}z^{2}}}{\pi^{2}\vt_{2}z^{2}}dz-\frac{\Delta}{2}+\O(\Delta^{3/2}).\label{eq:integral_at_0}
\end{equation}
Since 
$
\int_{0}^{\infty}\frac{1-\e^{-\pi^{2}\vt_{2}z^{2}}}{\pi^{2}\vt_{2}z^{2}}dz=\frac{1}{\sqrt{\vt_{2}\pi}},
$
we obtain 
$
F_{\vt_2}(0,\Delta)-F_{\vt_2}(\delta,\Delta)=\frac{\sqrt{\Delta}}{\sqrt{\vt_{2}\pi}}+\O (\frac{\Delta^{3/2}}{\delta^{2}}).
$

Finally, we derive the asymptotic expression for the case $\delta /\sqrt{\Delta}\equiv r$, while $\delta /\sqrt{\Delta}\to r$ can be handled similarly. We have
\begin{align*}\Phi_{\vt}(\delta,\Delta)&=2(F_{\vt_2}(0,\Delta)-F_{\vt_2}(\delta,\Delta))\e^{-\kappa \delta/2}+F_{\vt_2}(0,\Delta)(1+\e^{-\kappa \delta }-2\e^{-\kappa \delta /2})\\
&=2(F_{\vt_2}(0,\Delta)-F_{\vt_2}(\delta,\Delta))\e^{-\kappa \delta/2}+\O(\Delta^{3/2})
\end{align*}
and since $1-\cos(0)=0,$ Lemma \ref{lem:Riemann}
yields
\begin{align*}
F_{\vt_2}(0,\Delta)-F_{\vt_2}(r\sqrt{\Delta},\Delta) & =\sum_{\ell\geq1}\frac{1-\e^{-\pi^{2}\vt_{2}\ell^{2}\Delta}}{\pi^{2}\vt_{2}\ell^{2}}\left(1-\cos\left({\pi \ell r\sqrt{\Delta}}\right)\right)\\
 & =\sqrt{\Delta}\int_{0}^{\infty}\frac{1-\e^{-\pi^{2}\vt_{2}z^{2}}}{\pi^{2}\vt_{2}z^{2}}\left(1-\cos \left({\pi r z}\right)\right)dz+\O(\Delta^{3/2}).
\end{align*}
It remains to compute the integral. By substituting $\tilde r = r/\sqrt{\vt_2}$
we can pass to 
$$
\int_{0}^{\infty}\frac{1-\e^{-\pi^{2}\vt_{2}z^{2}}}{\pi^{2}\vt_{2}z^{2}}\left(1-\cos \left({\pi r z}\right)\right)dz=\frac{1}{\pi\sqrt{\vt_2}} \Big(h_1(\tilde{r})-h_2(\tilde{r})\Big)
$$
 where 
$$h_1(\tilde r)=\int_{0}^{\infty}\frac{1-\cos ({\tilde r z})}{z^{2}}dz,\qquad h_2(\tilde r)=\int_{0}^{\infty}\e^{-z^{2}}\frac{1-\cos ({\tilde r z})}{z^{2}}dz.
$$
To compute $h_1$, note that $S(z)+\frac{\cos(z)-1}{z}$ is an antiderivative of $\frac{1-\cos(z)}{z}$, where $S(z)=\int_0^z\frac{\sin(h)}{h}\,dh$ is the sine integral. Consequently, a substitution and $\lim_{z\to \infty} S(z)=\pi/2$ yields
$$h_1(\tilde r)=\tilde{r}\int_0^\infty \frac{1-\cos(z)}{z^2}\,dz=\frac{\pi \tilde r}{2}.$$
To treat $h_2$, note that $h_2(0)=h_2'(0)=0$ and hence,
$h_2(\tilde r)=\int_0^{\tilde r} \int_0^sh_2''(u)\,du\,ds$.
Now, plugging in
$
h_2''(\tilde r)= \int_0^\infty \e^{- z^2}\cos(\tilde r z)\,dz = \frac{\sqrt \pi}{2}\e^{-{\tilde r^2}/4}
$
and integrating by parts yields
\begin{align*}
h_2(\tilde r)&=\frac{\sqrt \pi}{2}\int_0^{\tilde r }\int_0^s \e^{-{u^2}/4}\,du
= \sqrt{\pi}\tilde r \int_0^{{\tilde{r}}/{2}} \e^{-u^2}\,du +\sqrt \pi \left( \e^{-\tilde r^2/4}-1\right) .
\end{align*}
The claim thus follows from 
\begin{align*}
h_1(\tilde r)-h_2(\tilde r)&=\frac{\pi\tilde r}{2}\left(1-\frac{2}{\sqrt \pi}\int_0^{{\tilde{r}}/{2}} \e^{-u^2}\,du\right) +\sqrt \pi \left(1- \e^{-\tilde r^2/4}\right)\\
&=\tilde r \sqrt \pi\int_{\tilde r /2}^{\infty} \e^{-u^2}\,du +\sqrt \pi \left(1- \e^{-\tilde r^2/4}\right).\qedhere
\end{align*}
\end{proof}

\begin{proof}[Proof of Theorem \ref{thm:CLT_double}]
Asymptotic normality follows just like in the proof of Theorem \ref{thm:CLT_space}. Using the notation from the proof of the latter theorem (with space increments replaced by double increments) we have 
$$\V(\mathbb V)=\frac{2}{M^2N^2\Phi_{\vt}^2(\delta,\Delta)}(v_1+v_2+v_3+v_4).$$
To determine the asymptotic variances, we have to treat the three different sampling regimes separately. 

\textit{Case $\delta/\sqrt{\Delta}\to 0 $.} By Lemmas \ref{lem:var_basic} and \ref{lem:var_dbyD0} we have 
\begin{align*}
\V(\rsti{ki})^2 &= \frac{\sigma^4}{\vt_2^2} \e^{-\kappa \delta}\cdot \delta^2 +o(\delta^2),\quad
\Cov(\rsti{ki},\rsti{k(i+1)})^2 =\frac{\sigma^4}{4\vt_2^2} \e^{-\kappa \delta}\cdot \delta^2 +o(\delta^2)
\end{align*}
as well as 
\begin{align*}
\Cov(\rsti{ki},\rsti{kj})^2 &= o \left(\frac{\delta^2}{|i-j|^5} \right),\quad |i-j|\geq 2,\\
\Cov(\rsti{ki},\rsti{lj})^2 &=\O \left( \frac{\delta^4}{(|i-j|+1)^4}\left(\frac{M^2}{(k-l)^2}\wedge \frac{1}{\Delta}\right)\right),\quad k\neq l.
\end{align*}
The latter covariances are negligible for the asymptotic variance since
$\sum_{k\leq M}(\frac{M^2}{k^2}\wedge\frac{1}{\Delta})\lesssim \frac{M}{\sqrt \Delta},$ cf.~the proof of Theorem \ref{thm:CLT_space}.
Inserting
$\Phi_{\vt}^2(\delta,\Delta) = \frac{\e^{-\kappa \delta}}{\vt_2^2} \delta^2 + o(\delta^2)$ from Proposition~\ref{prop:mean_double} yields the claim.

\textit{Case $\delta/\sqrt{\Delta}\to \infty$.} By Lemmas \ref{lem:var_basic} and \ref{lem:var_dbyDinf} we have
\begin{align*}
\V(\rsti{ki})^2 &= \frac{4\sigma^4}{\pi \vt_2}\e^{-\kappa \delta}\cdot \Delta +o(\Delta),\quad
\Cov(\rsti{ki},\rsti{(k+1)i})^2 = \frac{\sigma^4}{\pi \vt_2}\e^{-\kappa \delta}\cdot \Delta +o(\Delta).
\end{align*}
From $\sqrt{J-1}+\sqrt{J+1}-2\sqrt{J}=\O(J^{-3/2})$ and $\sqrt \Delta /\delta \to 0$ it follows for $J=|i-j|\geq 1$ that 
\begin{align*}
\Cov(\rsti{ki}, \rsti{kj})^2 &= \frac{\sigma^4}{\pi \vt_2}\left(\sqrt{J-1}+\sqrt{J+1}-2\sqrt{J}\right)^2\e^{-\kappa \delta}\cdot \Delta +o\left(\frac{{\Delta}}{J^{3/2}}\right)+\O(\Delta^3),\\
\Cov(\rsti{ki},\rsti{(k+1)j})^2 &= \frac{\sigma^4}{4\pi \vt_2}\left(\sqrt{J-1}+\sqrt{J+1}-2\sqrt{J}\right)^2\e^{-\kappa \delta}\cdot \Delta+o\left(\frac{{\Delta}}{J^{3/2}}\right)+\O(\Delta^3).
\end{align*}
Note that the $\O(\Delta^3)$-term is negligible for the asymptotic variance since
$$N^2 M\Delta^3 =MN\Delta \cdot N\Delta^2=MN\Delta \cdot \frac{T}{M}\cdot M\sqrt \Delta \cdot \sqrt \Delta =o(NM\Delta).$$
The remaining covariances do not contribute to the asymptotic variance since for \mbox{$|k-l|\geq 2$} we have
$$\Cov(\rsti{ki}, \rsti{lj})^2=\O\left(\frac{\Delta \delta^4}{(J+1)^{3}}\right)+\O\left(\frac{\Delta^2 }{(J+1)^{2}}\frac{M^2}{(k-l)^2}\right).$$
The claim is now proved by inserting  
$\Phi_{\vt}^2(\delta,\Delta) = \frac{4}{\pi \vt_2} \e^{-\kappa \delta} \Delta +o(\Delta)$
and noting that for the function $g(j)=(\sqrt{j-1}+\sqrt{j+1}-2\sqrt{j})^2$ we have 
$$\frac{1}{N}\sum_{\stackrel{i,j=0}{i\neq j}}^{N-1}g(|i-j|)=\frac{2}{N}\sum_{i=1}^{N-1}\sum_{j=1}^i g(j)\longrightarrow 2\sum_{j\geq 1}g(j),\quad N\to \infty$$
by Ces\`aro summation.

\textit{Case $\delta/\sqrt{\Delta}\equiv r \in (0,\infty)$.} For $f\colon\R^2\to \R$ define
\begin{align*}
D^2_xf(x,y):=f(x+2,y)+f(x,y)-2f(x+1,y),\\
D^2_yf(x,y):=f(x,y+2)+f(x,y)-2f(x,y+1).
\end{align*}
We show that the asymptotic variance is given by $C(r/\sqrt{\vt_2})\sigma^4$ where 
\begin{equation} \label{eq:C.r}
C(h):=\frac{2}{\Lambda^2_{0,0}(h)}\sum_{j,l\in \Z}\Lambda^2_{j,l}(h),\qquad\Lambda_{j,l}(h):= \left(D^2_xD^2_y G_h\right)(|j|-1,|l|-1)
\end{equation}
and
\begin{align}
G_h(j,l):=\sqrt{|j|}H \left(\frac{h|l|}{\sqrt{|j|}}\right)\1_{\{j\neq 0\}},\qquad
H(x)&:=\frac{1}{2\sqrt{\pi}}\left(\exp\left(-\frac{ x^2}{4{}}\right)-{ x }\int_{x/2}^\infty\e^{-z^2}\,dz\right).\label{eq:GH}
\end{align}
Define 
$$\xi_{i-j,k-l}^\Delta:=
\begin{cases}
2D_\delta F_{|i-j|,\Delta}(0),&l=k,\\
D_\delta^2F_{|i-j|,\Delta}((|k-l|-1)\delta),&l\neq k,
\end{cases}
$$ with $\delta=r\sqrt \Delta$ such that Lemma~\ref{lem:var_basic} reads as
\begin{equation}
  \Cov(\rsti{ik},\rsti{ik})=-\sigma^2 \e^{-\kappa \delta /2}\xi_{i-j,k-l}^\Delta+\O\left(\frac{\Delta^{3/2}}{(J+1^{3/2})}\right).\label{eq:CovDD} 
\end{equation}
Since each term $\xi_{J,L}^\Delta$ is a Riemann sum multiplied by $\sqrt \Delta $, we have for $J,L\geq 0$
\begin{align*}
&\lim_{\Delta \to 0 }\Delta^{-1/2} \xi_{J,L}^\Delta=-
\begin{cases}
2(\Psi_r(J,1))-\Psi_r(J,0)),&L=0,\\
\Psi_r(J,L-1)+\Psi_r(J,L+1)-2\Psi_r(J,L),& L\geq 1,
\end{cases}
\end{align*}
where
\[
\Psi_r(J,L):=\begin{cases}
{\displaystyle \int_{0}^{\infty}\frac{1-\e^{-\pi^{2}\vt_{2}z^{2}}}{\pi^{2}\vt_{2}z^{2}}\cos\left(\pi rLz\right)\,dz}, & J=0,\\
{\displaystyle \int_{0}^{\infty}\frac{2\e^{-J\pi^{2}\vt_{2}z^{2}}-\e^{-(J+1)\pi^{2}\vt_{2}z^{2}}-\e^{-(J-1)\pi^{2}\vt_{2}z^{2}}}{2\pi^{2}\vt_{2}z^{2}}\cos\left(\pi rLz\right)\,dz}, & J\geq1.
\end{cases}
\]
By symmetry of the cosine, 
\begin{align*}
\lim_{M,N\to \infty}\Delta^{-1/2}\xi_{J,L}=&- 
\Big(\Psi_r(J,|L|-1)+\Psi_r(J,|L|+1)-2\Psi_r(J,|L|)\Big)
\end{align*}
also holds for negative $L$. Hence, we can write for all $L\in\Z$ and $J\geq 0$ and with $G$ from \eqref{eq:GH}
\begin{align*}
\Psi_r(J,L)&=
{\displaystyle \int_{0}^{\infty}\frac{2\e^{-J\pi^{2}\vt_{2}z^{2}}-\e^{-(J+1)\pi^{2}\vt_{2}z^{2}}-\e^{-|J-1|\pi^{2}\vt_{2}z^{2}}}{2\pi^{2}\vt_{2}z^{2}}\cos\left(\pi rLz\right)\,dz} \\
&=\left(G_{r/\sqrt{\vt_2}}(J+1,L)+G_{r/\sqrt{\vt_2}}(J-1,L)-2G_{r/\sqrt{\vt_2}}(J,L)\right)/\sqrt{\vt_2},
\end{align*}
where the last equality follows from
$$\frac{G_{r/\sqrt{\vt_2}}(j,l)}{\sqrt{\vt_2}}=\int_{0}^{\infty}\frac{1-\e^{-|j|\pi^{2}\vt_{2}z^{2}}}{2\pi^{2}\vt_{2}z^{2}}\cos\left(\pi rlz\right)\,dz,\qquad j,l\in \Z,$$ 
which may be shown analogously to the calculation of $\psi_{\vt_2}(r)$.
Consequently, for all $J\in \{1-N,\ldots,N-1\}$ and $L\in \{1-M,\ldots,M-1\}$ we have
$$\lim_{M,N\to \infty}\Delta^{-1/2} \xi_{J,L}=-
\Lambda_{J,L}(r/\sqrt{\vt_2})/\sqrt{\vt_2}.$$
The usual Riemann sum argument yields $F_{J,\Delta}(0)\lesssim \frac{\sqrt \Delta}{(J+1)^{3/2}}\lesssim \frac{\sqrt \Delta}{(J+1)}$ for $J\geq 0$ and Lemma~\ref{lem:var_dbyDinf} (more precisely \eqref{eq:FJDeltaz_bound}) yields $F_{J,\Delta}(L\delta)\lesssim \frac{\Delta}{(J+1)L\delta}\lesssim \frac{\sqrt\Delta}{(J+1)(L+1)}$ for $J \in \N_0$ and $L\geq 1$.
We obtain
\begin{equation}
\Delta^{-1/2}\xi^\Delta_{J,L}=\O\left(\frac{1}{(|J|+1)(|L|+1)}\right),\qquad J,L \in \Z.
\end{equation} 
Therefore, 
\begin{align*}
\V \left( \frac{1}{\sqrt{NM\Delta}}\sum_{i=0}^{N-1}\sum_{k=0}^{M-1}\rsti{ik}^2 \right) 
%&=  \frac{2\sigma^4}{{NM\Delta}}\sum_{i,j=0}^{N-1}\sum_{k,l=0}^{M-1}\Cov(\rsti{ik},\rsti{jl})^2\\
 & =  \frac{2\sigma^4}{{NM}\Delta}\sum_{i,j=0}^{N-1}\sum_{k,l=0}^{M-1}(\xi_{i-j,k-l}^\Delta)^2 +o (1).
\end{align*}
By dominated convergence and taking Ces\`aro limits twice, we conclude
\begin{align*}
\lim_{M,N\to \infty}\V \left( \frac{1}{\sqrt{NM\Delta}}\sum_{i=0}^{N-1}\sum_{k=0}^{M-1}\rsti{ik}^2 \right) &=\lim_{M,N\to \infty }\frac{2\sigma^4}{{\vt_2}{NM}}\sum_{i,j=0}^{N-1}\sum_{k,l=0}^{M-1}\Lambda_{i-j,k-l}^2(r/\sqrt{\vt_2})\\ 
&=\frac{2\sigma^4}{\vt_2}\sum_{i,k\in \Z}\Lambda_{i,k}^2(r/\sqrt{\vt_2}).
\end{align*}
Since $\psi_{\vt_2}(r)=-\Lambda_{0,0}(r/\sqrt{\vt_2})/\sqrt{\vt_2}$, we have 
$\Phi_{\vt}^2(\delta, \Delta) =  \e^{-\kappa \delta} \Lambda_{0,0}^2(r/\sqrt{\vt_2})/\vt_2\cdot \Delta+o(\Delta)$
and dividing by $\lim_{M,N\to \infty} \Delta^{-1} \Phi_\vt^2(\delta,\Delta)=\Lambda^2_{0,0}(r/\sqrt{\vt_2})/\vt_2$ yields the claimed asymptotic variance.
\end{proof}

\subsection{Proofs for the estimators}

Propositions~\ref{prop:sigma} and \ref{prop:theta} follow immediately from the central limit theorems for the realized quadratic variations and the delta method. Before proving Theorem~\ref{thm:CLT_LS}, we introduce some notation that will be used throughout the proof and we state the asymptotic covariance matrix explicitly. 
Recall the definition of $\Lambda_{i,k}(\cdot)$ from \eqref{eq:C.r} and for any $i,k\in \Z$ let 
\begin{align*}
A^r_{ik} &:= -\Lambda_{ik}(r/\sqrt{\vt_2})/\sqrt{\vt_2},\qquad
&B^r_{ik}:=2 A^r_{ik}+A^r_{(i-1)k}+A^r_{(i+1)k},\qquad
&C^r_{ik}:=A^r_{ik}+A^r_{(i-1)k},\\
A_r &:=\sum_{i,k\in \Z} (A^r_{ik})^2,\qquad
&B_r :=\sum_{i,k\in \Z} (B^r_{ik})^2,\qquad\qquad\qquad
&C_r :=\sum_{i,k\in \Z} (C^r_{ik})^2.
\end{align*}
In terms of 
$$H(x) := \frac{4x}{\sqrt \pi} \left(1-\e^{-x^2}+2x \int_{x}^\infty\e^{-z^2}\,dz \right),\qquad H'(x)=\frac{4}{\sqrt \pi} \left(1-\e^{-x^2}+4x \int_{x}^\infty\e^{-z^2}\,dz \right),$$ 
$x\geq 0$, we have $\psi_{\vt_2}(r)= \frac{1}{r}H\big(\frac{r}{2\sqrt{\vt_2}}\big)$ and $\frac{\partial}{\partial \vt_2}\psi_{\vt_2}(r)=-H'\big(\frac{r}{2\sqrt{\vt_2}}\big)\frac{1}{4 \vt_2^{3/2}}$.
Denoting $r_i:=r/\sqrt i$, let
\begin{align*}
g_\eta^i(z):= \e^{-\kappa z}\left(\frac{1}{r_i}H\Big(\frac{r_i}{2\sqrt{\vt_2}}\Big),-\frac{\sigma^2}{4\vt_2^{3/2}}H'\Big(\frac{r_i}{2\sqrt{\vt_2}}\Big), -z\frac{\sigma^2}{r_i}H\Big(\frac{r_i}{2\sqrt{\vt_2}}\Big)\right)^\top,\quad h_\eta^i(z) :=\e^{-\kappa z } g_\eta^i(z)
\end{align*}
for $i=1,2$ and $z\in [b,1-b]$, where $g_\eta^i$ is the gradient of $\eta \mapsto f_{\eta}^i(z)$. Moreover, we write $\langle f,g \rangle_b:=\frac1{1-2b}\int_b^{1-b}f(x)g(x)dx$ for $f,g\in L^2([b,1-b])$. We will prove that the asymptotic covariance matrix equals
\begin{align} \label{eq:LS_Cov}
\Omega^r_{\eta}:= V^{-1} U V^{-1},
\end{align}  
where  $U=U(\eta)$ and $V=V(\eta)$ are defined via
\begin{align*}
U_{ij}&=4\sigma^4 \left(2A_r \langle (h^1_\eta)_i,(h^1_\eta)_j \rangle_b+B_r \langle (h^2_\eta)_i,(h^2_\eta)_j \rangle_b+\sqrt 2 C_r\big( \langle (h^1_\eta)_i,(h_\eta^2)_j \rangle_b+ \langle (h^2_\eta)_i,(h^1_\eta)_j \rangle_b\big) \right),\\
V_{ij} &= 2\left( \langle (g^1_\eta)_i,(g^1_\eta)_j \rangle_b +\langle (g_\eta^2)_i,(g^2_\eta)_j \rangle_b \right),\qquad\qquad i,j\in \{1,2,3\}.
\end{align*}

\begin{proof}[Proof of Theorem~\ref{thm:CLT_LS}]

The proof uses the classical theory on minimum contrast estimators, see e.g.~\cite{Dacunha86}. In particular, the mean value theorem yields $$-\dot K_{N,M}(\eta)=\dot K_{N,M}(\hat \eta)-\dot K_{N,M}(\eta)= \left(\int_0^1 \ddot K_{N,M}(\eta+\tau(\hat \eta -\eta))\,d\tau\right)(\hat \eta -\eta ) $$ 
as soon as $[\hat \eta,\eta]\subset H$, where $\dot K_{N,M}$ and $\ddot K_{N,M}$ denote gradient and Hessian with respect to $\eta$, respectively. In the sequel, we will verify that $K_{N,M}$ is associated with the contrast function $$K(\eta,\tilde \eta)= K^1(\eta,\tilde \eta)+K^2(\eta,\tilde \eta),\quad\text{where}\quad
K^i(\eta, \tilde \eta)=\frac{1}{1-2b}\int_b^{1-b}(f^i_\eta(z)-f^i_{\tilde{\eta}}(z))^2\,dz,$$ (Steps 1-2), show consistency of $\hat \eta$ (Step~3), prove asymptotic normality of $\dot K_{N,M}(\eta)$ with covariance matrix $U$ (Steps~4-7) and deduce stochastic convergence of $\int_0^1 \ddot K_{N,M}(\eta+\tau(\hat \eta -\eta))\,d\tau$ to the invertable matrix $V$ (Steps~8-9). The result then follows from Slutsky's Lemma and $-\sqrt{MN} {V(\eta)^{-1}}\dot K_{N,M}(\eta)\overset{\mathcal D}{\longrightarrow} \mathcal N (0,\Omega^r_\eta)$.

\emph{Step 1.}  We show that $K$
 is a contrast function in the sense that for each $\eta$ the function $\tilde \eta \mapsto K(\eta,\tilde \eta)$ attains its unique minimum in $\tilde \eta=\eta$.  Since $f^i_{\eta}(\cdot)$ is continuous it is sufficient to show that $(f^1_{\eta},f^2_{\eta})=(f^1_{\tilde \eta},f^2_{\tilde \eta})$ if and only if  $\eta= \tilde \eta$. Clearly, $(f^1_{\eta},f^2_{\eta})=(f^1_{\tilde \eta},f^2_{\tilde \eta})$ holds if and only if $\kappa = \tilde \kappa$ and $\sigma^2 \psi_{ \vt_2}( r_i)=\tilde \sigma^2 \psi_{\tilde \vt_2}( r_i)$ for $i=1,2$. Therefore, in order to prove identifyability, it is sufficient to show that $\vt_2\mapsto \psi_{\vt_2}(r_1)/\psi_{\vt_2}(r_2)$ is injective, which in turn is implied by strict monotonicity of $ H(r_1 z)/H(r_2 z)$ in $z>0$. We show that the corresponding derivative or, equivalently, the function $z\mapsto H'(r_1z)H(r_2z)r_1-H'(r_2z)H(r_1z)r_2$, is strictly negative for all $z>0$:
%\begin{align*}
%\frac{\partial}{\partial z} \frac{  H(r_1 z)}{ H(r_2 z)}= \frac{H'(r_1z)H(r_2z)r_1-H'(r_2z)H(r_1z)r_2}{H^2(r_2z)}
%\end{align*}
For $x>0$ define $p(x)=\int_x^\infty \e^{-z^2}\,dz$ and $q(x)=1-\e^{-x^2}$.
% we have $H(x)=\frac{4x}{\sqrt \pi}(h(x)+2xg(x))$ and $H'(x)=\frac{4}{\sqrt \pi}(h(x)+4xg(x))$
A simple calculation shows that
\begin{align*}
H'(r_1z)H(r_2z)r_1-H'(r_2z)H(r_1z)r_2=\frac{32}{\pi}r_1r_2 z \Big(p(r_1z)q(r_2z)r_1z-p(r_2z)q(r_1z)r_2z \Big)
\end{align*}
which is strictly negative if we can show that $p(b)q(a)b-p(a)q(b)a <0$ for all $0<a<b$.
%\begin{align*}
%H'(r_1z)H(r_2z)r_1-H'(r_2z)H(r_1z)r_2\\
%\propto (h(r_1z)+4r_1zg(r_1z))r_2z(h(r_2z)+2r_2zg(r_2z))r_1 - (h(r_2z)+4r_2zg(r_2z))r_1z(h(r_1z)+2r_1zg(r_1z))r_2\\
%=(h(r_1)+4r_1g(r_1))r_2(h(r_2)+2r_2g(r_2))r_1 - (h(r_2)+4r_2g(r_2))r_1(h(r_1)+2r_1g(r_1))r_2\\
%\propto 
%(h(r_1)+4r_1g(r_1))(h(r_2)+2r_2g(r_2))- (h(r_2)+4r_2g(r_2))(h(r_1)+2r_1g(r_1))\\
%= 2r_2g(r_2)h(r_1) + 4r_1 g(r_1) h(r_2) -2r_1g(r_1)h(r_2)-4h(r_1)g(r_2)r_2\\
%=g(r_2)\left(2r_2h(r_1)-4r_2h(r_1) \right)+g(r_1) \left(4r_1h(r_2)-2r_1 h(r_2) \right)\\
%=-2g(r_2)\left(r_2h(r_1) \right)+2g(r_1) \left(r_1h(r_2) \right)\\
%\propto g(r_1) r_1h(r_2) -g(r_2)r_2h(r_1). 
%\end{align*}
Now, a substitution yields 
$p(x) = x \int_1^\infty \e^{-x^2t^2}\,dt$ and $q(x)=2x^2\int_0^1s\e^{-x^2s^2}\,ds$
and therefore, 
\begin{align*}
p(b) q(a)b -p(a)q(b)a
%\int_0^1 \int_1^\infty b^2 \e^{-b^2 t^2} 2 a^2 s\e^{-a^2 s^2}- a^2 \e^{-a^2 t^2} 2 b^2 s\e^{-b^2 s^2}\,dt\,ds\\
=2a^2b^2\int_0^1 \int_1^\infty  s\left(\e^{-b^2 t^2-a^2s^2}-\e^{-a^2 t^2-b^2s^2} \right)\,dt\,ds<0
\end{align*}
follows from negativity of the integrand.
\\
In the sequel we follow the series of arguments from Theorem 5.1 of \cite{Bibinger18}.
\\
\emph{Step 2.} $K$ is the contrast function associated with the process $K_{N,M}$ in the sense that $K_{N,M}(\tilde \eta)\overset{\P_\eta}\longrightarrow K(\eta,\tilde \eta),\, N,M\to \infty,$ for all $\tilde \eta \in H$: 
Recall from the proof of Theorem \ref{thm:CLT_double} that for $i,j,k,l\in \N$ we have 
\begin{align} 
\Cov(D_{ik},D_{jl})&=\sigma^2 \e^{-\kappa\frac{z_k+z_l}{2}}
\xi^\Delta_{i-j,k-l} +\O\left(\frac{\Delta^{3/2}}{(|i-j|+1)^{3/2}}\right), \label{eq:xi.ikDelta}\\
\xi^\Delta_{i,k}&=\O\left(\frac{\sqrt \Delta}{(|i|+1)(|k|+1)}\right)\label{eq:xi.ikDelta.bound}
\end{align}
and $\lim_{N,M\to \infty} \Delta^{-1/2}\xi^\Delta_{i-j,k-l} =A^r_{ik}=-\Lambda_{ik}(r/\sqrt{\vt_2})/\sqrt{\vt_2}$. Now, in terms of
$$r_{ik}(\eta) = D_{ik}^2/\sqrt \Delta - f^1_{\eta}(z_k),\qquad R_{k}(\eta)=\frac{1}{N}\sum_{i=0}^{N-1}r_{ik}(\eta)$$
we can write
\begin{align} 
K^1_{N,M}(\tilde\eta)
 =& \frac{1}{M}\sum_{k=0}^{M-1}\left( f_{\eta}^1(z_k)
- f_{\tilde \eta}^1\left(z_k\right)\right)^2 \notag\\
&+\frac{2}{M}\sum_{k=0}^{M-1}R_k(\eta)\left( f_{\eta}^1(z_k) - f_{\tilde\eta}^1\left(z_k\right)\right)
+\frac{1}{M}\sum_{k=0}^{M-1} R_k^2(\eta). \label{eq:K1_decomp}
\end{align}
Clearly, the first summand converges to $K^1(\eta,\tilde\eta)$. To prove that the other two summands are negligible, note that 
\begin{align*}
\E(r_{ik}r_{jl})&=\E \left((D_{ik}^2/\sqrt \Delta - \E(D_{ik}^2/\sqrt \Delta) +\O(\Delta))(D_{jl}^2/\sqrt \Delta - \E(D_{jl}^2/\sqrt \Delta) +\O(\Delta)) \right)\\
&= \frac{1}{\Delta}\Cov(D_{ik}^2,D_{jl}^2) +\O(\Delta^2)\\
&= \frac{2}{\Delta}\Cov(D_{ik},D_{jl})^2 +\O(\Delta^2)
=\O\left( \frac{1}{(|i-j|+1)^2(|k-l|+1)^2}\right) + \O(\Delta^2).
\end{align*}
By Markov's inequality and boundedness of $\phi(\cdot)=f_\eta^1(\cdot) -f_{\tilde \eta}^1(\cdot)$, we have for any $\eps>0$,
\begin{align*}
\P \left( \left|\frac{1}{M}\sum_{k=0}^{M-1}R_k \phi(z_k)\right|\geq \eps\right) &\leq \frac{1}{\eps^2M^2}\sum_{k,l=0}^{M-1}|\E \left(R_k R_l \right)\phi(z_k)\phi(z_l)| \lesssim \frac{1}{M^2}\sum_{k,l=0}^{M-1}|\E \left(R_k R_l \right)|\\& \leq \frac{1}{M^2N^2}\sum_{k,l=0}^{M-1}\sum_{i,j=0}^{N-1}|\E \left(r_{ik} r_{jl} \right)|=o(1),
\end{align*}
hence, the second summand in \eqref{eq:K1_decomp} converges to zero in probability. For the third summand the same conclusion holds since
\begin{align*}
\E \left(\frac{1}{M}\sum_{k=0}^{M-1} R_k^2 \right)=\frac{1}{M}\sum_{k=0}^{M-1} \E \left(R_k^2 \right)=\frac{1}{MN^2}\sum_{k=0}^{M-1} \sum_{i,j=0}^{N-1}\E \left(r_{ik} r_{jk}\right)=o(1)
\end{align*}
and $L^1$-convergence implies convergence in probability.  $K_{N,M}^2$ can be handled similarly by considering a decomposition into two sums of non-overlapping increments:
$$\bar R_k(\eta)=2\left(\frac{1}{2N}\sum_{\substack{i\leq N-1\\i\text{ even}}}\bar r_{ik}(\eta)+\frac{1}{2N}\sum_{\substack{i\leq N-1\\i\text{ odd}}}\bar r_{ik}(\eta)\right)$$
where $\bar r_{ik} = \bar D_{ik}^2/\sqrt{2 \Delta} - f_\eta^2(z_k)$.
\\
\emph{Step 3.} Consistency of $\hat \eta$ follows from uniform convergence in probability of the contrast process. Since $K_{N,M}$ and $K$ are continuous, this in turn follows from
$$\forall \eps>0:\quad\lim_{h\to 0} \limsup_{M,N \to \infty}\P_\eta\left(\sup_{|\eta_1-\eta_2|<h}|K_{N,M}(\eta_1) - K_{N,M}( \eta_2)|\geq \epsilon \right)=0:$$
By compactness of the parameter space, for each $a>0$ there exists $h>0$ such that $\Vert f^i_{\eta_1}-f^i_{\eta_2}\Vert_\infty,\Vert (f^i_{\eta_1})^2-(f^i_{\eta_2})^2\Vert_\infty\leq a$ for all ${|\eta_1-\eta_2|<h}$. Therefore, 
\begin{align*}
&|K^1_{N,M}(\eta_1)-K^1_{N,M}(\eta_2)|\\
&\qquad\leq \frac{2}{M}\sum_{k=0}^{M-1}\left(\frac{1}{N\sqrt \Delta}\sum_{i=0}^{N-1}D_{ik}^2\right)|f_{\eta_2}^1(z_k)-f_{\eta_1}^1(z_k)|
+  \frac{1}{M}\sum_{k=0}^{M-1}| f_{\eta_1}^1(z_k)^2-f_{\eta_2}^1(z_k)^2|\\
&\qquad\leq a \left( \frac{2}{M}\sum_{k=0}^{M-1}\left(\frac{1}{N\sqrt \Delta}\sum_{i=0}^{N-1}D_{ik}^2\right)+  1\right)
\end{align*}
and hence, 
\begin{align*}
 &\limsup_{M,N \to \infty}\P_\eta \left(\sup_{|\eta_1-\eta_2|<h}|K^1_{N,M}(\eta_1)-K^1_{N,M}(\eta_2)|\geq \eps \right)\\
 &\qquad\qquad\leq  \limsup_{M,N \to \infty} \frac{1}{\eps} \E \left(\sup_{|\eta_1-\eta_2|<h}|K^1_{N,M}(\eta_1)-K^1_{N,M}(\eta_2)| \right)\\
 &\qquad\qquad\leq  \limsup_{M,N \to \infty} \frac{a}{\eps} \E \left( \frac{2}{M}\sum_{k=0}^{M-1}\left(\frac{1}{N\sqrt \Delta}\sum_{i=0}^{N-1}D_{ik}^2\right)+  1 \right)
 \lesssim \frac{a}{\eps}.
\end{align*}
The same argument applies to $K^2_{N,M}$ and the result follows.
\\
\emph{Step 4.} Let $F_1,F_2\in C^1([0,1])$ and $(a_k)_{k\in \Z}$ be absolutely summable. Then we can write
\begin{align*}
\frac{1}{n}\sum_{k,l=0}^{n-1}a_{k-l}F_1(z_k)F_2(z_l)=&\frac{a_0}{n}\left(F_1(z_0)F_2(z_0)+\cdots +F_1(z_{n-1})F_2(z_{n-1})\right)\\
&+\frac{a_1}{n}\left(F_1(z_1)F_2(z_0)+\cdots +F_1(z_{n-1})F_2(z_{n-2})\right)\\
&+\frac{a_{-1}}{n}\left(F_1(z_0)F_2(z_1)+\cdots +F_1(z_{n-2})F_2(z_{n-1})\right)
+\cdots 
%&+ a_{n-1}h_1(z_{n-1})h_2(z_{0}) + a_{1-n}h_1(z_{0})h_2(z_{n-1})
\end{align*} 
and, consequently, we have
$\frac{1}{n}\sum_{k,l=0}^{n-1}a_{k-l}F_1(z_k)F_2(z_l)\to \langle F_1,F_2 \rangle_b\cdot \sum_{k\in \Z}a_k,\,n\to \infty$, by dominated convergence.
\\
\emph{Step 5.} We show that the asymptotic covariance matrix of $\sqrt{NM}\dot K_{N,M}(\eta)$ is given by $U$: We have $\dot K_{N,M}(\eta)=\dot K^1_{N,M}(\eta)+\dot K^1_{N,M}(\eta)$ as well as
$$\dot K_{N,M}^1(\eta) = -\frac{2}{M}\sum_{k=0}^{M-1}\left(\frac{1}{N\sqrt \Delta}\sum_{i=0}^{N-1}D_{ik}^2 - f_\eta^1\left(z_k\right)\right) g^1_\eta(z_k) $$
and similarly for $\dot K_{N,M}^2(\eta)$.
From Isserlis' theorem, \eqref{eq:xi.ikDelta} and $\bar D_{ik} = D_{ik}+D_{(i+1)k}$ it follows that 
\begin{align*}
\Cov(D_{ik}^2,D_{jl}^2)&=2\left(\sigma^2 \e^{-\frac{z_k+z_l}{2}}\xi^\Delta_{i-j,k-l} +\O\left(\frac{\Delta^{3/2}}{(|i-j|+1)^{3/2}}\right)\right)^2,\\
\Cov(\bar D_{ik}^2,\bar D_{jl}^2)&=2\left(\sigma^2 \e^{-\frac{z_k+z_l}{2}}({2\xi^\Delta_{i-j,k-l}+\xi^\Delta_{i-j-1,k-l}+\xi^\Delta_{i-j+1,k-l}})+\O\left(\frac{\Delta^{3/2}}{(|i-j|+1)^{3/2}}\right)\right)^2,\\
\Cov( D_{ik}^2,\bar D_{jl}^2)&=2\left(\sigma^2 \e^{-\frac{z_k+z_l}{2}}({\xi^\Delta_{i-j,k-l}+\xi^\Delta_{i-j-1,k-l}}) +\O\left(\frac{\Delta^{3/2}}{(|i-j|+1)^{3/2}}\right)\right)^2.
\end{align*}
Now, for any $1\leq e,f\leq 3$, the first summand in the expansion 
\begin{align} 
\Cov((\dot K_{N,M})_e,(\dot K_{N,M})_f)=&\Cov((\dot K_{N,M}^1)_e,(\dot K_{N,M}^1)_f)+\Cov((\dot K_{N,M}^2)_e,(\dot K_{N,M}^2)_f) \nonumber \\
&+\Cov((\dot K_{N,M}^1)_e,(\dot K_{N,M}^2)_f)+\Cov((\dot K_{N,M}^2)_e,(\dot K_{N,M}^1)_f) \label{eq:U_parts}
\end{align}
is given by
\begin{align*} 
\Cov((\dot K_{N,M}^1)_e,(\dot K_{N,M}^1)_f)= \frac{4}{M^2 N^2 \Delta} \sum_{i,j=0}^{N-1} \sum_{k,l=0}^{M-1} \Cov(D^2_{ik},D^2_{jl}) \,(g^1_\eta)_e(z_k)\,(g^1_\eta)_f(z_l).
\end{align*}
Like in the proof of Theorem \ref{thm:CLT_double}, the covariances may be replaced by their asymptotic expressions due to dominated convergence. Further, using
$(h_\eta^i)_e(z) = \e^{-\kappa z}(g_\eta^i)_e(z)$ and Step 4, we have
\begin{align*}
MN\cdot\Cov((\dot K_{N,M}^1)_e,(\dot K_{N,M}^1)_f)\to 8\sigma^4 \sum_{i,k\in \Z}(A_{i,k}^r)^2 \cdot \langle (h^1_\eta)_e,(h^1_\eta)_f \rangle_b,\qquad M,N \to \infty.
\end{align*}
Analogously,
\begin{align*}
MN\cdot\Cov((\dot K_{N,M}^2)_e,(\dot K_{N,M}^2)_f)&\to 4\sigma^4 \sum_{i,k\in \Z}(B_{i,k}^r)^2 \cdot \langle (h^2_\eta)_e,(h^2_\eta)_f \rangle_b,\qquad M,N \to \infty,\\
MN\cdot\Cov((\dot K_{N,M}^1)_e,(\dot K_{N,M}^2)_f)&\to 4\sqrt{2}\sigma^4 \sum_{i,k\in \Z}(C_{i,k}^r)^2 \cdot \langle (h^1_\eta)_e,(h^2_\eta)_f \rangle_b,\qquad M,N \to \infty,
\end{align*}
and insertion into \eqref{eq:U_parts} yields the claimed asymptotic covariance matrix. 
%\begin{gather}
%U_{ef}= 8\sigma^4 \left(A_r \langle h^1_e,h^1_f \rangle_b+B_r \langle h^2_e,h^2_f \rangle_b+C_r\big( \langle h^1_e,h^2_f \rangle_b+ \langle h^2_e,h^1_f \rangle_b\big) \right),\qquad e,f \in\{1,2,3\},\\
%A_r=\sum_{i,k\in \Z}(A_{i,k}^r)^2, \quad B_r= \sum_{i,k\in \Z}(B_{i,k}^r)^2, \quad C_r =\sum_{i,k\in \Z}(C_{i,k}^r)^2 . \nonumber
%\end{gather}
\\
\emph{Step 6.} $U$ is strictly positive definite:
It is sufficient to show that $C_r<\sqrt{A_r B_r}$, then it follows for any $\alpha\in \R^3\setminus\{0\}$ and $H_\alpha^i = \sum_{j=1}^3 \alpha_j (h_\eta^i)_j,\,i=1,2$ that
\begin{align*}
\alpha^\top U \alpha&=4\sigma^4\left(2A_r \Vert H^1_\alpha\Vert_b^2 +B_r \Vert H^2_\alpha\Vert_b^2+2\sqrt{2}C_r\langle H^1_\alpha,H^2_\alpha \rangle_b \right)\\
&> 4\sigma^4\left(2A_r \Vert H^1_\alpha\Vert_b^2 +B_r \Vert H^2_\alpha\Vert_b^2+2\sqrt{2A_r B_r}\langle H^1_\alpha,H^2_\alpha \rangle_b \right)\\
&= 8 \sigma^4 \left\Vert \sqrt{ 2A_r} H_\alpha^1 +\sqrt B_r H_\alpha^2 \right\Vert_b^2 
\geq 0,
\end{align*}
where we may assume $\langle H^1_\alpha,H^2_\alpha \rangle_b<0$ since otherwise $\alpha^\top U\alpha>0$ follows immediately from the first equality. Now, consider $(A^r_{i,k})$ and $(B^r_{i,k})$  as elements in the Hilbert space $\ell^2$ of square summable sequences indexed by $\Z\times \Z$. Clearly, $A_r = \Vert (A^r_{i,k})\Vert^2_{\ell^2}$,  $B_r = \Vert (B^r_{i,k})\Vert^2_{\ell^2}$ and a direct calculation shows that $C_r = \langle (A^r_{i,k}),(B^r_{i,k}) \rangle _{\ell^2}$. Thus, by the Cauchy-Schwarz inequality we have $C_r \leq \sqrt{A_r B_r}$ and equality is ruled out by the fact that $(A^r_{i,k})$ and $(B^r_{i,k})$ are not linearly dependent.
\\
\emph{Step 7.} We show 
$\sqrt{NM}\dot K^1_{N,M}(\eta)\overset{\mathcal D}{\longrightarrow} \mathcal N (0,U)$ under  $\P_\eta$.
In view of the Cram\'{e}r-Wold device, we have to prove $\sqrt{NM}\alpha^\top \dot{K}_{N,M} \overset{\mathcal D}{\longrightarrow}\mathcal N(0,\alpha^\top U \alpha) $ for any $\alpha\in \R^3$. 
Let $s_{ik}$ and $Z_{ik}$ be given by the relation $s_{ik} Z_{ik}^2= -\frac{2\alpha^T \dot f_\eta ^1 (z_k)}{\sqrt{NM \Delta}}D_{ik}^2 $ where $s_{ik} \in \{-1,1\}$ is deterministic. Analogously, define $\bar s_{ik}$ and $\bar Z_{i,k}^2$. Then,
$\mathcal Z_{N,M}=(Z_{ik},\bar Z_{j,l})_{i,j,k,l}$ is a Gaussian vector and from Proposition \ref{prop:mean_double} it follows that
$$\sqrt{NM}\alpha^\top \dot K_{N,M}(\eta)=S_{N,M}-\E(S_{N,M})+o(1)$$ where $S_{N,M}=\sum_{i=0}^{N-1}\sum_{k=0}^{M-1}s_{ik} Z_{ik}^2+\sum_{i=0}^{N-1}\sum_{k=0}^{M-2}\bar s_{ik} \bar Z_{ik}^2$. From Steps 5 and 6 we can deduce that $\V\left(S_{N,M}\right)\to \alpha^\top U\alpha>0,\,N,M\to \infty$ and thus, in view of criterion \eqref{eq:lyapunov}, asymptotic normality follows if the absolute row sums of the covariance matrix of $\mathcal Z_{N,M}$ vanish uniformly. This in turn is a simple consequence of \eqref{eq:xi.ikDelta} and \eqref{eq:xi.ikDelta.bound}.
\\
\emph{Step 8.} In order to prove $\int_0^1 \ddot K_{N,M}(\eta+\tau(\hat \eta -\eta))\,d\tau \overset{\P_\eta}{\longrightarrow}V(\eta)$,  we show 
$\ddot K_{N,M}(\eta_{N,M})\overset{\P_\eta}{\longrightarrow}V(\eta)$ for any consistent estimator $\eta_{N,M}$ of $\eta$: We have
\begin{align*}
\ddot K_{N,M}^1(\eta) &= \frac{2}{M}\sum_{k=0}^{M-1}g_\eta^1(z_k) g_\eta^1(z_k)^\top 
-\frac{2}{M}\sum_{k=0}^{M-1}\left(\frac{1}{N\sqrt \Delta}\sum_{i=0}^{N-1}D_{ik}^2 - f_\eta^1\left(z_k\right)\right) \ddot f_\eta^1\left(z_k\right). 
\end{align*}
and analogously for $\ddot K_{N,M}^2$.
By using $\P_\eta(\eta_{N,M}\in H)\to 1$ and the uniform continuity of $f^i_\eta(z)$ and its derivatives in the parameter $(z,\eta)\in [0,1]\times H$, it is straightforward to show $\ddot K_{N,M}(\eta_{N,M})-\ddot K_{N,M}(\eta) \overset{\P_\eta}{\longrightarrow}0$. Now, write $V =2( V^1+V^2)$ where $V^i$ is the Gram matrix of the functions $\{(g_\eta^i)_1,\,(g_\eta^i)_2,\,(g_\eta^i)_3\}$ with respect to the inner product $\langle \cdot,\cdot \rangle_b$, i.e.~$V^i_{ef}=\langle (g_\eta^i)_e,(g_\eta^i)_f \rangle_b,\,1\leq e,f\leq 3$. Clearly, first summand  of $\ddot K^1_{N,M}(\eta)$
%$$\ddot K^1_{N,M}(\eta)=\frac{2}{M}\sum_{k=0}^{M-1}\dot f_\eta^1\left(z_k\right)^\top\dot f_\eta^1\left(z_k\right)-\frac{2}{M}\sum_{k=0}^{M-1}\left(\frac{1}{N\sqrt \Delta}\sum_{i=0}^{N-1}D_{ik}^2 - f_\eta^1\left(z_k\right)\right) \ddot f_\eta^1\left(z_k\right)$$
converges to $2V^1$ while the calculations of Step 2 show that the second summand converges to 0 in probability. The same reasoning holds for $\ddot K^2_{N,M}(\eta)$ and the result follows.
\\
\emph{Step 9.} $V$ is strictly positive definite: 
%$$g^i_1(z)=\e^{-\kappa z}\psi_{\vt_2}(r_i),\quad g^i_2(z)=\sigma^2\e^{-\kappa z}\frac{\partial\psi_{\vt_2}(r_i)}{\partial \vt_2},\quad g^i_3(z)=-z\e^{-\kappa z}\sigma^2\psi_{\vt_2}(r_i)$$
Being Gram matrices, $V^1$ and $V^2$ are positive semi-definite and consequently, the same holds for $V$. Clearly, the only way $V$ can be singular is if there exists $\alpha \in \R^3$ such that
$0=\alpha^\top V^i \alpha = \big\Vert\sum_{e=1}^3 \alpha_e (g_\eta^i)_e \big\Vert^2_{b}$ holds for both $i\in \{1,2\}$.
From the particular form of the functions $(g_\eta^i)_e$ it is apparent that this would imply that
$\alpha_1 \psi_{\vt_2}(r_i) + \alpha_2 \sigma^2
\frac{\partial\psi_{\vt_2}(r_i)}{\partial \vt_2}=\alpha_3 = 0$  for both $i\in\{1,2\}$, which is impossible.
\end{proof}

%Proposition~\ref{prop:LSestimator.1_rate} is a corollary from the previous proof.
\begin{proof}[Proof of Proposition \ref{prop:LSestimator.1_rate}]
We have to prove
$$\forall \eps>0 \,\exists C>0 :\limsup_{N,M \to \infty}\P_\eta\left(\sqrt{M^3\wedge N^{3/2}}\Vert\hat \eta_{v,w} -\eta\Vert \geq C \right) \leq \eps.$$
Similar calculations as in Theorem \ref{thm:CLT_LS} show that Steps 1-3 and 8-9 of the corresponding proof remain valid. Consequently, we have the representation $-\dot {\mathcal{K}}_{N,M} (\eta) = V_{N,M}(\hat \eta_{v,w},\eta)(\hat  \eta_{v,w} -\eta)$, where $V_{N,M}(\tilde \eta, \eta )=\int_0^1 \ddot{\mathcal K}_{N,M}(\eta + \tau (\tilde \eta- \eta))\,d\tau$ as well as $V_{N,M}(\hat \eta_{v,w},\eta) \overset{\P_\eta}{\longrightarrow} V(\eta)$ where $V(\eta)$ is an invertible deterministic matrix. In particular, the set
$$A_{N,M}=\left\{ V_{N,M}(\hat \eta_{v,w},\eta) \text{ is invertible with } \Vert V_{N,M}(\hat \eta_{v,w},\eta)^{-1} \Vert_2 \leq \Vert V(\eta)^{-1} \Vert_2+1\right\}$$
satisfies $\P_\eta(A_{N,M}) \to 1$. Further, $\dot{\mathcal K}_{N,M}(\eta)$ can be written as an average of expressions of the type $\dot K_{N,M}$ from Theorem \ref{thm:CLT_LS} so that the calculations of Step 5 show together with the Cauchy-Schwarz  inequality that $\E_\eta\left(\Vert \dot{\mathcal K}_{N,M}(\eta)\Vert^2\right)=\mathcal O((M^3\wedge N^{3/2})^{-1})$.
Now,
\begin{align*}
\P_\eta \left(\sqrt{M^3\wedge N^{3/2}}\Vert\hat \eta_{v,w} -\eta\Vert\geq C\right) 
&\leq \P_\eta \left(\big\{\sqrt{M^3\wedge N^{3/2}}\Vert\hat \eta_{v,w} -\eta\Vert\geq C\big\}\cap A_{N,M}\right) + \P_\eta(A_{N,M}^{\mathrm c}).
\end{align*}
The second summand becomes arbitrarily small as $M,N\to \infty$. For the first summand, let $\gamma(\eta) = \Vert V(\eta)^{-1} \Vert_2+1$, then it follows from Markov's inequality that
\begin{align*}
&\P_\eta\left(\{\sqrt{M^3\wedge N^{3/2}}\Vert\hat \eta_{v,w} -\eta\Vert\geq C\}\cap A_{N,M}\right) \\
&= \P_\eta\left(\{\sqrt{M^3\wedge N^{3/2}}\Vert V_{N,M}(\hat \eta_{v,w},\eta)^{-1}\dot{\mathcal K}_{N,M}(\eta)\Vert\geq C\}\cap A_{N,M}\right)\\
&\leq \P_\eta\left(\{\sqrt{M^3\wedge N^{3/2}}\Vert\dot{\mathcal K}_{N,M}(\eta)\Vert\geq \frac{C}{\gamma(\eta)}\}\cap A_{N,M}\right)\\
&\leq \P_\eta\left(\sqrt{M^3\wedge N^{3/2}}\Vert \dot{\mathcal K}_{N,M}(\eta) \Vert\geq \frac{C}{\gamma(\eta)}\right)
\leq (M^3\wedge N^{3/2})\E_\eta(\Vert \dot{\mathcal K}_{N,M}(\eta)\Vert ^2 ) \frac{\gamma(\eta)^2}{C^2}
\lesssim \frac{1}{C^2}.\qedhere
\end{align*}
\end{proof}

\subsection{Proofs of the lower bounds}
Before we prove Theorem~\ref{thm:lower_bound_main}, we verify its ingredients Proposition~\ref{prop:FisherCoeff} and Proposition~\ref{prop:FisherCoeff_2}.
\begin{proof}[Proof of Proposition \ref{prop:FisherCoeff}]
By setting $a=k^2$, $\mu=\pi^{2}\vartheta_{2}$
and $\nu^{2}=\frac{\sigma^{2}}{\pi^{2}\vartheta_{2}}$ in Lemma \ref{lem:FisherOU} and using
independence of $(u_{\ell},\,\ell\in\N)$
we get the Fisher information matrix $I$ for the parameters $(\mu, \nu^2)$, namely 
\begin{align*}
I_{11} & =N\sum_{\ell=1}^{M}\frac{{\ell^{4}}\Delta^2(\e^{-4\mu\ell^{2}\Delta}+\e^{-2\mu\ell^{2}\Delta})}{(1-\e^{-2\mu\ell^{2}\Delta})^{2}}=N\sum_{\ell=1}^{M}g_{11}(\ell\sqrt{\Delta}),\quad 
&g_{11}(x):=\frac{x^{4}(\e^{-4\mu x^{2}}+\e^{-2\mu x^{2}})}{(1-\e^{-2\mu x^{2}})^{2}},\\
I_{12} & =N\sum_{\ell=1}^{M}\frac{\ell^{2}\Delta\e^{-2\mu\ell^{2}\Delta}}{\nu^{2}(1-\e^{-2\mu\ell^{2}\Delta})}=N\sum_{\ell=1}^{M}g_{12}(\ell\sqrt{\Delta}),\quad 
&g_{12}(x):=\frac{x^{2}\e^{-2\mu x^{2}}}{\nu^{2}(1-\e^{-2\mu x^{2}})},\qquad\\
I_{22} & =\frac{(N+1)M}{2\nu^{4}}.
\end{align*}
The Fisher information matrix $J=J_{M,N}$ for the parameters  $(\sigma^2,\rho^2)$ can be computed via the change of variables formula $J= A^\top IA$ where
$$A = \begin{pmatrix}\pi^2/\rho^2&-\pi^2\sigma^2/\rho^4\\0&1/\pi^2 \end{pmatrix}$$
is the Jacobian of the function transforming $(\sigma^2,\rho^2)$ to $(\mu,\nu^2)$. Hence, the diagonal entries of $J$ are given by $$J_{11}=\frac{\pi^4}{\rho^4}I_{11},\qquad J_{22}=\frac{\pi^4\sigma^4}{\rho^8}I_{11}-\frac{2\sigma^2}{\rho^4}I_{12}+\frac{1}{\pi^4}I_{22}.$$
If $M\sqrt{\Delta}$ is bounded away from $0$, then $I_{11}$ can be interpreted as a Riemann sum. We obtain
\[
J_{11}\eqsim I_{11}\eqsim N^{3/2}\int_{0}^{M\sqrt{\Delta}}g_{11}(x)\,dx\eqsim N^{3/2}.
\]
On the other hand, if $M\sqrt{\Delta}\to0$, it follows
from Lemma \ref{lem:RiemannTaylor} and $g_{11}(0)=\frac{1}{2\mu^{2}}=\frac{\rho^4}{2\pi^4\sigma^{4}},\,g_{12}(0)=\frac{1}{2\mu\nu^{2}}=\frac{1}{2\sigma^2}$
as well as $g_{11}'(0)=g_{12}'(0)=0$ that
\begin{align*}
I_{11} & =N^{3/2}(M\sqrt{\Delta}g_{11}(0)+\frac{M^{2}\Delta}{2}g_{11}'(0)+\mathcal{O}(M^{3}\Delta^{3/2}))=\frac{\rho^4}{2\pi^4\sigma^{4}}NM+\mathcal{O}(M^{3}),\\
I_{12} & =N^{3/2}((M\sqrt{\Delta}g_{12}(0)+\frac{M^{2}\Delta}{2}g_{12}'(0)+\mathcal{O}(M^{3}\Delta^{3/2})))=\frac{NM}{2\sigma^{2}}+\mathcal{O}(M^{3}),\\
I_{22} &= \frac{\pi^4}{2\rho^4} MN + \O(M).
\end{align*}
Therefore, the leading terms in $J_{22}$ cancel and consequently, $J_{22}=\O(M^3)$.
\end{proof}

\begin{proof}[Proof of Proposition~\ref{prop:FisherCoeff_2}]

For a discrete time, centered, stationary Gaussian process $(Z_j)_{j\in\Z}$ whose covariance function depends on an unknown parameter $\theta\in\R$ we denote the Fisher information of a sample $(Z_0,\ldots,Z_{n-1})$  with respect to $\theta$ by $I_n(Z)$. 
A particularly useful result to calculate $I_n(Z)$ for the above class of Gaussian processes is given by \citet{Whittle53}:
\begin{equation} \label{eq:Whittle}
\lim_{n\to\infty}\frac{1}{n}I_n(Z)
= \frac{1}{4\pi} \int_{-\pi}^\pi\left( \frac{\frac{\partial}{\partial \theta}\phi_\theta(\omega)}{\phi_\theta(\omega)}\right)^2\, d\omega,\qquad n\to \infty, 
\end{equation}
where 
$\phi(\omega)=\sum_{j\in \Z}\E[Z_0Z_j]\e^{-ij\omega}, \omega\in [-\pi,\pi],$
is the spectral density of $Z$. 

Setting $\theta=\pi^2 \vt_2$, \eqref{eq:Whittle} cannot be directly applied to the process $Z=\bar U_k$, for $1\le k\le M-1$, since $\bar U_k$ arises from high-frequency increments of the continuous time process $U_k$. In this case, the spectral density $\Phi_k^\Delta$ of $\bar U_k$ hinges on $\Delta=1/N$  and therefore, even for large $N$, $I_N(\bar U_k)/N$ is not necessarily close to the asymptotic Fisher information defined in \eqref{eq:Whittle}.

To circumvent this difficulty, consider the $N$-th order Fourier approximation to $\Phi_k^\Delta$:
\begin{equation} \label{eq:spetralDens_approx}
\Phi_k ^{N,\Delta}(\omega)=\sum _{j=1-N}^{N-1}\E[\bar U_k(0)\bar U_k(j)] \e^{-ij\omega}\geq 0,\qquad \omega\in [-\pi,\pi].
\end{equation}
Lemma~\ref{lem:spec_properties}(i) verifies that $\Phi_k^{N,\Delta}$ is positive. Therefore, there exists a stationary Gaussian process $Y_k=(Y_k(j))_{j\in \Z}$ with spectral density $\Phi_k^{N,\Delta}$. Clearly,
$$\left(Y_k(j),\ldots,Y_k(j+N-1))\right)\overset{\mathcal D}=\left(\bar U_k(0) ,\ldots,\bar U_k(N-1)\right),\qquad j\in \N_0,$$  and $\left(Y_k(j),\ldots,Y_k(j+N-1))\right)$ is independent of $\left(Y_k(h),\ldots,Y_k(h+N-1))\right)$ whenever $|j-h|>2N$. Consequently, it is possible to extract $L$ independent copies of $\left(\bar U_k(0) ,\ldots,\bar U_k(N-1)\right)$ from a sample $(Y_k(0),\ldots ,Y_k({2NL-1}))$ for any $L\in \N$. Now, using the fact that a statistic never has larger information than the data from which it is constructed (cf. \cite[Theorem~I.7.2]{IbragimovHasminskii1981})  yields
\begin{equation} \label{eq:FI_bound1}
L \cdot I_N(\bar U_k) \leq  I_{2NL}(Y_k).
\end{equation}  
For fixed $\Delta=1/N$ we can now apply Whittle's formula \eqref{eq:Whittle} for $L\to\infty$: For each $\eps>0$ we can choose $L\in \N$ such that 
\begin{equation} \label{eq:FI_bound2}
I_{2NL}(Y_k)\leq 2 NL(1+\eps) \mathscr{I}_k ,
\end{equation}  
where
$$\mathscr{I}^{N,\Delta }_k:=\frac{1}{4\pi}\int_{-\pi}^\pi S^2(\omega)\,d\omega,\qquad S:=\frac{\partial}{\partial \vt_2}\log \Phi_k^{N,\Delta}.$$
By combining \eqref{eq:FI_bound1} and \eqref{eq:FI_bound2} we get $I_N(\bar U_k)\leq 2N\mathscr I _k$. Proving below that uniformly in $k=0,\dots, M-1$
\begin{equation}\label{eq:proofFisher}
\mathscr I^{N,\Delta }_k\lesssim M^2\Delta \log \frac{1}{M^2\Delta},
\end{equation}
we obtain 
$I_N(\bar{U}_k)\lesssim M^2 \log \frac{1}{M^2 \Delta}$
and the results follows by independence of the processes $\bar U_1,\ldots,\bar U_{M-1}$.

In order to verify \eqref{eq:proofFisher}, we only have to consider the integral over $[0,\pi]$ by symmetry. From Lemma \ref{lem:spec_properties} we can deduce for $\omega\geq k^2 \Delta$
$$S(\omega) \lesssim 
\begin{cases}
\frac{M\sqrt \Delta}{\sqrt \omega},&\omega\geq M^2 \Delta\\
1,&\omega\in [k^2\Delta,M^2\Delta]
\end{cases}
\quad\text{implying}\quad
\int_{k^2 \Delta}^\pi S^2(\omega)\,d\omega \lesssim M^2 \Delta \log \frac{1}{M^2\Delta}.
$$
For $\omega\leq k^2 \Delta$, Lemma \ref{lem:spec_properties} gives
$S(\omega)\lesssim (\frac{\omega^2}{k^4\Delta^2}+k^2 \e^{-\theta k^2})/(\frac{\omega^2}{k^4\Delta^2}+ \e^{-\theta k^2}).$
Since
\begin{align*}
\int_0^1 \frac{d\omega}{(\omega^2+\e^{-\theta k^2})^2} \leq \int_0^{\e^{-\theta k^2/2}} \frac{1}{\e^{-2\theta k^2}}\,d\omega + \int_{\e^{-\theta k^2/2}}^1 \frac{1}{\omega^4}\,d\omega\lesssim \exp \left(\frac{3}{2}\theta k^2\right),
\end{align*}
a substitution yields
\begin{align*}
\int_0^{k^2\Delta}S^2(\omega)\,d\omega &\lesssim k^2\Delta\int_0^1 
\left(\frac{{\omega^2}+k^2 \e^{-\theta k^2}}{{\omega^2}+ \e^{-\theta k^2}}\right)^2\,d\omega
%&\lesssim M^2 \Delta\left(1+ k^4 \e^{-2\theta k^2}\int_0^1 \frac{d\omega}{(\omega^2+\e^{-\theta k^2})^2} \right)
\lesssim M^2 \Delta.\qedhere
\end{align*}
\end{proof}

We can now conclude the main lower bound.

\begin{proof}[Proof of Theorem~\ref{thm:lower_bound_main}]

The proof of the lower bound relies on the fact that if $(P_\gamma)_{\gamma \in G}$ is a dominated family of distributions with a convex parameter space $G\subset \R$, then the Hellinger distance $\mathcal H$ can be bounded in terms of the Fisher Information $J$: Let $\nu$ be a dominating measure, $p(\cdot,\gamma)=dP_\gamma/d\nu$ and $g=\sqrt p$. 
Then, as shown in \cite[Theorem~I.7.6]{IbragimovHasminskii1981}, Jensen's inequality yields
\begin{align*}
\mathcal H^2(P_\gamma,P_{\gamma+h})&= \int (g(x,\gamma)-g(x,\gamma+h))^2\,\nu(dx) \leq h^2\int \int_0^1 \frac{\partial g}{\partial \gamma}(x,\gamma+sh)^2\,ds\,\nu(dx) \\
&=\frac{h^2}{4} \int_0^1 \int \left(\frac{\partial}{\partial \theta}\log p(x,\gamma+sh) \right)^2P_{\gamma+sh}(dx)\,ds =\frac{h^2}{4} \int_0^1 J(\gamma+sh)\,ds.
\end{align*}
Combining this bound of the Hellinger distance (in the setting of Theorem~\ref{thm:lower_bound_main}) with Theorem~2.2 by \citet{Tsybakov10}, it suffices that for each sampling regime there is a reparametrization $(\gamma_1,\gamma_2)$ of $(\sigma^2,\vt_2)$ such that the corresponding Fisher information satisfies $J_{M,N}(\gamma_2)\lesssim r_{M,N}^{-2}$ locally uniformly. Inspection of the proofs of Propositions \ref{prop:FisherCoeff} and \ref{prop:FisherCoeff_2}  shows that the bounds on the Fisher information are indeed locally uniform.

$(ii)$ \emph{Case $M/\sqrt N\gtrsim 1$.} For $L\in \N$ define the process $X^L$ via $X^L_t(y)=\sum_{\ell=1}^L u_\ell(t)e_\ell(y),\,t\geq0,\,y\in [0,1],$ and let $\mathcal{X}^L_{N,M}=\{X^L_{t_i}(y_k),\,i=0,\ldots,N-1,\,k=0,\ldots,M\}$ as well as $\mathcal{X}_{N,M}= \mathcal{X}^\infty_{N,M}$. Denoting the corresponding covariance matrices by $\Sigma^L_{N,M}$ and $\Sigma_{N,M}$ and using the result of \cite{Devroye19}, we can bound the total variation distance of the Gaussian distributions by $\mathrm{TV}(\Sigma_{N,M},\Sigma^L_{N,M})\leq \frac{3}{2}\Vert \Sigma_{N,M}^{-1}(\Sigma_{N,M}^{L}-\Sigma_{N,M}
)\Vert_F\leq \frac{3}{2}\Vert \Sigma_{N,M}^{-1}\Vert_F \Vert\Sigma_{N,M}^{L}-\Sigma_{N,M}\Vert_F$. Consequently, we can pick a sequence $L_{N,M} \to \infty$ such that $\mathcal{X}^{L_{N,M}}_{N,M}$ and $\mathcal{X}_{N,M}$ are statistically equivalent in the sense of Le Cam and it is sufficient to derive a lower bound for $\mathcal{X}^{L_{N,M}}_{N,M}$, or even $\{u_\ell(t_i),\,i\leq N,\ell\leq L_{N,M}\}$. Assuming $L_{N,M} \geq M$ without loss of generality, for this observation scheme Proposition \ref{prop:FisherCoeff} yields under the parametrization $(\sigma^2/\vt_2,\sigma^2)$:
\[
  J_{M,N}(\sigma^2)\lesssim N^{3/2}\wedge L_{N,M}^3=N^{3/2}=r_{N,M}^{-2}.
\]
\emph{Case $M/\sqrt N\to 0$.} For $b\in \Q \cap(0,1/2)$ write $b=p/q$ where $p\in \Z$ and $q \in \N$ such that
$y_k =\frac{pM+k(q-2p)}{q M},\,k \leq M,$
and consequently $\{y_k,\,k=0,\ldots,M\}$ is a subset of $\{z_k,\,k=1,\ldots, qM-1\}$ where $z_k=\frac{k}{qM}$. Now,  $ qM \sqrt\Delta   \to 0$ and since
$q^3M^3\log\left(\frac{1}{q^2M^2\Delta}\right)\lesssim M^3\log \left(\frac{1}{M^2\Delta}\right)$
Proposition~\ref{prop:FisherCoeff_2} implies under the parametrization $(\sigma^2/\sqrt{\vt_2},\vt_2)$:
\[
  J_{M,N}(\vt_2)\lesssim M^3\log(\frac{1}{M^2 \Delta})=r_{N,M}^{-2}.
\]

$(i)$ If $\min(M,N)$ remains finite and $M/\sqrt N \gtrsim 1$, then $N$ necessarily remains finite and the result follows from $(ii)$.  On the other hand, if $M/\sqrt N \to 0$, then $M$ must remain finite. Like in the proof of $(ii)$, extend the set of spatial locations to $\{z_k,\, k<qM\}$ and consider the corresponding processes $U_k,\,k=1,\ldots,qM-1$ from \eqref{eq:U_k_def}. A similar calculation as in the proof of Proposition \ref{prop:abs_cont} shows that for any $k<qM$, the laws of the independent continuous processes $\{U_k(t),\,t\leq 1\}$ are absolutely continuous for different parameter values $(\sigma^2,\vt_2)$ and $(\tilde \sigma^2,\tilde \vt_2)$ as long as $\sigma^2/\sqrt{\vt_2}=\tilde\sigma^2/\sqrt{\tilde\vt_2}$   and hence, consistent estimation of $(\sigma^2,\vt_2)$ based on continuous or discrete observations is impossible: Note that the continuous spectral density of $U_k$ is 
$f_k(u)=\frac{1}{2u^2}\sum_{\ell \in \mathcal I_k}h_{(\sigma^2,\vt_2)}\left(\frac{\ell}{\sqrt{ |u|}} \right),\,u\in \R$, where $h_{(\sigma^2,\vt_2)}$ is defined in the proof of Proposition \ref{prop:abs_cont}.
Now, a { Riemann sum midpoint approximation, cf. Lemma~\ref{lem:Riemann},} shows that 
\begin{align*}
f_k^+(u)&:=\frac{1}{2u^2} \sum_{\ell\geq 0}h_{(\sigma^2,\vt_2)}\Big(\frac{k+2M\ell}{\sqrt u} \Big) = \frac{1}{2u^2} \left( \frac{\sqrt u}{2M} \int_{(k-M)/\sqrt u}^\infty h_{(\sigma^2,\vt_2)}(z)\,dz+\O\Big( \frac{1}{\sqrt u} \Big) \right)\\
f_k^-(u)&:=\frac{1}{2u^2} \sum_{\ell\geq 0}h_{(\sigma^2,\vt_2)}\Big(\frac{2M-k+2M\ell}{\sqrt u} \Big) = \frac{1}{u^2} \left( \frac{\sqrt u}{4M} \int_{(M-k)/\sqrt u}^\infty h_{(\sigma^2,\vt_2)}(z)\,dz+\O\Big( \frac{1}{\sqrt u} \Big) \right).
\end{align*}
as $u\to \infty$. Since $h_{(\sigma^2,\vt_2)}$ is symmetric around 0 we obtain
\begin{align*}
f_k(u)=f_k^+(u)+f_k^-(u) =  \frac{1}{u^2} \left( \frac{\sqrt u}{2M} \int_0^\infty h_{(\sigma^2,\vt_2)}(z)\,dz+\O\left( \frac{1}{\sqrt u} \right) \right)
\end{align*}
from which equivalence follows as in Proposition \ref{prop:abs_cont}.
\end{proof}

\subsection{Proofs for Section 2}
\begin{proof}[{Proof of Proposition \ref{prop:Ito_representation}}]
Due to \eqref{eq:cov} and the trigonometric identity 
\begin{equation} \label{eq:trig_id}
\sin(\alpha)\sin(\beta) =\frac{1}{2}\left(\cos(\alpha-\beta)-\cos(\alpha+\beta)\right)
\end{equation}
we have 
\begin{align*}
\Cov(X_t(x),X_t(y))
&=\frac{\sigma^2}{2\pi^2 \vt_2}\e^{-\frac{\kappa}{2}(x+y)} \sum_{\ell\geq 1}\frac{1}{\ell^2+\Gamma/\pi^2}(\cos(\pi \ell (y-x))-\cos(\pi \ell (x+y))).
\end{align*} 
The claimed formulas now follow by inserting the closed expressions
\begin{equation} \label{eq:Fourier_expl}
\sum_{\ell\geq 1}\frac{1}{\ell^2+\beta}\cos(\pi \ell x)=
\begin{cases}
-\frac{\pi\cos(\pi \sqrt{| \beta|}(x-1))}{2\sqrt{ |\beta|} \sin(\pi \sqrt {|\beta|})}+\frac{1}{2|\beta|}, &-1<\beta < 0\\
\frac{\pi^2(x-1)^2}{4}-\frac{\pi^2}{12},&\beta = 0\\
\frac{\pi\cosh(\pi \sqrt \beta(x-1))}{2\sqrt \beta \sinh(\pi \sqrt \beta)}-\frac{1}{2\beta}, &\beta > 0 \\
\end{cases}
\end{equation} 
for  $x\in [0,1]$ and again applying \eqref{eq:trig_id} and
$\sinh(\alpha)\sinh(\beta)=\frac{1}{2}(\cosh(\alpha+\beta)-\cosh(\alpha -\beta)),$
respectively. To prove the second statement we use the ansatz $Z(x)=u(x)B({v(x)})$, $u,v$ positive and $v$
non-decreasing, which is the  general form of a Gaussian Markov process, cf.~\cite{Neveu68}. Comparison of covariance functions yields explicit expressions for $u$ and $v$.
Further, $u(x)B({v(x)})\stackrel{(d)}{=}u(x)\int_{0}^{x}\sqrt{v'(z)}\,dB({z})$ for $v(0)=0$ and the claimed semi-martingale representation follows from It\^o's formula.
\end{proof}

\begin{proof}[Proof of Proposition \ref{prop:abs_cont}]
The necessity of the conditions on the parameters follow from the fact that $(i)$ the parameter ${\sigma^2}/{\sqrt{\vt_2}}\e^{-\kappa x_0}$ may be consistently estimated using time increments, see  \cite{Bibinger18}, and $(ii)$ the parameters $\frac{\sigma^2}{\vt_2}$ and $\kappa$ may be consistently estimated by computing the quadratic variation of the process $x\mapsto X_t(x)$ on two different sub-intervals of $[0,1]$ in view of Proposition \ref{prop:Ito_representation}.

It remains to prove sufficiency of the conditions on the parameters: 

$(i)$ is a simple consequence of 
\cite[Proposition 1]{KoskiLoges85}: Set $\lambda_\ell= \vt_2(\pi^2 \ell^2 +\Gamma)$ and $\tilde \lambda_\ell= \vt_2(\pi^2 \ell^2 +\tilde \Gamma)$ where $\Gamma= \frac{\vt_1^2}{4\vt_2^2}-\frac{\vt_0}{\vt_2}$ and $\tilde\Gamma = \frac{\vt_1^2}{4\vt_2^2}-\frac{\tilde \vt_0}{\vt_2}$. Then, absolute continuity follows from $\sum_{\ell\geq 1}\frac{(\lambda_\ell-\tilde \lambda_\ell)^2}{\lambda_\ell}<\infty.$ Thanks to $(i)$ and due to the one to one correspondence between $\Gamma$ and $\vt_0$  we may assume $\Gamma = \tilde \Gamma=0$ for the remainder of the proof.

$(ii)$ follows from the fact that  $\Cov (X_{t_0} (x),X_{t_0}(y))$ only depends on $\left(\frac{\sigma^2}{{\vt_2}},\kappa\right)$ in view of the Gaussianity of $X$.

For $(iii)$ note that $t\mapsto X_t(x_0)$ is a stationary Gaussian process with covariance function
$$\rho (t)= \sigma^2\sum_{k\geq 1} \frac{\e^{-\lambda_k t}}{2 \lambda_k} e_k^2(x_0). $$  Let 
$$f_{(\sigma^2,\vt_2)} (u)= \frac{1}{2\pi} \int \e^{-iut}\rho(|t|)\,dt=\frac{1}{\pi} \int_0^\infty  \cos(ut)\rho(t)\,dt=\frac{\sigma^2}{2\pi}\sum_{\ell\geq 1} \frac{e_\ell^2(x_0)}{\lambda_\ell^2 +u^2}$$ be the spectral density of $t\mapsto X_t (x_0)$. By Theorem 17 and its preceding discussion in \cite{Ibragimov78} it suffices to show 
\begin{equation*} 
\exists r>1:\qquad \lim _{u\to \infty}u^r f_{(\sigma^2,\vt_2)} (u)\in (0,\infty)\qquad\text{and}\qquad
\frac{f_{(\sigma^2,\vt_2)} -f_{(\tilde\sigma^2,\tilde\vt_2)} }{f_{(\sigma^2,\vt_2)} } \in L^2(\R).
\end{equation*}
To prove these statements, we may assume $\kappa=0$ without loss of generality.
Set 
$h_{(\sigma^2,\vt_2)}(z)=\frac{\sigma^2}{\pi(\pi^4\vt_2^2z^4+1)},\, z\in \R.$
By Lemma \ref{lem:Riemann} (ii) we have for $u\to \infty$
\begin{align*}
f_{(\sigma^2,\vt_2)} (u)&=\frac{1}{u^2}\sum_{\ell\geq 1} h_{(\sigma^2,\vt_2)}\left(\frac{\ell}{\sqrt{u}}\right)\sin^2(\pi \ell x_0)
=\frac{1}{u^2}\left(\frac{\sqrt{u}}{2}\int_0^\infty h_{(\sigma^2,\vt_2)}(z)\,dz +\O\left(\frac{1}{\sqrt{u}}\right) \right),
\end{align*}
which proves the first condition. Now, if $\sigma^2/\sqrt{\vt_2}=\tilde \sigma^2/\sqrt{\tilde\vt_2}$ then clearly
$\int_0^\infty h_{(\sigma^2,\vt_2)}(z)\,dz=\int_0^\infty h_{(\tilde\sigma^2,\tilde \vt_2)}(z)\,dz$ and therefore, the second condition follows:
\[\frac{f_{(\sigma^2,\vt_2)} (u)-f_{(\tilde\sigma^2,\tilde\vt_2)} (u)}{f_{(\sigma^2,\vt_2)} (u)}=\O\left(\frac{1}{u}\right),\quad u \to \infty.\qedhere
\]
\end{proof}

\appendix

\section{Remaining proofs and auxiliary results}\label{sec:AuxProofs}

\subsection{Covariances of double increments}\label{sec:AuxDoubleIncr}
The following three lemmas are used to calculate the asymptotic variance of $\mathbb V$.  
%The first one identifies the relevant terms in the covariance structure and shows independence of $\Gamma$. 
Recall the definition of $\rsti{ik}$ from \eqref{eq:DTilde}.

%%%%%%%%%%%%%%%%%%%%%%%%%%%%%%%%%%%%%%%%%%%%%%%%%%%%%%%%%%%%%%%%%%%%
\begin{lem} \label{lem:var_basic}
Let $b\in(0,1/2)$. For $J\geq 1$ define
\[
F_{J,\Delta}(z)=\sum_{\ell\geq1}\frac{2\e^{-\pi^{2}\vt_{2}J\ell^{2}\Delta}-\e^{-\pi^{2}\vt_{2}(J+1)\ell^{2}\Delta}-\e^{-\pi^{2}\vt_{2}(J-1)\ell^{2}\Delta}}{2\pi^{2}\vt_{2}\ell^{2}}\cos(\pi \ell z)
\]
and $F_{0,\Delta}=F_{\vt_2}(\cdot\,,\Delta)$. Then, for $J=|i-j|$,
\begin{align*}
\Cov(\rsti{ik},\rsti{jl})
=&- \sigma^{2}\e^{-\kappa\delta/2}\cdot\begin{cases}
2D_{\delta}F_{J,\Delta}(0) & l=k\\
D^2_{\delta}F_{J,\Delta}(y_l-y_{k+1}) & l>k
\end{cases}+\O\left(\frac{\sqrt{\Delta}\delta^{2}}{(J+1)^{3/2}}\right).
\end{align*}  
\end{lem} 

\begin{proof}
It immediately follows from the covariance structure 
$
\Cov(u_\ell(s),u_\ell(t))
%=\e^{-\lambda_\ell (s+t)}\left(\V(u_\ell(0))+\sigma^2 \int_0^{s\wedge t}\e^{2\lambda_\ell r}\,dr\right)
=\frac{\sigma^2}{2\lambda_\ell}\e^{-\lambda_\ell|t-s|}$, $s,t\geq 0$,
of the coefficient processes that
\begin{align*}
&\Cov(\sti{ik},\,\sti{jl})\\
&\quad=\sigma^2 \sum_{\ell\geq 1}(e_\ell(y_{k+1})-e_\ell(y_k))(e_\ell(y_{l+1})-e_\ell(y_l))\cdot
\begin{cases}
\frac{1-\e^{-\lambda_{\ell}\Delta}}{\lambda_{\ell}},&J=0,\\
\frac{2\e^{-\lambda_{\ell}J\Delta}-\e^{-\lambda_{\ell}(J+1)\Delta}-\e^{-\lambda_{\ell}(J-1)\Delta}}{2\lambda_{\ell}},&J\geq 1. 
\end{cases}
\end{align*}
\textit{Step 1.} We show negligibilty of $\Gamma$. From the first step
of the last proof we already know that 
\begin{align*}
\Cov(\sti{ik},\sti{il})=\sigma^{2}\sum_{\ell\geq1}\frac{1-\e^{-\pi^{2}\vt_{2}\ell^{2}\Delta}}{\pi^{2}\vt_{2}\ell^{2}}(e_{\ell}(y_{k+1})-e_{\ell}(y_k))(e_{\ell}(y_{l+1})-e_{\ell}(y_l))+\O\left(\sqrt{\Delta}\delta^{2}\right).
\end{align*}
For $J\geq1$ we will show now that
\begin{multline*}
\Cov(\sti{ik},\sti{jl})=\sigma^{2}\sum_{\ell \geq1}\frac{2\e^{-\pi^{2}\vt_{2}\ell^{2}J\Delta}-\e^{-\pi^{2}\vt_{2}\ell^{2}(J+1)\Delta}-\e^{-\pi^{2}\vt_{2}\ell^{2}(J-1)\Delta}}{2\pi^{2}\vt_{2}\ell^{2}}\\ \cdot(e_{\ell}(y_{k+1})-e_{\ell}(y_{k}))(e_{\ell}(y_{l+1})-e_{\ell}(y_{l}))+\O\left(\frac{\sqrt{\Delta}\delta^{2}}{(J+1)^{3/2}}\right).
\end{multline*}
If $J=1$ this directly follows from the case $J=0$ since
\begin{equation} \label{eq:J0to1}
\frac{2\e^{-\lambda_{\ell}\Delta}-\e^{-2\lambda_{\ell}\Delta}-1}{2\lambda_{\ell}}=\frac{1-\e^{-2\lambda_{\ell}\Delta}}{2\lambda_{\ell}}-\frac{1-\e^{-\lambda_{\ell}\Delta}}{\lambda_{\ell}}.
\end{equation}
For $J\geq 2$ define
$g_{J}(x)=\frac{2\e^{-Jx}-\e^{-(J+1)x}-\e^{-(J-1)x}}{2x}.$
A first order Taylor approximation of $g_J$ gives
\begin{align*}
\Cov(\sti{ik},\sti{jl}) & =\Delta\sum_{\ell\geq1}g_{J}(\lambda_{\ell}\Delta)(e_{\ell}(y_{k+1})-e_{\ell}(y_{k}))(e_{\ell}(y_{l+1})-e_{\ell}(y_{l}))\\
 & =\Delta\sum_{\ell\geq1}g_{J}(\pi^{2}\vt_{2}\ell^{2}\Delta)(e_{\ell}(y_{k+1})-e_{\ell}(y_{k}))(e_{\ell}(y_{l+1})-e_{\ell}(y_{l}))+R,
\end{align*}
where 
$
R\lesssim\Delta^{2}\sum_{\ell\geq1}g_{J}'(\vt_{2}(\pi^{2}\ell^{2}+\xi_{\ell})\Delta)\ell^{2}\delta^{2}
$
for some $|\xi_\ell|\leq |\Gamma|$. It can be shown easily that
$
g_{J}'(x)\lesssim\e^{-(J-1)x/2}.
$
Therefore, for some $\omega >0$ and by regarding $R$ as a Riemann sum with lag $\sqrt{(J-1)\Delta}$,
\begin{align*}
R\lesssim \Delta^{2}\sum_{\ell\geq1}\e^{-\omega(J-1)\ell^2\Delta}\ell^{2}\delta^{2}
\lesssim \frac{\sqrt{\Delta}\delta^{2}}{(J-1)^{3/2}}\lesssim \frac{\sqrt{\Delta}\delta^{2}}{(J+1)^{3/2}}.
\end{align*}
\textit{Step 2.} By Step 1 we may assume $\lambda_\ell = \pi^2 \vt_2 \ell^2$. By \eqref{eq:eig_single} we have
\begin{align*}
\Cov(\rsti{ik},\rsti{jk})  =&-2\sigma^{2}g(\delta)D_{\delta}F_{J,\Delta}(0)+\sigma^{2}F_{J,\Delta}(0)D_{\delta}^2g(0) -\sigma^{2}D_{\delta}^2\left(g(\cdot)F_{J,\Delta}(2y_k+\cdot) \right)(0)
\end{align*}
and by \eqref{eq:eig_mult} for $l>k$ 
\begin{align*}
\Cov(\rsti{ik},\rsti{jl}) =&-\sigma^{2}g(\delta)D^2_{\delta}F_{J,\Delta}(y_l-y_{k+1})
+\sigma^{2}F_{J,\Delta}(y_l-y_{k})D_{\delta}^2g(0)\\&-\sigma^{2}D_{\delta}^2\left(g(\cdot)F_{J,\Delta}(y_l+y_k+\cdot)\right)(0).
\end{align*}
Hence, as in previous Lemmas it is sufficient to establish 
\[
F_{J,\Delta}(0),\,F_{J,\Delta}(z),\,F_{J,\Delta}'(z)\,F_{J,\Delta}''(z)\lesssim\frac{\sqrt{\Delta}}{J^{3/2}},\quad z \in[2b,2(1-b)].
\]
For $J=0$ this was already proven in Proposition \ref{prop:mean_double}.
The case $J=1$ follows from the case $J=0$ since \eqref{eq:J0to1} shows
\begin{align} \label{eq:FvsF1}
F_{1,\Delta}(z)=\frac{1}{2}F_{2\Delta}(z)-F_{\Delta}(z).
\end{align}
For $J\geq 2$ we have
\begin{align} 
2\e^{-\lambda_{\ell}J\Delta}-\e^{-\lambda_{\ell}(J+1)\Delta}-\e^{-\lambda_{\ell}(J-1)\Delta} 
%&=\e^{-\lambda_{\ell}(J-1)\Delta}(2\e^{-\lambda_{\ell}\Delta}-\e^{-\lambda_{\ell}2\Delta}-1) \nonumber\\
&\lesssim\e^{-\lambda_{\ell}(J-1)\Delta}(\lambda_{\ell}\Delta)^{2}, \label{eq:exp_order2diff}
\end{align}
and therefore, again using a Riemann sum approximation with lag $\sqrt{(J-1)\Delta}$,
\begin{align*}
F_{J,\Delta}(z) & \lesssim F_{J,\Delta}(0)
 \lesssim\sum_{\ell\geq1}\lambda_{\ell}\Delta^{2}\e^{-\lambda_{\ell}(J-1)\Delta}
  =\O\left(\frac{\sqrt{\Delta}}{(J-1)^{3/2}}\right).
\end{align*}
The bound on the first derivative is provided by Lemma \ref{lem:trig_series_fundamental},
\begin{align*}
F_{J,\Delta}'(z) & \lesssim \sum_{\ell\geq1}\frac{2\e^{-\lambda_{\ell}J\Delta}-\e^{-\lambda_{\ell}(J+1)\Delta}-\e^{-\lambda_{\ell}(J-1)\Delta}}{2\lambda_{\ell}}\ell\sin(\pi \ell z)\\
 & \lesssim\sup_{\ell}\left|\frac{2\e^{-\lambda_{\ell}J\Delta}-\e^{-\lambda_{\ell}(J+1)\Delta}-\e^{-\lambda_{\ell}(J-1)\Delta}}{2\lambda_{\ell}}\ell\right|\frac{1}{z\wedge(2-z)}
  \lesssim\sup_{\ell}\left|\lambda_{\ell}\Delta^{2}\e^{-\lambda_{\ell}J\Delta}\ell\right|
  \lesssim\frac{\sqrt{\Delta}}{J^{3/2}}.
\end{align*}
Finally, to bound $F_{J,\Delta}''$ we define $h_J(z)=2\e^{-Jz^2}-\e^{-(J+1)z^2}-\e^{-(J-1)z^2}$. Clearly, $h_J(0)=0$ and
\begin{align*}
\frac{d}{dz}h_J(z)
&=-2(J-1)z\e^{-(J-1)z^2}\underbrace{(2\e^{-z^2}-\e^{-2z^2}-1)}_{\lesssim z^4}-\e^{-(J-1)z^2}\underbrace{(4z\e^{-z^2}-4z\e^{-2z^2})}_{\lesssim z^3}
\lesssim \frac{1}{J^{3/2}},
\end{align*}
i.e.~$\Vert h_J'\Vert_\infty\lesssim J^{-3/2}$. In view of Lemma \ref{lem:trig_integral} this shows
\begin{align*}
F_{J,\Delta}''(z) & \lesssim\sum_{\ell\geq1}\frac{2\e^{-\lambda_{\ell}J\Delta}-\e^{-\lambda_{\ell}(J+1)\Delta}-\e^{-\lambda_{\ell}(J-1)\Delta}}{2\lambda_{\ell}}\ell^{2}\cos(\pi \ell z)\\
&\lesssim \sum_{\ell\geq 1} h_J(\sqrt{ \lambda_\ell \Delta})\cos(\pi \ell z)  =\O\left(\frac{1}{\left(z\wedge(2-z)\right)^{2}}\frac{\sqrt{\Delta}}{J^{3/2}}\right).\qedhere
\end{align*}
\end{proof}
%%%%%%%%%%%%%%%%%%%%%%%%%%%%%%%%%%%%%%%%%%%%%%%%%%%%%%%%%%%%%%%%%%%%

%The following Lemma is useful for calculating the asymptotic variance in case $\delta/\sqrt{\Delta}\to 0.$

%%%%%%%%%%%%%%%%%%%%%%%%%%%%%%%%%%%%%%%%%%%%%%%%%%%%%%%%%%%%%%%%%%%%
\begin{lem} \label{lem:var_dbyD0} 
For $J\in \N_0$ and $z\in(\text{0,2})$ it holds that
\begin{enumerate}[(i)]
\item
$F_{J,\Delta}(0)-F_{J,\Delta}(\delta)=\delta\frac{1}{2\vt_{2}}\mathbf{1}_{\{J=0\}}-\delta\frac{1}{4\vt_{2}}\mathbf{1}_{\{J=1\}}+\O\left(\frac{\delta^{2}}{(J+1)^{5/2}\sqrt{\Delta}}\right)$
\item
$2F_{J,\Delta}(z)-F_{J,\Delta}(z+\delta)-F_{J,\Delta}(z-\delta)=\O\left(\frac{\delta^{2}}{(J+1)^{2}}\left(\frac{1}{\sqrt{\Delta}}\wedge\frac{1}{z\wedge(2-z)}\right)\right).$
\end{enumerate}
\end{lem}

\begin{proof}
$(i)$
The validity for the case $J=0$ follows from the proof of Proposition \ref{prop:mean_double} (ii), the case $J=1$ follows from \eqref{eq:FvsF1}. For $J\geq2$ we have by Taylor's theorem
\[
F_{J,\Delta}(0)-F_{J,\Delta}(\delta)=-\delta F_{J,\Delta}'(0)-\frac{\delta^{2}}{2}F_{J,\Delta}''(\xi)
\]
for some $\xi\in[0,\delta].$ Now, the claim is proved by inserting $F_{J,\Delta}'(0)=0$ and noting due to \eqref{eq:exp_order2diff}: 
\begin{align*}
\left\Vert F_{J,\Delta}''\right\Vert _{\infty} & \lesssim\sum_{\ell \geq1}\left(2\e^{-\lambda_{\ell }J\Delta}-\e^{-\lambda_{\ell}(J+1)\Delta}-\e^{-\lambda_{\ell}(J-1)\Delta}\right)
  \lesssim\sum_{\ell\geq1}\lambda_{\ell}^{2}\Delta^{2}\e^{-\lambda_{\ell}(J-1)\Delta} \lesssim\frac{1}{J^{5/2}\sqrt{\Delta}}.
\end{align*}
$(ii)$
As in previous Lemmas it suffices to establish 
\[
F_{J,\Delta}''(z)\lesssim \frac{1}{(J+1)^2}\left( \frac{1}{\sqrt{\Delta}}\wedge\frac{1}{z\wedge(2-z)}\right).
\]
For the case $J=0$ we employ the representation $F_{\Delta}=H_{\Delta}+G_{\Delta}$ from Proposition \ref{prop:mean_double} . The validity of the bound on $H_{\Delta}''$
follows from 
$H_\Delta''(z)
%\lesssim \frac{\exp \left(-\frac{\pi}{\sqrt{\vt_2 \Delta}}(y\wedge (2-y)) \right)}{\sqrt{\Delta}}
\lesssim \frac{1}{\sqrt \Delta}\wedge \frac{1}{z\wedge (2-z)}.$
The bound on $G_\Delta''(z)$ follows from $\left\Vert G_{\Delta}''\right\Vert _{\infty}\lesssim 1/\sqrt{\Delta}$
and 
$
G_{\Delta}''(z)  \lesssim\sup_{\ell}\left|\frac{1-\e^{-\lambda_\ell \Delta}(1+\lambda_\ell\Delta)}{1+\lambda_\ell\Delta}\right|\frac{1}{z\wedge(2-z)}
 \lesssim\frac{1}{z\wedge(2-z)},
$
see Lemma \ref{lem:trig_series_fundamental}. The case $J=1$ follows from the case $J=0$, see \eqref{eq:FvsF1}. For $J\geq2$ we proceed in the same way: In the proof of $(i)$ it was shown that $\Vert F_{\Delta,J}''\Vert _{\infty}\lesssim\frac{1}{J^{5/2}\sqrt{\Delta}} \lesssim\frac{1}{J^{2}\sqrt{\Delta}}$.
Finally, by Lemma \ref{lem:trig_series_fundamental},
\begin{align*}
F_{J,\Delta}''(z) & \lesssim\sup_{\ell}\left|2\e^{-\lambda_{\ell}J\Delta}-\e^{-\lambda_{\ell}(J+1)\Delta}-\e^{-\lambda_{\ell}(J-1)\Delta}\right|\frac{1}{z\wedge(2-z)}\\
 & \lesssim\sup_{\ell}\left|(\lambda_{\ell}\Delta)^{2}\e^{-\lambda_{\ell}(J-1)\Delta}\right|\frac{1}{z\wedge(2-z)} \lesssim\frac{1}{(J+1)^{2}}\frac{1}{z\wedge(2-z)}.\qedhere
\end{align*}
\end{proof}
%%%%%%%%%%%%%%%%%%%%%%%%%%%%%%%%%%%%%%%%%%%%%%%%%%%%%%%%%%%%%%%%%%%%

%The following Lemma is useful for calculating the asymptotic variance in case $\delta /\sqrt{\Delta}\to \infty.$

%%%%%%%%%%%%%%%%%%%%%%%%%%%%%%%%%%%%%%%%%%%%%%%%%%%%%%%%%%%%%%%%%%%%
\begin{lem} \label{lem:var_dbyDinf} 
For $J\in\N_{0}$ and $z\in(0,2)$ we have
\begin{enumerate}[(i)]
\item $\displaystyle
F_{J,\Delta}(0)-F_{J,\Delta}(\delta)=\begin{cases}
\frac{\sqrt{\Delta}}{\sqrt{\vt_{2}\pi}}+\O\left(\frac{\Delta^{3/2}}{\delta^{2}}\right), & J=0,\\
\frac{\sqrt{\Delta}}{2\sqrt{\pi\vt_{2}}}\left(\sqrt{J-1}+\sqrt{J+1}-2\sqrt{J}\right)+\O\left(\Delta^{3/2}+\frac{\Delta}{(J+1)\delta}\right), & J\geq1,
\end{cases}
$
\item $\displaystyle
%\begin{multline*}
2F_{J,\Delta}(\delta)-F_{J,\Delta}(0)-F_{J,\Delta}(2\delta)$\\ $\hspace*{1cm}\displaystyle
=\begin{cases}
-\frac{\sqrt{\Delta}}{\sqrt{\vt_{2}\pi}}+\O\left(\frac{\Delta^{3/2}}{\delta^{2}}\right), & J=0,\\
-\frac{\sqrt{\Delta}}{2\sqrt{\pi\vt_{2}}}\left(\sqrt{J-1}+\sqrt{J+1}-2\sqrt{J}\right)+\O\left(\Delta^{3/2}+\frac{\Delta}{(J+1)\delta}\right), & J\geq1,
\end{cases}
$
\item $\displaystyle
2F_{J,\Delta}(z)-F_{J,\Delta}(z-\delta)-F_{J,\Delta}(z+\delta)=\O\left(\frac{\Delta}{J+1}\frac{1}{z\wedge(2-z)}\right).
$
\end{enumerate}
\end{lem}

\begin{proof}
$(iii)$ It is sufficient to show 
\begin{equation} \label{eq:FJDeltaz_bound}
F_{J,\Delta}(z)=\O\left(\frac{\Delta}{J+1}\frac{1}{z\wedge(2-z)}\right) 
\end{equation}
for $J\in\N_{0}$ and $z\in(0,2)$: If $J=0$, Lemma
\ref{lem:trig_series_fundamental} gives 
\begin{align*}
F_{\Delta}(z)
 & \lesssim\sup_{\ell\geq 1}\left|\frac{1-\e^{-\lambda_{\ell}\Delta}}{\lambda_{\ell}}\right|\frac{1}{z\wedge(2-z)}
 \lesssim\frac{\Delta}{z\wedge(2-z)}.
\end{align*}
By (\ref{eq:FvsF1}) this bound is also valid for $F_{1,\Delta}(z)$.
For $J\geq2$ the same method gives 
\begin{align*}
F_{J,\Delta}(z)
 & \lesssim\sup_{\ell\geq 1}\left|\frac{2\e^{-\lambda_{\ell}J\Delta}-\e^{-\lambda_{\ell}(J+1)\Delta}-\e^{-\lambda_{\ell}(J-1)\Delta}}{\lambda_{\ell}}\right|\frac{1}{z\wedge(2-z)}\\
 & \lesssim\sup_{\ell\geq 1}\left|\lambda_{\ell}\Delta^{2}\e^{-\lambda_{\ell}J\Delta}\right|\frac{1}{z\wedge(2-z)} \lesssim\frac{\Delta}{J}\frac{1}{z\wedge(2-z)},
\end{align*}
where we have used \eqref{eq:exp_order2diff}.\\
$(i)$ The case $J=0$ was already shown in the proof of Proposition \ref{prop:mean_double}. For $J\geq1$ we prove  
\[
F_{J,\Delta}(0)=\frac{\sqrt{\Delta}}{2\sqrt{\pi\vt_{2}}}\left(\sqrt{J-1}+\sqrt{J+1}-2\sqrt{J}\right)+\O(\Delta^{3/2}),
\]
then $(ii)$ follows in view of \eqref{eq:FJDeltaz_bound}: If $J=1$ we use (\ref{eq:integral_at_0}) to calculate
\begin{align*}
F_{1,\Delta}(0) & =\frac{1}{2}F_{2\Delta}(0)-F_{\Delta}(0)
  =\frac{1}{2}\left(\frac{\sqrt{2\Delta}}{\sqrt{\pi\vt_{2}}}-\Delta\right)-\left(\frac{\sqrt{\Delta}}{\sqrt{\pi\vt_{2}}}-\frac{\Delta}{2}\right)+\O\left(\Delta^{3/2}\right)\\
 & =\frac{\sqrt{\Delta}}{2\sqrt{\pi\vt_{2}}}\left(\sqrt{2}-2\right)+\O\left(\Delta^{3/2}\right).
\end{align*}
For $J\geq2$ define 
$
g_{J}(z)=\frac{2\e^{-J\pi^{2}\vt_{2}z^{2}}-\e^{-(J+1)\pi^{2}\vt_{2}z^{2}}-\e^{-(J-1)\pi^{2}\vt_{2}z^{2}}}{2\pi^{2}\vt_{2}z^{2}}.
$
%Since 
%\begin{multline*}
%\int_{0}^{t}\frac{2\e^{-Jx^{2}}-\e^{-(J+1)x^{2}}-\e^{-(J-1)x^{2}}}{x^{2}}dx\\
%=\frac{-2\e^{-Jt^{2}}+\e^{-(J+1)t^{2}}+\e^{-(J-1)t^{2}}}{t}\\
%+2\sqrt{J-1}\psi(\sqrt{J-1}t)+2\sqrt{J+1}\psi(\sqrt{J+1}t)-4\sqrt{J}\psi(\sqrt{J}t)
%\end{multline*}
%where $\psi(t)=\int_{0}^{t}\e^{-s^{2}}ds$ and $\psi(\infty)=\sqrt{\pi}/2$
Then,
\begin{align*}
\int_{0}^{\infty}g_{J}(z)\,dz 
%& =\frac{1}{2\pi\sqrt{\vt_{2}}}\int_{0}^{\infty}\frac{2\e^{-Jz^{2}}-\e^{-(J+1)z^{2}}-\e^{-(J-1)z^{2}}}{z^{2}}\,dz\\
 & =\frac{1}{2\sqrt{\pi\vt_{2}}}\left(\sqrt{J-1}+\sqrt{J+1}-2\sqrt{J}\right)
\end{align*}
and since $g_{J}(0)=0$ we have by Lemma \ref{lem:Riemann} 
\begin{align*}
F_{J,\Delta}(0) 
  {=\Delta\sum_{\ell\geq1}g_{J}(\ell\sqrt{\Delta})}
 & =\sqrt{\Delta}\int_{0}^{\infty}g_{J}(z)\,dz+\O\left(\Delta^{3/2}\right)\\
 & =\frac{\sqrt{\Delta}}{2\sqrt{\pi\vt_{2}}}\left(\sqrt{J-1}+\sqrt{J+1}-2\sqrt{J}\right)+\O(\Delta^{3/2}).
\end{align*}
Finally, $(ii)$ is a direct consequence of $(i)$.
\end{proof}

\subsection{Auxiliary results for the lower bounds}
For the proofs of Propositions~\ref{prop:FisherCoeff} and \ref{prop:FisherCoeff_2} we require the following auxiliary lemmas.
\begin{lem} \label{lem:FisherOU}
Consider a discrete sample $(u(i\Delta),\:i=0,\ldots,N)$ of an Ornstein-Uhlenbeck
process given by 
\[
du({t})=-a\mu u(t)\,dt+\nu\sqrt{\mu}\,dB_{t},\quad u(0)\sim\mathcal{N}\left(0,\,\frac{\nu^{2}}{2a}\right)
\]
and assume $\Delta=1/N$. Then, the Fisher information $I=I_N$ for
the parameter $(\mu,\nu^{2})$ is given by 
\[
I_{11}=\frac{a^{2}\Delta(\e^{-4\mu a \Delta}+\e^{-2\mu a \Delta})}{(1-\e^{-2\mu a\Delta})^{2}},\qquad I_{12}=\frac{a\e^{-2\mu a\Delta}}{\nu^{2}(1-\e^{-2\mu a\Delta})},\qquad I_{22}=\frac{N+1}{2\nu^{4}}.
\]
\end{lem}

\begin{proof}
By the Markov property of $u$, the log-likelihood function of $(\mu,\nu^{2})$
is given by 
\[
\ell(\mu,\nu^{2})=\log\pi_{0}(u(0))+\sum_{i=0}^{N-1}\log p_{\Delta}(u(i\Delta),u((i+1)\Delta)),
\]
 where 
$
p_{t}(x,y)  =\frac{1}{\sqrt{\pi\nu^{2}(1-\e^{-2\mu a t})/a}}\exp\left(-\frac{(y-x\e^{-\mu a t})^{2}}{\nu^{2}(1-\e^{-2\mu a t})/a}\right)
$
is the transition density of $u$ and $\pi_{0}$ is the density of
the initial distribution $\mathcal{N}\left(0,\,\frac{\nu^{2}}{2 a}\right)$.
By stationarity of $u$, the Fisher information simplifies to 
\begin{align*}
I & =-\E\left(D^{2}\ell(\mu,\nu^{2})\right)=-\E\left(D^{2}\log\pi_{0}(u(0))\right)-N\E\left(D^{2}\log p_{\Delta}(u(0),u(\Delta))\right),
\end{align*}
where we write $D^2g$ for the Hessian of a function $g$.
This expression can be computed explicitly, yielding the claimed formulas.
\end{proof}

%To investigate the spectral density of the processes $\bar U_k$ from \eqref{eq:barU_k_def}, the following auxiliary lemma will be necessary:

\begin{lem} \label{lem:g_properties}
The function $g\colon[0,\infty)\times [-\pi,\pi]\to \R$ defined by $$g(x,\omega)=\frac{2x^2-\sinh(x^2)\cosh(x^2)+\cos(\omega)(\sinh(x^2)-2x^2\cosh(x^2))}{x^2(\cosh(x^2)-\cos(\omega))^2}(1-\cos(\omega))$$ satisfies
\begin{enumerate}[(i)]
\item $\int_0^\infty g(x,\omega)\,dx=0,$ for all $\omega \in [-\pi,\pi],$ 
\item $\sup_{|\omega|\leq \pi}\Vert \frac{\partial}{\partial x} g(\cdot,\omega) \Vert_{L^1}<\infty.$
\item  $|g(x,\omega)| \lesssim \frac{1+x^2}{x^4}\omega^2$ uniformly in $\omega\in [-\pi,\pi],\,x>0$.
\end{enumerate}
\end{lem}

\begin{proof}
$(i)$ follows from the fact that 
$$G(x,\omega):=\frac{\sinh(x^2)(1-\cos(\omega))}{x(\cosh(x^2)-\cos(\omega))},\qquad x>0,\, \omega\in [-\pi,\pi],$$
is a primitive of $x\mapsto g(x,\omega)$ and since $\lim_{x\to \infty}G(x,\omega)=\lim_{x\to 0}G(x,\omega)=0$ for all $\omega \in [-\pi,\pi]$.

$(ii)$ can be shown by writing $G(\cdot,\omega)$ as a sum of monotonic functions and noting that for a monotonic function  $g:\R_+\to \R$ it holds that $\Vert g' \Vert_{L^1} = |\lim_{x\to \infty}g(x)-\lim_{x\to 0}g(x)|$.

Finally, $(iii)$ follows by direct calculations.
\end{proof}

%The following Lemma analyzes the $N$-th order Fourier approximation to the spectral density of the processes $\{\bar U_k(j),\,j\in \N_0\}$ for $k=1,\ldots,M-1$.

\begin{lem} \label{lem:spec_properties}
Consider the parametrization of Proposition~\ref{prop:FisherCoeff_2}  and the function $ \Phi_k^{N,\Delta}$ from \eqref{eq:spetralDens_approx}.
If $M\sqrt{\Delta}\to 0$,  then
\begin{enumerate}[(i)]
\item $\Phi_k^{N,\Delta}(\omega)>0$ for all $\omega \in [-\pi,\pi],$
\item 
\begin{subnumcases}{ \Phi_k^{N,\Delta}(\omega)\gtrsim}
\frac{\sqrt\Delta }{M }\sqrt{ |\omega|},&$|\omega| \geq M^2 \Delta$, \label{eq:spec.lowerbound.1}\\
\Delta,&$k^2 \Delta\leq |\omega|\leq  M^2 \Delta$, \label{eq:spec.lowerbound.2} \\
\frac{\omega^2}{k^4\Delta}+\Delta \e ^{-\vt_2 k^2},&$|\omega|\leq k^2 \Delta$, \label{eq:spec.lowerbound.3}
\end{subnumcases}
\item 
\begin{subnumcases}{\frac{\partial}{\partial \vt_2} \Phi_k^{N,\Delta}(\omega)\lesssim}
\Delta,&$\omega\in [-\pi,\pi]$, \label{eq:spec.upperbound.1}\\
\frac{\omega^2}{k^4\Delta}+\Delta k^2\e ^{-\vt_2 k^2},&$|\omega|\leq k^2 \Delta$.\label{eq:spec.upperbound.2}
\end{subnumcases}
\end{enumerate}
\end{lem}
\begin{proof}
Without loss of generality let $\theta=\pi^2 \vt_2$ and $\sigma_0^2=\pi^2$. We denote the covariance function of $\bar U_k$ by $\rho_k:\Z \to \R$ and write $ \Phi_k^N$ instead of $ \Phi_k^{N,\Delta}$, i.e. 
$ \Phi_k^N(\omega) = \sum_{j = 1-N}^{N-1} \rho_k(j)\e^{-ij\omega },\, \omega \in [-\pi,\pi].$
$(i)$ Let $r_k$ be the covariance function of the process  $(U_k(t_0),U_k(t_1),\ldots)$, i.e. $$r_k(j)=\sum_{\ell \in \mathcal I_k}\frac{\e^{-\theta\ell^2|j|\Delta}}{2\sqrt{\theta}\ell ^2},\qquad j \in \Z,$$ where $\mathcal{I}_k=\mathcal{I}_k^+\cup\mathcal{I}_k^-$. Note that $r_k$ and $\rho_k$ are related by $\rho_k(j) = 2r_k(j)-r_k(j-1)-r_k(j+1)$, $j\in \Z,$ which is a second order difference. Since $x\mapsto \e^{-x}$ has a positive second derivative, it follows that $\rho_k(j)<0$  if $j\neq 0$. On the other hand, for $j=0$ we have $\rho_k(0)=\V(\bar U_k(t_0))>0$ and therefore,
\begin{align*}
\Phi_k^N(\omega) &= \rho_k(0)+2\sum_{j=1}^{N-1}\rho_k(j)\cos(j\omega)
\geq  \rho_k(0)+2\sum_{j=1}^{N-1}\rho_k(j)
=2(r_k(N-1)-r_k(N))
>0.
\end{align*}
To treat $(ii)$ and $(iii)$ we calculate 
\begin{align*}
\Phi_k^N(\omega) = \sum_{j=1-N}^{N-1}\rho_k(j)\e^{-ij\omega}
%&= 2\sum_{j=1-N}^{N-1}r_k(j)\e^{-ij\omega} -\sum_{j=1-N}^{N-1}r_k(j-1)\e^{-ij\omega}-\sum_{j=1-N}^{N-1}r_k(j+1)\e^{-ij\omega}\\
%&= 2\sum_{j=1-N}^{N-1}r_k(j)\e^{-ij\omega} -\e^{-i\omega}\sum_{j=-N}^{N-2}r_k(j)\e^{-ij\omega}-\e^{i\omega}\sum_{j=2-N}^{N}r_k(j)\e^{-ij\omega}\\
%&= (2-\e^{i\omega}-\e^{-i\omega})\sum_{j=2-N}^{N-2}r_k(j)\e^{-ij\omega} \\
%&\qquad +4r_k(N-1)\cos((N-1)\omega)-2(r_k(N)\cos((N-1)\omega)+r_k(N-1)\cos((N-2)\omega))\\
&= 2(1-\cos(\omega))\sum_{j=2-N}^{N-2}r_k(j)\e^{-ij\omega}+4r_k(N-1)\cos((N-1)\omega) \\
&\qquad -2r_k(N)\cos((N-1)\omega)-2r_k(N-1)\cos((N-2)\omega).
\end{align*}
From $\sum_{j=0}^{J-1} z^j = \frac{1-z^J}{1-z}$ for $z\in \C\setminus\{1\}$ it follows that 
\begin{align*}
\mathrm{\sum_{j=1-J}^{J-1}}\mathrm{e}^{-\theta\ell^{2}|j|\Delta}\mathrm{e}^{-ij\omega} 
 & =\frac{1-\mathrm{e}^{-2\theta\ell^{2}\Delta}+2\mathrm{e}^{-(J+1)\theta\ell^{2}\Delta}\cos((J-1)\omega)-2\mathrm{e}^{-J\theta\ell^{2}\Delta}\cos(J\omega)}{1+\mathrm{e}^{-2\theta\ell^{2}\Delta}-2\mathrm{e}^{-\theta\ell^{2}\Delta}\cos(\omega)}\\
 & =\frac{\sinh(\theta\ell^{2}\Delta)+\mathrm{e}^{-J\theta\ell^{2}\Delta}\cos((J-1)\omega)-\mathrm{e}^{-(J-1)\theta\ell^{2}\Delta}\cos(J\omega)}{\cosh(\theta\ell^{2}\Delta)-\cos(\omega)}
\end{align*}
for $J\geq 1$ and by elementary manipulations  we can pass to the representation
$\Phi_k^N = \Phi + R_N,$
where
\begin{align*}
\Phi(\omega)&=(1-\cos(\omega))\sum_{\ell \in \mathcal{I}_k}\frac{1}{\sqrt \theta \ell ^2 }\frac{\sinh(\theta\ell^{2}\Delta)}{\cosh(\theta\ell^{2}\Delta)-\cos(\omega)},\\
R_N(\omega)
%&=(1-\cos(\omega))\sum_{\ell \in \mathcal{I}_k}\frac{1}{\sqrt \theta \ell ^2 }\frac{\mathrm{e}^{-(N-1)\theta\ell^{2}\Delta}\cos((N-2)\omega)-\mathrm{e}^{-(N-2)\theta\ell^{2}\Delta}\cos((N-1)\omega)}{\cosh(\theta\ell^{2}\Delta)-\cos(\omega)}\\
%&\qquad+\sum_{\ell \in \mathcal I _k}\frac{2\e^{-\theta \ell^2 (N-1)\Delta}\cos((N-1)\omega)-\e^{-\theta \ell^2 N\Delta}\cos((N-1)\omega)-\e^{-\theta \ell^2 (N-1)\Delta}\cos((N-2)\omega)}{\sqrt \theta \ell ^2}\\
&=\sum_{\ell \in \mathcal I_k}(1-\cosh(\theta \ell ^2 \Delta))\frac{\e^{-\theta\ell^2(N-1)\Delta}}{\sqrt \theta \ell^2} \frac{\e^{-\theta\ell^2 \Delta}\cos((N-1)\omega)-\cos(N\omega)}{\cosh(\theta \ell ^2 \Delta)-\cos(\omega)}.
\end{align*}
Note that we have suppressed the dependence on $k$ for ease of notation. We remark that 
$\Phi(\omega)=\sum_{j\in \Z}\rho_k(j)\e^{-ij\omega},\,\omega\in [-\pi,\pi],$
is the spectral density of the process $(\bar U_k(j))_{j\geq0}$.\\
$(ii)$ To prove \eqref{eq:spec.lowerbound.1} we note that for $\omega \geq M^2 \Delta$ we have 
\begin{align*}
&\Big|\e^{-\theta\ell^2 \Delta}\cos((N-1)\omega)-\cos(N\omega)\Big|\\
&=\Big|(\e^{-\theta\ell^2 \Delta}-1)\cos((N-1)\omega)+\cos((N-1)\omega)-\cos(N\omega)\Big|\lesssim\ell^2 \Delta +\omega
\lesssim \ell^2 \omega.
\end{align*}
Consequently,
\begin{align*}
R_N(\omega)
& \lesssim \sum_{\ell \in \mathcal I_k}\ell^2 \Delta \sinh(\theta \ell^2 \Delta)\frac{\e^{-\theta\ell^2(N-1)\Delta}}{\sqrt \theta \ell^2} \frac{\ell^2\omega}{\cosh(\theta \ell ^2 \Delta)-\cos(\omega)}\\
&\lesssim \frac{\Delta}{\omega}  \sum_{\ell \in \mathcal I_k} \frac{\sinh(\theta \ell^2 \Delta)}{\ell^2(\cosh(\theta \ell ^2 \Delta)-\cos(\omega))}(1-\cos(\omega))\lesssim \frac{1}{M^2} \Phi(\omega)
\end{align*}
and hence, $R_N$ is negligible compared to $\Phi$. In order to compute an asymptotic expression for $\Phi$, set
$$h(x,\omega)=\frac{\sinh(\theta x^2)(1-\cos(\omega))}{x^2(\cosh(\theta x^2)-\cos(\omega))},\quad x>0, \,\omega \in [-\pi,\pi].$$
We have $\frac{\partial h}{\partial x}\leq 0$ and therefore,
$\left\Vert\frac{\partial}{\partial x}h(\cdot,\omega)\right\Vert_{L^1}=h(0,\omega)-\lim_{x\to \infty}h(x,\omega)=\theta$ is uniformly bounded in $\omega$. Thus, using the mean value theorem and a  Riemann sum approximation with mesh size $M\sqrt \Delta $ for $\frac{\partial}{\partial x}h(\cdot,\omega)$, we obtain
\begin{equation*} 
\Phi(\omega)\eqsim \Delta\sum_{\ell \in \mathcal I_k}h(\ell \sqrt \Delta,\omega)
= \Delta \sum_{\ell =1}^\infty h(2\ell M \sqrt \Delta,\omega)+\mathcal O(\Delta).
\end{equation*}
Further, since 
\begin{equation} \label{eq:Riemann_firstorder}
\Big|\eps \sum_{\ell\geq 1}f(\ell\eps)-\int_0^\infty f(x)\,dx\Big|\leq \eps \Vert f'\Vert_{L^1}
\end{equation}
 for any function $f\in C^1[0,\infty)$,
we get
$
\Phi(\omega)\eqsim \frac{\sqrt \Delta}{M}\int_0^\infty h(x,\omega)\,dx + \mathcal{O}(\Delta).
$
Finally, due to
\begin{equation} \label{eq:sum_equiv_max}
a+b \eqsim \max(a,b), \qquad a,b>0,
\end{equation}
we have
$(\cosh(\theta \omega x^2)-\cos(\omega))\eqsim \max \left(\cosh(\theta \omega x^2)-1,1-\cos(\omega) \right)$
and consequently,
$$h(\sqrt \omega x,\omega)=\frac{\sinh(\theta \omega x^2)(1-\cos(\omega))}{\omega x^2(\cosh(\theta \omega x^2)-\cos(\omega))}\gtrsim \frac{\sinh(\theta \omega x^2)}{\omega x^2} \gtrsim 1,\qquad x\leq \theta^{-1/2}.$$
Therefore,
$$\int_0^\infty h(x,\omega)\,dx =\sqrt \omega \int_0^\infty h(\sqrt \omega x,\omega)\,dx\gtrsim \sqrt \omega,$$
finishing the proof of \eqref{eq:spec.lowerbound.1}.\\
To prove \eqref{eq:spec.lowerbound.2} and \eqref{eq:spec.lowerbound.3}, let us write 
$\Phi = \sum_{\ell \in \mathcal I_k}\varphi_\ell$ and $R_N=\sum_{\ell \in \mathcal I_k}\varrho_\ell^N.$
Since the argument in the proof of $(i)$ was on a summand-wise level, also each of the functions $\varphi_\ell+\varrho_\ell^N$ is positive, $\ell \in \N$.  Therefore, we can bound $\Phi_k^N$ from below with the first summand,
$$\Phi_k^N \geq \varphi_k+\varrho_k^N =\varrho_k^N(0)+\varphi_k+\left(\varrho_k^N-\varrho_k^N(0)\right).$$ We show that there exists an environment $U$ around zero and some $\delta \in (0,1)$ such that 
\begin{equation} \label{eq:PsiminPsi0}
|\varrho_k^N(\omega)-\varrho_k^N(0)|\leq  (1-\delta)\varphi_k(\omega),\qquad \omega \in  U:
\end{equation}
A simple calculation yields
\begin{align*}
\varrho_k^N(\omega)-\varrho_k^N(0)
&= \mathrm{e}^{-(N-1)\theta k^{2}\Delta}\frac{(\cos((N-1)\omega))-\cos(N\omega))(1-\cosh(\theta k^{2}\Delta))}{\sqrt{\theta} k^{2}(\cosh(\theta  k^{2}\Delta)-\cos(\omega))}\\
&\qquad +
\mathrm{e}^{-(N-1)\theta k^{2}\Delta}\frac{(1-\mathrm{e}^{-\theta k^{2}\Delta})(1-\cos((N-1)\omega)))(1-\cosh(\theta k^{2}\Delta))}{\sqrt{\theta} k^{2}(\cosh(\theta k^{2}\Delta)-\cos(\omega))}\\
&\qquad +
\mathrm{e}^{-(N-1)\theta k^{2}\Delta}\frac{\left(\mathrm{e}^{-\theta k^{2}\Delta}-1\right)(1-\cos(\omega))}{\sqrt{\theta} k^{2}(\cosh(\theta k^{2}\Delta)-\cos(\omega))} .
\end{align*}
Since $\cos(x)-\cos(y)=-2\sin \frac{x+y}{2}\sin \frac{x-y}{2},\,x,y\in\R,$ we have
\begin{align} \label{eq:bound_Nomega2}
\Big|\cos((N-1)\omega)-\cos(N\omega)\Big|=\Big|2\sin\left( \frac{(2N-1)\omega}{2} \right) \sin \left( \frac{\omega}{2} \right)\Big|\leq N\omega^2.
\end{align}
Therefore, for any $\alpha>0$ there exists an environment $U$ of 0 such that 
\begin{align*}
|\cos((N-1)\omega)-\cos(N\omega)|&\leq N \omega^2 \leq N (1-\cos(\omega))(2+\alpha)\\
1-\cos((N-1)\omega) &\leq \frac{N^2 \omega^2}{2} \leq \frac{N^2}{2}(1-\cos(\omega))(2+\alpha)
\end{align*}
holds for all $\omega \in U$.
Further, for all $x \geq 0$ we have
$\cosh(x)-1  \leq \frac{\sinh(x)x}{2}$, $1-\e^{-x} \leq \sinh(x),$
and consequently,
\begin{align*}
\frac{|\varrho_k^N(\omega)-\varrho_k^N(0)|}{\varphi_k(\omega)} &\leq \mathrm{e}^{-(N-1)\theta k^{2}\Delta} (1+ \frac{2+\alpha}{2}\theta k^2+\frac{2+\alpha}{4} \theta^2 k^4 ) \\
&\leq \frac{2+\alpha}{2}\e^{\Delta\theta k^2}\e^{-\theta k^2}(1+\theta k^2 + \frac{\theta^2 k^4}{2})<\frac{2+\alpha}{2}\e^{\Delta\theta k^2}.
\end{align*}
Clearly, for $\Delta$ sufficiently small one can choose $\alpha$ in such a way that this bound is strictly less than 1 for all $k \leq M-1$, yielding \eqref{eq:PsiminPsi0}. Consequently, it is sufficient to prove  \eqref{eq:spec.lowerbound.2} and \eqref{eq:spec.lowerbound.3} with $\Phi_k^N$ replaced by $\varphi_k +\varrho_k^N(0)$: Now, 
$$\varphi_k(0) +\varrho_k^N(0)=\varrho_k^N(0) =\e^{-\theta k^2 (N-1)\Delta}\frac{1-\e^{-\theta k^2 \Delta}}{k^2}\eqsim \Delta \e^{-\theta k^2 } $$ 
and again by using \eqref{eq:sum_equiv_max},  we get
\begin{align*}
 \varphi_k (\omega) &\gtrsim \frac{\sinh(\theta k^{2}\Delta)}{k^2} \gtrsim \Delta,\qquad \omega\geq k^2 \Delta,\\
 \varphi_k(\omega) &\gtrsim (1-\cos(\omega))\frac{1}{\sqrt \theta k ^2 }\frac{\sinh(\theta k^{2}\Delta)}{\cosh(\theta k^{2}\Delta)-1} \gtrsim \frac{\omega^2}{k^4 \Delta},\qquad \omega \leq  k^2 \Delta.
\end{align*}
$(iii)$ We show \eqref{eq:spec.upperbound.1}: We have 
$\frac{\partial}{\partial \theta} \Phi(\omega)=  \frac{\Delta}{2\sqrt \theta}\sum_{\ell\in\mathcal{I}_{k}}g(\ell  \sqrt {\theta\Delta},\omega)$
with $g$ defined in Lemma \ref{lem:g_properties}.
Using the properties of $g$ derived in Lemma \ref{lem:g_properties} and the Riemann sum approximation \eqref{eq:Riemann_firstorder} with mesh size $M\sqrt \Delta$, we obtain
\begin{align*}
\frac{\partial}{\partial \theta} \Phi(\omega) 
&\eqsim \Delta\sum _{\ell \geq 1} g(\ell M\sqrt \Delta,\omega)+\mathcal{O}(\Delta)
=\frac{\sqrt \Delta}{M}\int_0^\infty g(x,\omega)\,dx +\O(\Delta)
=\O(\Delta).
\end{align*}
To show that also $\frac{\partial}{\partial \theta} R_N$ is of the claimed order, we write 
\begin{align*}
\varrho_\ell^N= \alpha_\ell \beta_\ell\quad\text{where}\quad\alpha_\ell(\omega)&= \frac{1-\cosh(\theta \ell ^2 \Delta)}{\sqrt \theta \ell^2\left(\cosh(\theta \ell ^2 \Delta)-\cos(\omega)\right)},\\
\beta_\ell(\omega)&=\e^{-\theta\ell^2(N-1)\Delta}\left(\e^{-\theta\ell^2 \Delta}\cos((N-1)\omega)-\cos(N\omega)\right).
\end{align*}
The corresponding derivatives are given by
\begin{align*} 
\frac{\partial}{\partial \theta}\alpha_\ell(\omega) 
&=  \underbrace{\frac{\cosh(\theta \ell ^2 \Delta)-1}{2\theta^{3/2} \ell^2\left(\cosh(\theta \ell ^2 \Delta)-\cos(\omega)\right)}}_{=:{a}_\ell^1(\omega)} 
\underbrace{- \frac{\Delta\sinh(\theta \ell ^2 \Delta) \left(1-\cos(\omega)\right)}{ \sqrt \theta\left(\cosh(\theta \ell ^2 \Delta)-\cos(\omega)\right)^2}} _{=:{a}_\ell^2(\omega)} \quad\text{and}\\
\frac{\partial}{\partial \theta}\beta_\ell(\omega)
&= \e^{-\theta\ell^2(N-1)\Delta} \left(-\ell^2N\Delta\e^{-\theta\ell^2\Delta}\cos((N-1)\omega)+\ell^2(N-1)\Delta\cos(N\omega)\right)
=:b_\ell(\omega).
\end{align*}
Using the estimates
\begin{align*}
\frac{\cosh(x)-1}{\cosh(x)-\cos(\omega)}\lesssim \frac{x^2}{x^2\vee \omega^2},\quad
\frac{x\sinh(x)(1-\cos(\omega))}{(\cosh(x)-\cos(\omega))^2}\lesssim \frac{x^2}{x^2\vee \omega^2}
\end{align*}
in combination with $
\beta_\ell(\omega) 
\lesssim \e^{-\theta\ell^2(N-1)\Delta}\left((\ell^2\Delta)\vee \omega\right)$ and 
$b_\ell(\omega)
\lesssim \e^{-\theta\ell^2(N-1)\Delta}\ell^2\left((\ell^2\Delta)\vee \omega\right)$
shows that any of the
 three products in 
\begin{equation} \label{eq:dR_N_product}
\frac{\partial}{\partial \theta}R_N=\sum_{\ell \in \mathcal I _k} a_\ell^1\beta_\ell+a_\ell^2\beta_\ell+\alpha_\ell b_\ell
\end{equation} 
can be bounded by 
\begin{align*}
\sum_{\ell \in \mathcal I_k} \e^{-\theta\ell^2(N-1)\Delta} \frac{\ell^4\Delta^2}{(\ell^4\Delta^2)\vee \omega^2}\left((\ell^2 \Delta) \vee \omega\right)\leq \Delta\sum_{\ell \in \mathcal I_k} \e^{-\theta\ell^2(N-1)\Delta} \ell^2\lesssim \Delta.
\end{align*}
Consequently, we have $\frac{\partial}{\partial \theta }R_N =\mathcal O (\Delta)$, which finishes the proof of \eqref{eq:spec.upperbound.1}.\\ 
%we note that for any $a>0$ we have 
%\begin{align*}
%&\left|\frac{2a-\sinh(a)\cosh(a)+\cos(\omega)(\sinh(a)-2a\cosh(a))}{(\cosh(a)-\cos(\omega))^2}\right|\\
%&=\left| \frac{(\cos(\omega)-1)(\sinh(a)-2a\cosh(a))}{(\cosh(a)-\cos(\omega))^2}
%+\frac{(2a+\sinh(a))(1-\cosh(a))}{(\cosh(a)-\cos(\omega))^2}\right|\\
%&\leq \left| \frac{\sinh(a)-2a\cosh(a)}{\cosh(a)-\cos(\omega)} \right| +\left|\frac{2a+\sinh(a)}{\cosh(a)-\cos(\omega)}\right|\\
%&\leq \left| \frac{\sinh(a)-2a\cosh(a)}{\cosh(a)-1} \right| +\left|\frac{2a+\sinh(a)}{\cosh(a)-1}\right|\\
%&\lesssim \left|\frac{2a+\sinh(a)}{\cosh(a)-1}\right|\\
%&\lesssim \frac{1+a}{a}
%\end{align*} 
%where the last step follows for instance from...
To prove \eqref{eq:spec.upperbound.2}, we use property $(iii)$ of Lemma \ref{lem:g_properties} to deduce
\begin{align*}
\frac{\partial}{\partial \theta} \Phi(\omega)
&\lesssim \omega^2\Delta\sum_{\ell\in \mathcal I_k}\frac{1+\theta \ell^2 \Delta}{\theta^2 \ell^4 \Delta^2}
\lesssim \frac{\omega^2}{\Delta}\left( \frac{1+\theta k^2 \Delta}{\theta^2 k^4 }+\sum_{\ell \geq 1}\frac{1+\theta (2\ell M)^2 \Delta}{\theta^2 (2\ell M)^4}\right)
\lesssim \frac{\omega^2}{k^4 \Delta},
\end{align*}
where the last step follows from $k^2 \Delta\leq M^2 \Delta \to 0$. Further, using decomposition \eqref{eq:dR_N_product},
\begin{align}
&\frac{\partial}{\partial \theta}( R_N(\omega)-R_N(0)) =\sum_{\ell \in \mathcal I_k}a_\ell^1(\omega)(\beta_\ell(\omega)-\beta_\ell(0))+ \sum_{\ell \in \mathcal I_k}(a_\ell^1(\omega)-a_\ell^1(0))\beta_\ell(0)\nonumber\\
&\qquad\qquad+\sum_{\ell \in \mathcal I_k}a_\ell^2(\omega)\beta_\ell(\omega)+\sum_{\ell \in \mathcal I_k}\alpha_\ell(\omega)(b_\ell(\omega)-b_\ell(0))+ \sum_{\ell \in \mathcal I_k}(\alpha_\ell(\omega)-\alpha_\ell(0))b_\ell(0).\label{eq:dR_min_dR0}
\end{align}
Now, by \eqref{eq:bound_Nomega2}, we have
\begin{align*}
\beta_\ell(\omega)-\beta_\ell(0)
&= \e^{-\theta\ell^2(N-1)\Delta}\left((\e^{-\theta\ell^2 \Delta}-1)(\cos((N-1)\omega)-1)+\cos((N-1)\omega)-\cos(N\omega)\right)\\
&\lesssim\e^{-\theta\ell^2(N-1)\Delta} \ell^2 N\omega^2.
\end{align*}
In a similar way we can bound
\begin{align*}
\beta_\ell(0)&\lesssim \e^{-\theta\ell^2(N-1)\Delta} \ell^2\Delta,\qquad
\beta_\ell(\omega)\lesssim \e^{-\theta\ell^2(N-1)\Delta} \left( (\ell^2 \Delta) \vee \omega\right) \lesssim \e^{-\theta\ell^2(N-1)\Delta} \ell^2 \Delta,\\
b_\ell(\omega)-b_\ell(0) &\lesssim \e^{-\theta\ell^2(N-1)\Delta} \ell^4 N \omega^2,\quad
b_\ell(0)\lesssim \e^{-\theta\ell^2(N-1)\Delta} \ell^4 \Delta,
\end{align*}
where the second inequality uses $\omega \leq k^2 \Delta \leq \ell^2 \Delta$ for $\ell \in \mathcal{I}_k$.
Also,
\begin{align*}
a_\ell^1(\omega)-a_\ell^1(0)&\lesssim\frac{1-\cos(\omega)}{\cosh(\theta\ell^2\Delta)-\cos(\omega)} \lesssim \frac{1-\cos(\omega)}{(\cosh(\theta\ell^2\Delta)-1)}\lesssim \frac{\omega^2}{k^4 \Delta^2}
\end{align*}
and similarly,
$\alpha_\ell(\omega)-\alpha_\ell(0)\lesssim \frac{\omega^2}{k^4 \Delta^2}$, 
$a_\ell^2(\omega)\lesssim \frac{\omega^2}{k^4\Delta^2}$,
$a_\ell^1(\omega)\lesssim 1$ and $\alpha_\ell(\omega)\lesssim 1$.

Using the bounds just developed in combination with $\e^{-\theta\ell^2(N-1)\Delta}\lesssim \frac{1}{k^4 \ell^m},\,m\in\N,$ shows that any of the five terms in \eqref{eq:dR_min_dR0} is of order $\mathcal O(\frac{\omega^2}{k^4\Delta})$ and hence, 
$\frac{\partial}{\partial \theta}(R_N(\omega)-R_N(0))\lesssim \frac{\omega^2}{k^4\Delta}.$
Now, the proof of \eqref{eq:spec.upperbound.2} is finalized by 
\begin{align*}
\frac{\partial}{\partial \theta}R_N(0)&=\sum_{\ell \in \mathcal I_k}\e^{-\theta\ell^2(N-1)\Delta}\frac{2\theta\ell^2(N-1)\Delta(\e^{-\theta \ell^2 \Delta}-1)+2\theta\ell^2\Delta\e^{-\theta\ell^2\Delta}+\e^{-\theta\ell^2\Delta}-1}{2\theta^{3/2} \ell^2}\\
&\lesssim \Delta \sum_{\ell \in \mathcal I_k}\e^{-\theta\ell^2(N-1)\Delta} \ell^2 \lesssim \Delta k^2 \e^{-\theta k^2}.\qedhere
\end{align*}
\end{proof}

\subsection{Bounds on Fourier series and Riemann summation}
The Lemmas in this section provide bounds for Fourier series and Taylor expansions for Riemann sums. 
%They are our basic tools for computing and bounding covariances. 
Similar results are stated in Lemma 7.2 of \cite{Bibinger18}. 
%Instead of the derivation in \cite{Bibinger18}, which uses the decay of Fourier transforms of $L^1$-functions, we show that these results also follow from the following simple Lemma.   
\begin{lem} \label{lem:trig_series_fundamental}
Let $(a_{n})$ be a real sequence and $\tau\in \{\sin, \cos\}$. Then, 
\[
\left|\sum_{k=1}^{N}a_{k}\tau(ky)\right|\leq\frac{1+2K_N}{y\wedge(2\pi-y)}\sup_{n\leq N}|a_{n}|
\]
holds for any $y\in(0,2\pi)$ where $K_N$ is the number of monotone sections of $(a_{n})_{1\leq n \leq N}$. 
\end{lem}

\begin{proof}
By Lagrange's trigonometric
identities, 
\begin{align*}
\sum_{k=1}^{N}\cos(ky) & =\frac{\sin\left((N+1/2)y\right)-\sin(y/2)}{2\sin(y/2)},\quad
\sum_{k=1}^{N}\sin(ky)  =\frac{\cos(y/2)-\cos\left((N+1/2)y\right)}{2\sin(y/2)},
\end{align*}
we have $\left|\sum_{k=M}^{N}\tau(ky)\right|\leq\frac{1}{\sin(y/2)}\leq\frac{1}{y\wedge(2\pi-y)}$ uniformly in $M\leq N$.
Therefore, $|\sum_{k=1}^{N}a_{k}\tau(ky)|$ can be decomposed by
\begin{align*}
 &\left|a_{1}\sum_{k=1}^{N}\tau(ky)+(a_{2}-a_{1})\sum_{k=2}^{N}\tau(ky)
 +(a_{3}-a_{2})\sum_{k=3}^{N}\tau(ky)+\cdots+(a_{N}-a_{N-1})\tau(Ny)\right|\\
 &\leq|a_{1}|\left|\sum_{k=1}^{N}\tau(ky)\right|+|a_{2}-a_{1}|\left|\sum_{k=2}^{N}\tau(ky)\right|
 +|a_{3}-a_{2}|\left|\sum_{k=3}^{N}\tau(ky)\right|+\cdots+|a_{N}-a_{N-1}|\left|\tau(Ny)\right|\\
 &\leq\frac{1}{y\wedge(2\pi-y)}\left(|a_{1}|+\sum_{k=1}^{N-1}|a_{k+1}-a_{k}|\right)
  \leq\frac{1+2K_N}{y\wedge(2\pi-y)}\sup_{n\leq N}|a_{n}|,
\end{align*}
where the last inequality follows from the fact that if $(a_{k})_{N_0\leq k\leq N_1}$ is monotone for some $N_0\leq N_1\leq N$, then 
$\sum _{k=N_0}^{N_1-1}|a_{k+1}-a_k|=|a_{N_1}-a_{N_0}| \leq 2\sup_{n\leq N}|a_n|.$
\end{proof}

\begin{lem} \label{lem:trig_integral}
Let $g\in C^{1}\left(\R_{+}\right)$ be such that $g'$ is bounded and has a finite number $K$ of monotone sections. Then, for
$y\in(0,2\pi)$, as $\eps\to 0,$
\begin{align*}
\sum_{k=1}^{\infty}g(k\eps)\cos(ky) & =-\frac{g(0)}{2}+\mathcal{O}\left(\frac{\eps\left\Vert g'\right\Vert _{\infty}}{\left(y\wedge(2\pi-y)\right)^{2}}\right)\\
\sum_{k=1}^{\infty}g(k\eps)\sin(ky) & =\frac{g(0)}{2}\cot\left(\frac{y}{2}\right)+\mathcal{O}\left(\frac{\eps\left\Vert g'\right\Vert _{\infty}}{\left(y\wedge(2\pi-y)\right)^{2}}\right).
\end{align*}
\end{lem}

\begin{proof}
We use  the formula $\sin(\alpha)-\sin(\beta)=2\cos\frac{\alpha+\beta}{2}\sin\frac{\alpha-\beta}{2},\,\alpha,\beta\in\mathbb{R}$,
to calculate 
\begin{align*}
&\frac{g(0)}{2}+\sum_{k=1}^{\infty}g(k\eps)\cos(ky)
 =\frac{g(0)}{2}+\frac{1}{2\sin\frac{y}{2}}\sum_{k=1}^{\infty}g(k\eps)\big(\sin\left(\left(k+1/2\right)y\right)-\sin\left(\left(k-1/2\right)y\right)\big)\\
 &\qquad =\frac{g(0)}{2}-\frac{g(\eps)}{2}+\frac{1}{2\sin\frac{y}{2}}\sum_{k=1}^{\infty}\sin\left(\left(k+1/2\right)y\right)\big(g(k\eps)-g((k+1)\eps)\big)\\
 &\qquad =-\frac{1}{2}\left(g'(\xi_{0}^{\eps})+\frac{1}{\sin\frac{y}{2}}\sum_{k=1}^{\infty}\sin\left(\left(k+1/2\right)y\right)g'(\xi_{k}^{\eps})\right)\eps
 \leq\frac{1+2K}{\left(y\wedge(2\pi-y)\right)^{2}}\left\Vert g'\right\Vert _{\infty}\eps,
\end{align*}
where $\xi_{k}^{\eps}\in[k\eps,(k+1)\eps]$. Here, the last step follows from $\sin((k+1/2)y)=\sin(ky)\cos(y/2)+\cos(ky)\sin(y/2)$ and then applying Lemma \ref{lem:trig_series_fundamental}. The  second
statement can be proved analogously, using $\cos(\alpha)-\cos(\beta)=-2\sin\left(\frac{\alpha+\beta}{2}\right)\sin\left(\frac{\alpha-\beta}{2}\right),\,\alpha,\beta\in\R.$
\end{proof}

\begin{lem} \label{lem:Riemann} 
Let $g\in C^{2}(\R_{+})\cap L^{1}(\R_{+})$, $g'\in L^\infty(\R_+)$ and $g''\in L^{1}(\R_{+})$.
Then,
\begin{enumerate}[(i)]
\item $\displaystyle
\varepsilon\sum_{k\geq1}g(k\varepsilon)=\int_{0}^{\infty}g(z)\,dz-\frac{g(0)}{2}\varepsilon+\O(\varepsilon^{2}\left\Vert g''\right\Vert _{L^{1}}),
$
\item $\displaystyle
\varepsilon\sum_{k\geq1}g(k\varepsilon)\sin^2(ky)=\frac{1}{2}\int_{0}^{\infty}g(z)\,dz+\O\left(\eps^2\left(\frac{\left\Vert g'\right\Vert_\infty }{(y\wedge (\pi -y))^2}\wedge \Vert g''\Vert_{L^1}\right)\right).
$
\end{enumerate}
\end{lem}

\begin{proof}
For a detailed proof of $(i)$ we refer to \cite[Lemma 7.2]{Bibinger18}. The main idea is to regard each term $\eps g(k\eps)$ as a midpoint integral approximation. 
%\textcolor{blue}{
%By Taylor's formula,
%\begin{align*}
%\Delta\sum_{k\geq1}g(k\Delta)-\int_{\Delta/2}^{\infty}g(z)\,dz & =\sum_{k\geq1}\int_{(k-1/2)\Delta}^{(k+1/2)\Delta}\left(g(k\Delta)-g(z)\right)\,dz\\
% & =-\sum_{k\geq1}\int_{(k-1/2)\Delta}^{(k+1/2)\Delta}\left((z-k\Delta)g'(k\Delta)+\frac{(z-k\Delta)^{2}}{2}g''(\xi_{k}^{\Delta,z})\right)\,dz\\
% & =-\sum_{k\geq1}\int_{(k-1/2)\Delta}^{(k+1/2)\Delta}\frac{(z-k\Delta)^{2}}{2}g''(\xi_{k}^{\Delta,z})\,dz\\
% & \lesssim\sum_{k\geq1}|g''(\eta_{k}^{\Delta})|\int_{(k-1/2)\Delta}^{(k+1/2)\Delta}\frac{(z-k\Delta)^{2}}{2}\,dz\\
% & \lesssim\Delta^{3}\sum_{k\geq1}|g''(\eta_{k}^{\Delta})|\\
% & =\O(\Delta^{2}\left\Vert g''\right\Vert _{L^{1}})
%\end{align*}
%where $\xi_{k}^{\Delta,z},\eta_{k}^{\Delta}\in\left[(k-1/2)\Delta,(k+1/2)\Delta\right]$.
%Further,
%\[
%\int_{\Delta/2}^{\infty}g(z)\,dz=\int_{0}^{\infty}g(z)\,dz-\int_{0}^{\Delta/2}g(z)\,dz=\int_{0}^{\infty}g(z)\,dz-\frac{g(0)}{2}\Delta+\O(\Delta^{2}).
%\]
%}
Since $\sin^2(y)=(1-\cos(2y))/2$, statement $(ii)$ is a direct consequence of $(i)$ and the previous lemma.
\end{proof}

\begin{lem} \label{lem:RiemannTaylor}
Let $g\in C^{2}(\mathbb{R}_{+})$ and $M\to\infty,\,M\epsilon\to0$.
Then,
\[
\epsilon\sum_{k=1}^{M}g(k\epsilon)=M\epsilon g(0)+\frac{(M^{2}+M)\epsilon^{2}}{2}g'(0)+\mathcal{O}((M\epsilon)^{3}).
\]
\end{lem}

\begin{proof}
First of all, by the midpoint rule there exist $\eta_{k}\in[(k-1/2)\epsilon,(k+1/2)\epsilon]$
such that
\begin{align*}
\left|\epsilon\sum_{k=1}^{M}g(k\epsilon)-\int_{\epsilon/2}^{(M+1/2)\epsilon}g(x)\,dx\right| & =\left|\sum_{k=1}^{M}\int_{(k-1/2)\epsilon}^{(k+1/2)\epsilon}(g(k\epsilon)-g(x))\,dx\right|
\leq\epsilon^{3}\sum_{k=1}^{M}|g''(\eta_{k})|
\lesssim  M^{3}\epsilon^{3}
\end{align*}
and secondly, a Taylor approximation shows that 
\begin{align*}
\int_{\epsilon/2}^{(M+1/2)\epsilon}g(x)\,dx 
 & =M\epsilon g(0)+\frac{(M^{2}+M)\epsilon^{2}}{2}g'(0)+\mathcal{O}((M\epsilon)^{3}).\qedhere
\end{align*}
\end{proof}

\bibliography{refs}

\begin{thebibliography}{}

\bibitem[Altmeyer and Rei{\ss}, 2019]{altmeyerReiss2019}
Altmeyer, R. and Rei{\ss}, M. (2019).
\newblock Nonparametric estimation for linear {SPDE}s from local measurements.
\newblock {\em arXiv preprint arXiv:1903.06984}.

\bibitem[Bibinger and Trabs, 2019a]{bibingerTrabs2019}
Bibinger, M. and Trabs, M. (2019a).
\newblock On central limit theorems for power variations of the solution to the
  stochastic heat equation.
\newblock In {\em Stochastic Models, Statistics and Their Applications.
  Springer Proceedings in Mathematics \& Statistics}, volume 294, pages 69--84.

\bibitem[Bibinger and Trabs, 2019b]{Bibinger18}
Bibinger, M. and Trabs, M. (2019b).
\newblock Volatility estimation for stochastic {PDE}s using high-frequency
  observations.
\newblock {\em Stochastic Process. Appl.}
\newblock Forthcoming.

\bibitem[Chong, 2019a]{Chong18}
Chong, C. (2019a).
\newblock High-frequency analysis of parabolic stochastic {PDE}s.
\newblock {\em Ann. Statist.}
\newblock Forthcoming.

\bibitem[Chong, 2019b]{chong2019}
Chong, C. (2019b).
\newblock High-frequency analysis of parabolic stochastic {PDE}s with
  multiplicative noise: {Part I}.
\newblock {\em arXiv preprint arXiv:1908.04145}.

\bibitem[Cialenco, 2018]{Cialenco18}
Cialenco, I. (2018).
\newblock Statistical inference for spdes: an overview.
\newblock {\em Statistical Inference for Stochastic Processes}, 21(2):309--329.

\bibitem[Cialenco and Huang, 2019]{Cialenco17}
Cialenco, I. and Huang, Y. (2019).
\newblock A note on parameter estimation for discretely sampled {SPDE}s.
\newblock {\em Stoch. Dyn.}
\newblock Forthcoming.

\bibitem[Cont, 2005]{Cont2005}
Cont, R. (2005).
\newblock Modeling term structure dynamics: an infinite dimensional approach.
\newblock {\em Int. J. Theor. Appl. Finance}, 8(3):357--380.

\bibitem[Dacunha-Castelle and Duflo, 1986]{Dacunha86}
Dacunha-Castelle, D. and Duflo, M. (1986).
\newblock {\em Probability and statistics. Vol. II}.
\newblock Springer-Verlag, Berlin Heidelberg New York.

\bibitem[Devroye et~al., 2019]{Devroye19}
Devroye, L., Mehrabian, A., and Reddad, T. (2019).
\newblock The total variation distance between high-dimensional {G}aussians.
\newblock {\em arXiv preprint arXiv:1810.08693v3}.

\bibitem[Dostal, 2019]{dostal2019}
Dostal, L. (2019).
\newblock The effect of random wind forcing in the nonlinear {S}chr{\"o}dinger
  equation.
\newblock {\em Fluids}, 4(3):121.

\bibitem[Hottovy and Stechmann, 2015]{hottovyStechmann2015}
Hottovy, S. and Stechmann, S.~N. (2015).
\newblock A spatiotemporal stochastic model for tropical precipitation and
  water vapor dynamics.
\newblock {\em J. Atmospheric Sci.}, 72(12):4721--4738.

\bibitem[Huebner et~al., 1993]{huebnerEtAl1993}
Huebner, M., Khasminskii, R., and Rozovskii, B. (1993).
\newblock Two examples of parameter estimation for stochastic partial
  differential equations.
\newblock In {\em Stochastic processes}, pages 149--160. Springer.

\bibitem[Huebner and Rozovskii, 1995]{HuebnerRozovskii1995}
Huebner, M. and Rozovskii, B.~L. (1995).
\newblock On asymptotic properties of maximum likelihood estimators for
  parabolic stochastic {PDE}'s.
\newblock {\em Probab. Theory Related Fields}, 103(2):143--163.

\bibitem[Ibragimov and Rozanov, 1978]{Ibragimov78}
Ibragimov, I. and Rozanov, Y. (1978).
\newblock {\em Gaussian random processes}.
\newblock Springer-Verlag, Berlin Heidelberg New York.

\bibitem[Ibragimov and Has'minskii, 1981]{IbragimovHasminskii1981}
Ibragimov, I.~A. and Has'minskii, R.~Z. (1981).
\newblock {\em Statistical estimation}, volume~16 of {\em Applications of
  Mathematics}.
\newblock Springer-Verlag, New York Berlin.
\newblock Asymptotic theory, Translated from the Russian by Samuel Kotz.

\bibitem[Isserlis, 1918]{Isserlis18}
Isserlis, L. (1918).
\newblock On a formula for the product-moment coefficient of any order of a
  normal frequency distribution in any number of variables.
\newblock {\em Biometrika}, 12:134--139.

\bibitem[Kaino and Uchida, 2019]{Uchida19}
Kaino, Y. and Uchida, M. (2019).
\newblock Parametric estimation for a parabolic linear {SPDE} model based on
  sampled data.
\newblock {\em arXiv preprint arXiv:1909.13557}.

\bibitem[Koski and Loges, 1985]{KoskiLoges85}
Koski, T. and Loges, W. (1985).
\newblock Asymptotic statistical inference for a stochastic heat flow problem.
\newblock {\em Statist. Probab. Lett.}, 3:185--189.

\bibitem[Kriz and Maslowski, 2019]{krizMaslowski2019}
Kriz, P. and Maslowski, B. (2019).
\newblock Central limit theorems and minimum-contrast estimators for linear
  stochastic evolution equations.
\newblock {\em Stochastics}, 0(0):1--32.

\bibitem[Lototsky, 2009]{lototsky2009}
Lototsky, S.~V. (2009).
\newblock Statistical inference for stochastic parabolic equations: a spectral
  approach.
\newblock {\em Publ. Mat.}, 53(1):3--45.

\bibitem[Markussen, 2013]{Markussen03}
Markussen, B. (2013).
\newblock Likelihood inference for a discretely observed stochastic partial
  differential equation.
\newblock {\em Bernoulli}, 9(5):745 -- 762.

\bibitem[Mathai and Provost, 1992]{Mathai92}
Mathai, A.~M. and Provost, S.~B. (1992).
\newblock {\em Quadratic Forms in Random Variables}.
\newblock Marcel Dekker, inc., New York.

\bibitem[Neveu, 1968]{Neveu68}
Neveu, J. (1968).
\newblock {Processus al\'eatoires gaussiens}.
\newblock {Seminaire de mathematiques superieures. Les Presses de
  l'Universit\'e de Montr\'eal}.

\bibitem[Piterbarg and Ostrovskii, 1997]{piterbargOstrovskii1997}
Piterbarg, L.~I. and Ostrovskii, A.~G. (1997).
\newblock {\em Advection and {D}iffusion in {R}andom {M}edia: {I}mplications
  for {S}ea {S}urface {T}emperature {A}nomalies}.
\newblock Springer Science \& Business Media.

\bibitem[Prato and Zabczyk, 2014]{DaPrato14}
Prato, G.~D. and Zabczyk, J. (2014).
\newblock {\em Stochastic Equations in Infinite Dimensions}.
\newblock Cambridge University Press, Cambridge.

\bibitem[Santa~Clara and Sornette, 2000]{santaclaraSornette2000}
Santa~Clara, P. and Sornette, D. (2000).
\newblock The dynamics of the forward interest rate curve with stoachstic
  string shocks.
\newblock {\em Rev. Financial Stud.}, 14(1):149--185.

\bibitem[Shevchenko et~al., 2019]{shevchenkoEtAl2019}
Shevchenko, R., Slaoui, M., and Tudor, C.~A. (2019).
\newblock Generalized $ k $-variations and {H}urst parameter estimation for the
  fractional wave equation via {M}alliavin calculus.
\newblock {\em arXiv preprint arXiv:1903.02369}.

\bibitem[Torres et~al., 2014]{torresEtAl2014}
Torres, S., Tudor, C., Viens, F., et~al. (2014).
\newblock Quadratic variations for the fractional-colored stochastic heat
  equation.
\newblock {\em Electron. J. Probab.}, 19.

\bibitem[Tsybakov, 2010]{Tsybakov10}
Tsybakov, A.~B. (2010).
\newblock {\em Introduction to Nonparametric Estimation}.
\newblock Springer, New York.

\bibitem[Tuckwell, 2013]{tuckwell2013}
Tuckwell, H.~C. (2013).
\newblock Stochastic partial differential equations in neurobiology: Linear and
  nonlinear models for spiking neurons.
\newblock In {\em Stochastic Biomathematical Models}, pages 149--173. Springer.

\bibitem[Whittle, 1953]{Whittle53}
Whittle, P. (1953).
\newblock The analysis of multiple stationary time series.
\newblock {\em J. R. Stat. Soc. Ser. B}, 15:125 – 139.

\end{thebibliography}
\bibliographystyle{apalike}
\end{document}